\documentclass[11pt,letterpaper]{amsart}
\usepackage{constants}
\usepackage{verbatim}
\usepackage{hyperref}
\usepackage{amsmath}
\usepackage{enumerate, amsthm}
\usepackage{mathrsfs}
\usepackage{rotating}
\usepackage{xcolor}
\usepackage{mathtools}
\usepackage{tikz}
\usetikzlibrary{patterns}
\usetikzlibrary{arrows}
\usepackage{centernot}
\usepackage{url}
\usepackage[utf8]{inputenc}
\usepackage[english]{babel}
\usepackage[margin=1.1 in]{geometry}
\usepackage[margin={1.5cm, 1.5cm},font=small, labelsep=endash]{caption}
\usepackage{newcommands}
\usepackage{dsfont}
\usepackage{etex}
\usepackage[shortlabels]{enumitem}
\usepackage{vwcol}

\newtheorem{theorem}{Theorem}[section]	
\newtheorem{lemma}[theorem]{Lemma}

\newtheorem{prop}[theorem]{Proposition}
\newtheorem{proposition}[theorem]{Proposition}

\theoremstyle{definition}
\newtheorem{defn}[theorem]{Definition}

\theoremstyle{remark}
\newtheorem{remark}[theorem]{Remark}

\numberwithin{equation}{section}

\newcommand{\vertiii}[1]{{\left\vert\kern-0.25ex\left\vert\kern-0.25ex\left\vert #1 
		\right\vert\kern-0.25ex\right\vert\kern-0.25ex\right\vert}}

\newcommand{\bE}{\mathbf{E}}
\newcommand{\bP}{\mathbf{P}}
\newcommand{\PB}{\mathbf{P}^{\rm B}}
\newcommand{\PBr}{\mathbf{P}^{\rm br}}

\newcommand{\ind}{\mathds{1}}

\newcommand{\cK}{\mathcal{K}}

\newcommand{\cF}{\mathcal{F}}

\newcommand{\cG}{\mathcal{G}}
\newcommand{\cE}{\mathcal{E}}
\newcommand{\cC}{\mathcal{C}}

\newcommand{\scC}{\mathscr{C}}

\DeclareMathOperator{\Var}{Var}

\makeatletter

\newconstantfamily{c}{symbol=c}
\newconstantfamily{F}{symbol=F}
\newconstantfamily{tilF}{symbol= F}
\makeatother

    \hypersetup{linktocpage,
	colorlinks   = true,
	urlcolor     = blue,
	linkcolor    = blue,
	citecolor   = black
}

\usepackage{multicol}
\usepackage{parcolumns}

\usepackage{float}

\usepackage{color}
\definecolor{Red}{rgb}{1,0,0}
\definecolor{Blue}{rgb}{0,0,1}
\definecolor{Yarok}{rgb}{0,0.5,0}

\begin{document}
\title[Critical level-set percolation]{Critical level-set percolation on finite 
graphs and spectral gap}
  \author[S.~Goswami]{Subhajit Goswami}
\address{School of Mathematics, Tata Institute of Fundamental Research, 1, Homi Bhabha Road, Colaba, Mumbai 400005, India.}
\curraddr{}
\email{goswami@math.tifr.res.in}
\thanks{}

\author[D.~Pal]{Dipranjan Pal} 
\address{Statistics and Mathematics Unit (SMU), Indian Statistical Institute, Kolkata, 203 B.T. Road, Kolkata 700108, India.}
\curraddr{}
\email{dipranjan\_r@isical.ac.in}
\thanks{}




\date{}

\dedicatory{}

\begin{abstract}
We study the bond percolation on finite graphs induced by the level-sets of zero-
average Gaussian free field on the associated metric graph above a given height 
(level) parameter $h \in \R$. We characterize the near- and off-critical phases of this model for 
any expanders family $\cG_n = (V_n, E_n)$ with uniformly bounded degrees. In particular, we show that the volume of the largest open cluster at level 
$h_n$ is of the order $|V_n|^{\frac23}$ when $h_n$ lies in the corresponding {\em critical window} which we identify as $|h_n| = O(|V_n|^{-\frac13})$. Outside this window, the volume starts to deviate from $\Theta(|V_n|^{\frac23})$ culminating into a 
linear order in the supercritical phase $h_n = h < 0$ (the giant component) and a 
logarithmic order in the subcritical phase $h_n = h > 0$.  We deduce these from effective 
estimates on tail probabilities for the maximum volume of an open cluster at any level $h$ for a 
generic base graph $\cG$. The estimates depend on $\cG$ only through 
its size and upper and lower bounds on its degrees and spectral gap respectively. To the best of our knowledge, this is the first instance where a mean-field critical 
behavior is derived under such general setup for finite graphs. The generality of these 
estimates preclude any local approximation of 
$\cG$ by regular infinite trees --- a standard approach in the area. 
Instead, our methods rely on exploiting the connection between 
spectral gap of the graph 
 $\cG$ and 
its connection to the level-sets of zero-average Gaussian free field mediated via a set function 
we call the {\em zero-average capacity}.
\end{abstract}

\maketitle

\section{Introduction}\label{sec:intro}
\subsection{Background and motivation}\label{subsec:motivation}
The study of percolation on finite graphs has a long and rich history which has 
spanned several areas in mathematics. See the book \cite{MR1864966} for a classical 
introduction to the topic and the survey article \cite{van2025percolation} along with 
the references therein for more recent developments in the context of {\em random 
graph models}. More general sequences of finite graphs were considered in \cite{MR2073175, MR2155704, MR2165583}; see \cite{MR4101348, MR4801596, MR4940667, blanc2024scaling, blanc2025critical} for 
some recent works in this area. Beyond Bernoulli percolation, a broad class of models on infinite transient graphs with {\em long range correlations} has attracted 
extensive attention over the past two decades. Notable examples include the random interlacements and its vacant set 
(see, e.g., \cite{MR2680403, MR2891880, zbMATH06257634, zbMATH07227743, zbMATH07226362, RI-I, RI-II, RI-III, GRS24.1, GRS24.2}) and the level-set percolation 
of the Gaussian free field (see, for instance, \cite{RodriguezSznitman13, MR3417515, DrePreRod, MR4112719, gosrodsev2021radius, DCGRS20, MR4557402, ganguly2024ant, MR4898098, werner2025switching}) among 
others. It is therefore natural to study analogous models on finite graphs. Two such 
instances are the vacant sets of random walk \cite{MR2680403, MR2884219, MR3019395, MR3094422} and level-sets of the {\em zero-average} Gaussian free field (GFF) 
\cite{abacherli2018local, MR4169171, MR4642820, MR4552959}. In the current article, we 
are interested in the latter.

\vspace{0.08cm}

Introduced in the context of studying percolation of GFF level-sets on discrete tori in \cite{abacherli2018local}, the zero-average Gaussian free field is a ``suitable'' version of the free field that can be defined 
on any finite graph. See the recent work \cite{hartung2025extremes} regarding the 
extremal behavior of this process. 
In \cite{MR4169171} (see also \cite{MR4115734}), A.~Ab\"acherli and J.~\v{C}ern\'y looked into the level-set percolation of zero-average Gaussian free field on a certain 
class of finite {\em locally tree-like} regular graphs which includes the {\em large 
girth regular} expanders and typical realizations of {\em random regular} graphs. 
Subsequently, the supercritical phase of the model was analyzed for these families of 
graphs in \cite{MR4642820} by J.~\v{C}ern\'y and in \cite{MR4552959} by G.~Conchon-Kerjan. In the current work, we initiate the study of 
the critical phase of (a variant of) this model on {\em any} sequence of finite 
graphs with growing sizes 
provided they have uniformly bounded degrees and uniformly positive spectral gaps.

\vspace{0.08cm}

Bernoulli percolation on such families of graphs, i.e., the uniform expanders, were 
studied by N.~Alon, I.~Benjamini and A.~Stacey in \cite{MR2073175} where they were 
primarily interested in the existence and uniqueness of giant clusters. Under the {\em 
additional} assumptions that these graphs are $d$-regular with diverging girth, they 
identified the critical probability for the emergence of a giant cluster as $\frac1{d-1}$ which is same as the critical probability for Bernoulli percolation on (infinite) 
$d$-regular trees. The {\em large girth} assumption is crucial as it allows to 
approximate neighborhoods of the graph of large radii with those of the regular 
tree. This was generalized in the very recent work \cite{MR4597319}. Results similar 
to \cite{MR2073175} were proved for the vacant sets of random walk in \cite{MR2884219} 
(see also \cite{MR3019395}) and the level-sets of the zero-average Gaussian free field 
in \cite{MR4642820, MR4552959}.

\vspace{0.08cm}

The {\em mean-field} (near-)critical behavior for the Bernoulli percolation on 
sequences of {\em random regular} graphs was established by A.~Nachmias and Y.~Peres 
in \cite{MR2583058} and by B.~Pittel in \cite{MR2435852}. J.~\v{C}ern\'y and 
A.~Teixeira obtained similar results in \cite{MR3094422} for the vacant set of random 
walk on random regular graphs.  In \cite{MR3094422}, the authors crucially use an observation 
made by C.~Cooper and A.~Frieze from \cite{MR3019395} whereby the vacant set of random 
walk on random regular graphs can be seen as an average (distributionally) over random 
graphs with prescribed (random) degree sequences whose critical regime is very 
well-understood (see, e.g.~\cite{MR2943428, MR2490288}). See also \cite{MR4084187} for 
more refined results on the geometry of critical clusters of the vacant set. The setting of 
random regular graphs thus plays an extremely important role in all these results.

One of the main novel features of the current article is that we derive 
mean-field bounds 
for a version of the level-set percolation of zero-average Gaussian free field on any 
finite graph $\cG = (V, E)$ that only involve the size $|V|$ and upper and lower bounds on the degrees and the spectral gap of $\cG$. 
To the best of our knowledge, this is the {\em first} instance where a 
mean-field characterization of the critical regime is established in such generality and 
precision for a percolation model (conjecturally) in the Erd\H{o}s-R\'enyi universality class.

The version of the level-set percolation alluded to in the previous paragraph refers 
to the bond percolation model on a (finite) graph $\cG = (V, E)$ defined as follows. 
Given a realization of the zero-average Gaussian free field $(\varphi(x): x \in V)$ 
(see \S\ref{subsec:results} below for definition) on the vertices of $\cG$ and $h \in \R$, we {\em open} each edge $\{x, y\} \in E$ independently (at level $h$) with 
conditional probability 
\begin{equation}\label{eq:bridge_prob}
1 - \exp\big(- 2 (\varphi(x) - h)_+(\varphi(y) - h)_+\big) \mbox{ where } a_+ \stackrel{{\rm def.}}{=} \max(a, 0) \mbox{ for } a \in \R
\end{equation}
(cf.~display~(1.6) in \cite{MR4557402} and also \cite{BorodinBrown}, Chap.IV, \S26, p.67). This ``rule'' corresponds to including the edge $(x, y)$ in the resulting 
random graph iff the value of the field $\varphi$, extended suitably to the metric 
graph $\widetilde \cG$ associated to $\cG$ where each edge is replaced by a unit 
interval connecting its two endpoints (see Section~\ref{sec:prelim} for precise 
definitions), stays above level $h$. 

\vspace{0.08cm}

The study of level-set percolation on metric graphs originated in the work \cite{MR3502602} by T.~Lupu for {\em transient 
graphs}. In recent years, there has been spectacular progress in this area for a class 
of transient graphs with {\em polynomial growth} which includes $\Z^d$, $d \ge 3$; 
see, e.g.~\cite{DrePreRod, MR4112719, MR4421175, MR4557402, drewitz2023arm, ganguly2024ant, cai2024onelow, drewitz2024cluster, cai2024quasi, cai2024incipient, MR4898098, werner2025switching, carpenter2025loops, cai2025heterochromatic, cai2025gapclusterdimensionsloop, cai2025separation, drewitz2025critical}. See also the 
very recent articles \cite{rodriguez2025anomalous, rodriguez2025cable} where the 
authors consider level-set percolation on some finite subgraphs of $\Z^2$ and $\Z^3$ 
for the (metric graph) free field associated to the {\em transient} random walk 
killed at a rate depending on the size and geometry of the subgraphs. However, the zero-average Gaussian free field --- which can be defined in a rather canonical manner 
on finite graphs --- exhibit new aspects for the level-set percolation, not least 
due to the global zero-average constraint (see \eqref{eq:Gff_sum}) imposed by the 
{\em recurrence} of the underlying random walk. On the other hand, the existing works on the 
level-set percolation of zero-average Gaussian free field (or, for that matter, the Bernoulli 
percolation, the vacant set of random walk etc.) on finite graphs require different degrees of 
control on the local geometry of the graph (see, e.g., \cite{abacherli2018local, MR4169171, MR4642820, MR2583058, MR2884219, MR3094422}). Unfortunately the setup under which 
we work in this paper does not allow any such control. Instead, we develop an approach that is 
based on ``global'' properties of the graph accessible only through its maximum degree and 
the spectral gap. We refer to \S\ref{subsec:overview} for an overview of our techniques.

\subsection{Main results}\label{subsec:results}
For any finite (simple, connected) graph $\cG = (V, E)$, the 
corresponding (discrete) {\em zero-average Gaussian free field} is a centered Gaussian 
process $\varphi_{\cG} = (\varphi(x) : x \in V)$ with covariances given by the {\em 
zero-average Green function} $g_{\cG}(\cdot, \cdot)$ on $\cG$ (see, e.g., displays~(2.3)--(2.5) in 
\cite[Section~2]{MR4642820}). Now for any $h \in \R$, let $E^{\ge h} \subset E$ denote the subset of open edges at level $h$ obtained as 
per the rule described around \eqref{eq:bridge_prob} above and $\cG^{\ge h} = (V, E^{\ge h})$ denote the (random) subgraph induced by these 
edges. Also let $\scC^{\ge h}_x = \scC^{\ge h}_x(\cG)$ denote the {\em cluster} (i.e., the connected component) of $x \in V$ in the graph $\cG^{\ge h}$ and 
$\scC^{\ge h}_{\max} = \scC^{\ge h}_{\max}(\cG)$ denote ``a'' cluster with {\em 
maximum} volume (chosen arbitrarily). Our first main result in this paper identifies 
the ``critical window''
around $h = 0$ in terms of the order of $|\scC^{\ge h}_{\max}|$. In the sequel, we let 
$d^\ast_{\cG}$ and $\lambda^\ast_{\cG}$ denote the maximum degree and the spectral gap 
of $\cG$ (see~\eqref{def:spectral_gap} in Section~\ref{sec:prelim} for definition) respectively 
and $\P_{\cG}$ denote probability measure underlying everything, i.e., the field $\varphi_{\cG}$ along with the variables $\{1_{\{e \in E^{\ge h}\}} : e \in E, h \in \R\}$.
\begin{theorem}[Mean-field behavior inside the critical window]\label{thm:main1}
Let $d > 1$, $\lambda > 0$ and $\cG = (V, E)$ be any finite 
graph with $d^\ast_\cG \le d$ and $\lambda^\ast_\cG \ge \lambda$. Then there exists 
$C = C(d, \lambda) \in (0, \infty)$ such that for any $h = A|V|^{-1/3}$ with $A \in \R$ and $\delta \in (0, 1)$ satisfying $|V| \ge \frac{(1 
+ |A|)^8}{\delta^3}$, we have
\begin{equation}\label{eq:main_critbnd}
\P_{\cG}\big[ \delta|V|^{2/3} \le |\scC^{\ge h}_{\max}| \le \tfrac 1\delta |V|^{2/3}\big] 
\ge 1 - C (1 + |A|) \, \delta^{1/5}.
\end{equation}
\end{theorem}
Our second main result concerns the order of $|\scC^{\ge h}_{\max}|$ when $h$ is 
farther away from 0.
\begin{theorem}[Mean-field behavior outside the critical window]\label{thm:main2}
Under the same assumptions as in Theorem~\ref{thm:main1}, there exist $C = C(d, \lambda)$ and $c = c(d, \lambda)$ in $(0, \infty)$ such that for any $h = A |V|^{-1/3}$ with $A \in (1, c |V|^{1/3})$, we have
\begin{equation}\label{eq:main_subcritbnd}
\P_{\cG}\big[ |\scC^{\ge h}_{\max}| \le  C\,\tfrac{\log (|V|h^3)}{h^2}
 \big] \ge 1 - 
 \tfrac C{A}.
\end{equation}
On the other hand, if $h = -A |V|^{-\frac13}$ with $A \in (1, |V|^{1/3})$, we have
\begin{equation}\label{eq:main_supcritbnd}
\P_{\cG}\big[ |\scC^{\ge h}_{\max}| \ge 
c \, |h| \, |V| \, \big] \ge 1 - \tfrac C {{A}^{1/5}}.
\end{equation}
\end{theorem}
The typical order of $|\scC^{\ge h}_{\max}|$ implied by these two theorems match those 
for the Erd\H{o}s-R\'enyi random graphs and Bernoulli percolation on random regular 
graphs; cf., for instance, \cite{MR2583058, MR2435852} and see also the references 
therein.

As an immediate corollary of these results, we can identify the {\em critical 
parameter} for the emergence of a giant cluster in this percolation model on 
any uniform expanders sequence with uniformly bounded degrees as 0 and also 
characterize the nature of its phase transition as follows.
\begin{theorem}[Nature of phase transition on expanders]\label{thm:phase_transition}
Let $d > 1$ and $\lambda > 0$. Also let $\cG_n = (V_n, E_n)$ be a sequence of finite 
graphs such that $d^\ast_{\cG_n} \le d$ and $\lambda^\ast_{\cG_n} \ge \lambda$ for all $n \ge 1$, and $|V_n| \to \infty$ as $n \to \infty$. Then for any $h < 0$, there exists $c(h) = c(h, d, \lambda) \in (0, \infty)$ such that
\begin{equation}\label{eq:pt_giant}
 \lim_{n \to \infty} \P_{\cG_n}[ |\scC^{\ge h}_{\max}| \ge c(h) |V_n| \, ] = 1 \quad\quad\quad\quad\quad\quad\quad\quad\quad\:\: \mathrm{(supercritical~phase)}.
\end{equation}
When $h > 0$, there exists $C(h) = C(h, d, \lambda) \in (0, \infty)$ such that 
\begin{equation}\label{eq:pt_subcrit}
 \lim_{n \to \infty} \P_{\cG_n}[ |\scC^{\ge h}_{\max}| \le C(h) \log |V_n| \, ] = 1 \quad\quad\quad\quad\quad\quad\quad \mathrm{(subcritical~phase)}.
\end{equation}
Finally for $h = 0$, we have
\begin{equation}\label{eq:pt_crit}
\lim_{n \to \infty} \P_{\cG_n}\big[ \delta|V|^{2/3} \le |\scC^{\ge h}_{\max}| \le \tfrac 1\delta |V|^{2/3}\big]  = 1 \quad\quad\quad\:\: \mathrm{(critical~phase)}.
\end{equation}
\end{theorem}
Theorems~\ref{thm:main1} and \ref{thm:main2} follow from more general bounds on the 
{\em tail probabilities} for $|\scC^{\ge h}_{\max}|$, namely 
Theorems~\ref{prop:ubd_cri_vol} and \ref{prop:lbd_cri_vol} which are stated and proved 
in Section~\ref{sec:critical_regime}.
\subsection{Proof outline}\label{subsec:overview}
We now present a brief account of the key ideas that go into the proof of our main results. 

\vspace{0.1cm}

\begin{itemize}[leftmargin=*]
\item {\bf Dirichlet martingale for exploration of clusters.} It is often useful in the study of percolation clusters to ``attach'' a 
suitably defined martingale to a cluster exploration process and analyze its properties. See, e.g., \cite{MR2583058, MR2653185} for Bernoulli percolation on random graphs and \cite{MR4112719} for the percolation of GFF level-sets on the metric graphs associated with $\Z^d$. In our situation, we identify the right martingale to be given by the functional $$K \mapsto \widetilde \cE(\varphi, f_K)$$
where $K$ ranges over ``nice'' compact subsets of $\widetilde \cG$ (the metric graph 
associated to $\cG$), $\widetilde \cE(\cdot, \cdot)$ is the Dirichlet form on $\widetilde \cG$ (see~\eqref{def:dirichlet}), $\varphi$ is the corresponding zero-average GFF and $f_K$ 
minimizes the form $\widetilde \cE(f_K, f_K)$ subject to the conditions $$f_{|K} \equiv 1 \mbox{ and } \sum_{x \in V} d_x f(x) = 0$$
where $d_x$ is the degree of the vertex $x$ in $\cG$. The constraint $\sum_{x \in V} d_x 
f_K(x) = 0$ manifests the zero-average condition. Indeed, as we explain in 
\S\ref{subsec:zeroavg}, $\varphi$ is the Gaussian free field on $\widetilde \cG$ with free 
boundary condition {\em normalized} so that $\sum_{x \in V} d_x\varphi(x) = 0$. From an exploration of clusters at a level $h$, we obtain a continuous family 
$(K_t)_{t \ge 0}$ of random (nice) compact subsets of $\widetilde \cG$ such that the process $$(\tilde M_t)_{t \ge 0} \stackrel{{\rm def.}}{=} (\cE(\varphi, f_{K_t}))_{t \ge 0} \mbox{ is a {\em continuous} martingale w.r.t. the filtration of $(K_t)_{t \ge 0}$}.$$ 

\vspace{0.15 cm}

\item {\bf Linking $\tilde M_t$ to cluster size via zero-average capacity.} 
The process $(\tilde M_t)_{t \ge 0}$ carries important information on the cluster volumes. When $K_t$ lies inside a cluster (at level $h$), 
we have $$\tilde M_t \ge \widetilde \cE(f_{K_t}, f_{K_t})\,h.$$ On the other hand, when $K_t$ is ``close'' to a cluster, we have $$\tilde M_t = \widetilde \cE(f_{K_t}, f_{K_t})\,h + o_{\P}(1)$$
where the $o_{\P}(1)$ term converges to 0 in $\P$ {\em uniformly} over all compact 
sets that are not too large. We call the quantity $\widetilde \cE(f_{K}, f_{K})$ as the {\em zero-average capacity} of the set $K$ (see below Remark~\ref{rmk:fKexp} in Section~\ref{sec:zeravgcap} for the reason behind such naming). 
The bound on the spectral gap gives yields the following linear isoperimetric condition for the zero-average capacity:
$$\mbox{$\widetilde\cE(f_{K}, f_{K})$ grows linearly in the volume of $K$ when $K$ is not large.}$$

\vspace{0.15 cm}

\item {\bf Long excursion events for $\tilde M_t$ and computation of probabilities.} 
Carefully using the observations made in the previous item, we can translate 
events involving the existence of large percolation clusters into events involving the 
existence of long excursions for $\tilde M_t$ away from some suitably chosen value 
after being reparametrized by the zero-average capacity of $K_t$ --- which also turns 
out to be the {\em quadratic variation} of $M_t$. At this point, we can compute the 
corresponding probabilities as probabilities of similar events for Brownian motion using 
Dubin-Schwarz (cf.~\cite{MR4112719}).
\end{itemize}

\vspace{0.3cm}

We now describe the organization of this article. In Section~\ref{sec:prelim}, we 
gather some necessary definitions and notions including the construction of the metric graph $\widetilde \cG$ and the corresponding zero-average Gaussian free field as well 
as some elements of potential theory on $\cG$ and $\widetilde \cG$. We formally 
introduce the zero-average capacity in Section~\ref{sec:zeravgcap} and prove some of its important properties. In 
Section~\ref{sec:dirchlet}, we establish the martingale property of the 
functional $K \mapsto \widetilde \cE(\varphi, f_K)$ (see the discussion above) w.r.t. 
set inclusion and record some of its useful features. Section~\ref{sec:exploration} is 
concerned with the exploration process $(K_t)_{t \ge 0}$ (see above) which we formally 
describe in \S\ref{subsec:exp_scheme} and the exploration martingale $(\tilde M_t)_{t \ge 0}$ which we define in \S\ref{subsec:martingale} using the construction in 
Section~\ref{sec:exploration}. Finally in Section~\ref{sec:critical_regime}, we state 
Theorems~\ref{prop:ubd_cri_vol} and \ref{prop:lbd_cri_vol} regarding the tail 
probabilities of cluster size from which we derive our main results, i.e., 
Theorems~\ref{thm:main1} and \ref{thm:main2} in a fairly straightforward manner. The 
remainder the of the section is devoted to the proof of 
Theorems~\ref{prop:ubd_cri_vol} and \ref{prop:lbd_cri_vol}.

\vspace{0.2cm}

\noindent{\em Notations and conventions.} In the upcoming sections, we will 
assume the graph $\cG = (V, E)$ to be {\em fixed} and drop it from the notations 
$\P_{\cG}$, $d^\ast_{\cG}$ etc. unless they need to be distinguished in some context. 
In view of the formulation of our main results, we may assume $|V| > 1$ without 
any loss of generality. We also assume that $d^\ast_\cG \le d$ and $\lambda^\ast_\cG \ge 
\lambda$ for some $d > 1$ and $\lambda > 0$. 
We denote by $c, c', C, C', \ldots$ etc. 
finite, positive constants which are allowed to vary from place to 
place. All constants may implicitly depend on $d$ and $\lambda$. Their dependence on 
other parameters will be made explicit. Numbered constants $c_1,c_2,\dots$ and $C_1,C_2,\dots$ stay fixed after their first appearance within the text.

\section{Preliminaries}\label{sec:prelim}
In this section we collect several definitions, preliminary facts and results that 
will be used throughout the paper. 
\subsection{The metric graph $\widetilde \cG$.} Let us start with the {\em metric graph} (or the {\em cable system}) $\widetilde{\cG}$ associated to $\cG$ which is topologically the 1-complex 
obtained by gluing a copy of the open unit interval $I_{\{x, y\}}$ at its endpoints $x, y \in V$ for each $\{x, y\} \in E$. We denote the (topological) closure and boundary of 
any $S \subset \widetilde \cG$ as $\overline S$ and $\partial S$ respectively. For any 
$x, y \in \overline{I_e}$ where $e \in E$, we can define the {\em intervals (of $\widetilde \cG$)} with endpoints at $x$ and $y$ that are open, closed, semi-closed at 
$x$ and semi-closed at $y$ in the obvious manner and denote them as (the ordering of $x, y$ is irrelevant)
\begin{equation}\label{def:intG}
\mbox{$(x,y)_{\cG}$, $[x,y]_{\cG}$, $[x,y)_{\cG}$ and $(x,y]_{\cG}$ respectively.}
\end{equation}

Equipping each $I_e, e \in E$ with the Lebesgue measure (of total mass 1), we get the 
push-forward measure $\rho$ on $\widetilde \cG$ which assigns a unit length to 
each interval $I_e$. The space $\widetilde \cG$ is metrizable. In fact, the graph distance $d(\cdot, \cdot)$ on $\cG$ extends naturally to a distance on 
$\widetilde{\cG}$ which we also denote as $d(\cdot, \cdot)$. More precisely, for any $x, y \in \widetilde{\cG}$, we let
\begin{equation}\label{def:distance}
\begin{split}
 d(x, y)  = \min_{x', y'} \,\, \big(\rho([x, x']_{\cG}) + d(x', y') + \rho([y', y]_{\cG})\big),
\end{split}
\end{equation}
where the minimum is over all pairs of endpoints of (closed) intervals containing $x$ 
and $y$ respectively. In plain words, $d(x, y)$ is the minimum length of a continuous 
path between $x, y$. For any $x \in \widetilde \cG$ and $r \ge 0$, we denote by $B(x, r)$ the ball of radius $r$ centered at $x$.

The metric space $\widetilde \cG = (\widetilde{\cG}, d)$ is compact and locally 
path-connected. Further, any connected subset of $\widetilde \cG$
\begin{equation}\label{eq:conn_set_G}
\mbox{is path-connected with a finite boundary set and is a union of finitely many intervals.}
\end{equation}
For any $x \in \widetilde \cG$ and $S \subset \widetilde \cG$, we let $\cC_x(S)$ 
denote the {\em (connected) component} of $x$ inside $S$ (as subspaces of $\widetilde 
\cG$). Notice that $\cC_x(S) = \emptyset$ if $x \notin S$. In order to avoid 
confusion, we use the term {\em cluster} only for the connected components of the 
random subgraph $\cG^{\ge h}; h \in \R$ which we denote by $\scC_x^{\ge h}$ (see \S\ref{subsec:results}) and reserve the term {\em component} for subsets of 
$\widetilde \cG$. More generally, for any $S' \subset \widetilde \cG$, we let $\cC_{S'}(S) \stackrel{{\rm def.}}{=} \cup_{x \in S'} \cC_x(S)$.

The following two special collections of subsets of $\widetilde \cG$ will appear 
frequently in this article. We let 
\begin{equation}\label{def:cK}
\mbox{$\cK$ denote the collection of all non-empty compact subsets of $\widetilde\cG$.}
\end{equation}
We further let
\begin{equation}\label{def:cKinfty}
\text{\begin{minipage}{0.6\textwidth}$\cK_{<\infty}$ denote the sub-collection of $\cK$ containing only sets $K$ having finitely many components 
and satisfying $V \setminus K \ne \emptyset$.
\end{minipage}}
\end{equation}

\subsection{Excerpts from the potential theory on $\widetilde \cG$.} Next we gather 
some useful facts from the potential theory on $\cG$ and $\widetilde \cG$. For $f,g \in \R^V$, the {\em Dirichlet form} on $\cG$ is given by
\begin{equation}\label{def:denergy}
\cE(f,g) \stackrel{\rm def.}{=} \sum_{\{x, y\} \in E} (f(x) - f(y))(g(x) - g(y)). 
\end{equation}
The {\em spectral gap} $\lambda^\ast$ of $\cG$ can be defined in terms of the 
Dirichlet form as follows:
\begin{equation}\label{def:spectral_gap}
\lambda^\ast = \inf\Big\{\cE(f,f): \sum_{x\in V} d_x f^2(x) =1, \sum_{x \in V} d_x f(x) = 0 \Big\}.
\end{equation}
For $S \subset V$, let $\partial_e S \stackrel{\rm def.}{=} \{ e \in E: e \mbox{ has 
an endpoint in both $S$ and $V \setminus S$}\}$ denote the {\em edge boundary} of $S$ 
in the graph $\cG$. The following is a consequence of the Cheeger’s inequality 
(see,~e.g.,~\cite[Theorem~6.15]{MR3616205}) and our standing assumption that  $\lambda^\ast \ge \lambda$,
\begin{equation}\label{eq:isoperimetry}
    |\partial_{e}S| \ge c |S| \text{ for all } S \subset V \text{ with } |S| \le |V|/2.
\end{equation}
In particular, this implies that the graph $\cG$ is connected. Letting $\tilde{H}^1(\cG)$ denote the subspace of $\R^V$ consisting of all $f \in \R^V$ satisfying 
$\sum_{x \in V} d_xf(x) = 0$, we see that $\tilde{H}^1(\cG)$ equipped with the inner 
product $\cE(\cdot, \cdot)$ is a Hilbert space.

The corresponding Dirichlet form on $\widetilde \cG$ is defined as
\begin{equation}\label{def:dirichlet}
    \widetilde \cE(f,g) 
    = 
    \int_{\widetilde \cG} f'g' \,d\rho \text{ for all } f,g \in H^1(\widetilde \cG)
\end{equation}
where $H^1(\widetilde \cG)$ denotes the space of all real-valued, {\em absolutely 
continuous} functions $f$ on $\widetilde \cG$ (i.e., $f_{|I_e}$ is absolutely continuous on $I_e$ w.r.t. the measure $\rho_{|I_e}$ for each $e \in E$) such that $f' \in L^2(\widetilde \cG, \rho)$. Like $\tilde{H}^1(\cG)$, we let $\tilde{H}^1(\widetilde \cG)$ denote the subspace of $H^1(\widetilde \cG)$ consisting of all $f \in H^1(\widetilde \cG)$ satisfying $\sum_{x \in V} d_xf(x) = 0$ so that it is a 
Hilbert space when equipped with the form $\widetilde \cE(\cdot, \cdot)$. Owing 
to the simple structure of the connected subsets of $\widetilde \cG$ as elucidated in 
\eqref{eq:conn_set_G}, we can similarly define a Dirichlet form $\widetilde \cE_K(\cdot, \cdot)$ and the corresponding Hilbert spaces $H^1(K)$ and $\tilde H^1(K)$ 
associated to any $K \subset \widetilde \cG$ with finitely many components (the 
constraint for $\tilde H^1(K)$ being $\sum_{x \in V \cap K} d_x f(x) = 0$). To simplify notations, we let 
\begin{equation}\label{def:tildeE-K}
	\widetilde \cE_{K^c} (f,g) \stackrel{\rm def.}{=} 
	\widetilde \cE_{\overline{K^c}} (f_{|\overline{K^c}},g_{|\overline{K^c}})
	\text{ for all } f, g:\widetilde \cG \to \R \mbox{ s.t. } f_{|\overline{K^c}}, g_{|\overline{K^c}} \in H^1(\overline{K^c}).
\end{equation}

\vspace{0.1cm}

Of special interests to us are those functions which are ``linear'' on intervals of 
$\widetilde \cG$. More precisely, for any $K \subset \widetilde \cG$, interval $[x, y]_{\cG} \subset K$ and function $f: K \to \R$, we say that $f$ is {\em linear} on $[x, y]_{\cG}$ if 
\begin{equation}\label{eq:linear_deriv}
f(z) = f(x) \tfrac{\rho([z, y]_{\cG})}{\rho([x, y]_{\cG})} + f(y) \tfrac{\rho([x, z]_{\cG})}{\rho([x, y]_{\cG})}    
\end{equation}
for all $z \in [x, y]_{\cG}$. Clearly if $f$ is linear on $[x, y]_{\cG}$, then $f$ is 
absolutely continuous with a constant derivative on the same interval. 
More generally, given any collection $\mathcal I$ of (closed) intervals of $K \, (\subset \widetilde \cG)$, we say that
\begin{equation}\label{def:linear_cI}
\mbox{$f$ is {\em linear on} $\mathcal I$ provided $f$ is linear on each $I \subset \mathcal I$.}
\end{equation}
We call $f$ as {\em piecewise linear} (on its domain $K$) if $f$ is linear on 
$\mathcal I$ for some {\em finite} collection $\mathcal I$ satisfying $\cup_{I \in 
\mathcal I} I = K$. Piecewise linear functions on $\widetilde\cG$ are elements of $H^1(\widetilde \cG)$.

\vspace{0.1cm}

The vocabulary of piecewise linear functions allow us to obtain the 
appropriate Dirichlet forms on {\em discrete subgraphs} of $\widetilde \cG$ as {\em 
traces} of $\widetilde \cE(\cdot, \cdot)$ (see, e.g.~\cite[Chapter~1.3]{MR3156983}). 
To this end, we recall the framework of {\em enhancements} from \cite[Section~2.2]{MR4421175}. Let $K \subset \widetilde \cG$ be compact with finitely many components 
so that the boundary $\partial K$ is a finite set by \eqref{eq:conn_set_G}. The {\em enhancement set} induced by $K$ is given by 
\begin{equation}\label{def:VK}
        V^{K} \stackrel{\rm def.}{=} V \cup \partial K.
\end{equation}
$\widetilde \cG$ can be expressed as a union of closed intervals with disjoint 
interior whose endpoints are points in $V^K$. In the sequel, 
\begin{equation}\label{def:int_GK}
\text{\begin{minipage}{0.7\textwidth}
we refer to this collection of intervals as $\mathcal I_K$ and denote by $\mathcal I_{-K}$ the sub-collection comprising only intervals that are {\em not} subsets of $K$
\end{minipage}}
\end{equation}
(recall the structure of connected subsets of $\widetilde \cG$ from \eqref{eq:conn_set_G}). Now let $E^K$ denote the set of all (unordered) pairs $\{x, y\}$ of points in $V^K$ that are the endpoints of some interval in $\mathcal I_K$. 
Then for any $f, g \in H^1(\widetilde \cG)$ that are linear on $\mathcal I_K$ (recall \eqref{def:linear_cI}), we have in view of \eqref{def:dirichlet} and \eqref{eq:linear_deriv}:
\begin{equation}\label{eq:tildeEdescent}
    \widetilde \cE(f, g) = \sum_{\{x,y\} \in E^K} (f(x) - f(y)) (g(x) - g(y)) \, \tfrac{1}{\rho([x, y]_{\cG})}
\end{equation}
More so, it is not difficult to see that for any such $f$, 
\begin{equation}\label{state:min_energy}
\widetilde \cE(f, f) = \min_{g} \widetilde \cE(g, g),
\end{equation}
where the minimum is over all $g \in H^1(\widetilde \cG)$ satisfying $g_{|V^K} = f_{|V^K}$. Thus we get the following {\em trace form} on the {\em enhancement graph} $\cG^K \stackrel{{\rm def.}}{=} (V^K, E^K)$ (induced by $K$):
\begin{equation}\label{def:dirichletK} 
\cE^K(f, g) \stackrel{{\rm def.}}{=} \sum_{\{x,y\} \in E^K} (f(x) - f(y)) (g(x) - g(y)) \, C^K_{\{x, y\}}
\end{equation}
for $f, g \in \R^{V^K}$ where
\begin{equation}\label{def:C^K}
 C^K_{\{x, y\}} \stackrel{\rm def.}{=} \tfrac{1}{\rho([x, y]_{\cG})} \mbox{ for $\{x, y\} \in E^K$ and $0$ otherwise.}
\end{equation}
Henceforth, we will assume the graph $\cG^K = (V^K, E^K)$ to be equipped with the 
conductances  $C^K_e, e \in E^K$. The corresponding {\em degree} of $x \in V^K$ is 
given by
  \begin{equation}\label{def:total_conduc}
      C^K_x \stackrel{\rm def.}{=} \sum_{\{x, y\} \in E^K} C^K_{\{x,y\}}.
      \end{equation}
The display~\eqref{eq:tildeEdescent} also gives us a way to extend the domain of the form $\widetilde \cE(\cdot, \cdot)$ as follows. For {\em any} function $f$ and a piecewise linear function $g$ on $\widetilde \cG$, let us define 
\begin{equation}\label{eq:tildeE_to_EK}
\widetilde \cE (f, g) = \cE^K(f^K, g)
\end{equation}
where $K$ is the set of endpoints of the intervals in $\mathcal I$ on which $g$ is 
linear (see below \eqref{def:linear_cI}) and $f^K$ is the piecewise linear extension 
of $f_{|V^K}$ to $\widetilde \cG$ (cf.~\eqref{eq:linear_deriv}).

\subsection{The canonical diffusions on $\widetilde \cG$ and $\cG$.} 
Now one can define, for each $x \in \widetilde \cG$, a 
$\rho$-symmetric diffusion process $(\widetilde X_t)_{t \ge 0}$ starting from $x$ with 
state space $\widetilde \cG$ via its associated Dirichlet form $\widetilde\cE(\cdot, 
\cdot)$ which we can view as the $\widetilde \cG$-valued Brownian motion. See \cite[\S2.1]{MR4421175} and \cite[Section~2]{MR3502602} for a detailed account of this 
construction. 
We denote its law (and the corresponding expectation) by $\bP_x$ (and $\bE_x$ respectively). For $K \subset \widetilde \cG$, we let
\begin{equation}\label{def:hitting_time}
H_K \stackrel{\rm def.}{=} \inf\{t\ge 0: \widetilde X_t \in K\} \mbox{ and } T_K \stackrel{\rm def.}{=} H_{K^c}
\end{equation}
denote the hitting and exit time of the set $K$ respectively. One can show (see \cite[\S2.1]{MR4421175} and \cite[Section~2]{MR3502602}) that the process 
$\widetilde X = (\widetilde X_t)_{t \ge 0}$ admits of a continuous family of local times $(\ell_x(t))_{x \in \widetilde \cG,\, t \ge 0}$ under $\bP_x$. Then for any open $U \subset \widetilde \cG$, we can define the corresponding {\em Green function} 
as
\begin{equation}\label{def:green-K}
    g_{U}(x,y) 
    \, (= g_{U}(y, x)) = \bE_x[\ell_y(T_{U})] \text{ for all } x,y \in \widetilde \cG.
\end{equation}
Let us make two useful observations on the Green functions which are standard (see, e.g., \cite[display~(2.1)]{MR3502602}). 
For any $K \in \cK$ (recall \eqref{def:cK}) with finitely many components and 
$x \in \widetilde \cG$,
\begin{equation}\label{def:harmgreen}
    \mbox{$g_{K^c}(x,y)$ is linear on $\mathcal I_{-(K \cup \{x\})}$ (recall \eqref{def:int_GK}) as a function in $y \in \widetilde \cG$.}
\end{equation}
Further, with the supremum ranging over all $K \in \cK$ with finitely many components,
\begin{equation}\label{eq:sup_greensfunc}
            \mathrm{G} \stackrel{\rm def.}{=} \sup_{K} 
            g_{K^c}(x, y) < \infty
    \end{equation}

Now by taking {\em traces} of the process $\widetilde X$ on $V^K$ (see, e.g., \S2.1 and \S2.2 in \cite{MR4421175} for a formal treatment) we obtain, for each compact $K$ with finitely many connected components, the canonical jump processes $(\widetilde X^K_t)_{t \ge 0 }$ on $\cG^K$ with jump rates $C^K_{\{x,y\}}, \{x,y\} \in E^K$ (see \eqref{def:C^K}) 
under the probability measures $\bP_x, x \in V^K$. For any $K' \subset \widetilde \cG$, with a slight abuse of notation, we use $H_{K'}$ and $T_{K'}$ as in 
\eqref{def:hitting_time} to denote the hitting and exit time of the set $K' \cap V^K$ 
for $\widetilde X^K$. As explained below display~(2.4) in \cite[\S2.1]{MR4421175}, we 
have for any $K' \in \cK$ with finitely many components such that $\partial K' \subset 
V^K$, \begin{equation}\label{eq:invar_greensfunc}
\begin{split}
&g_{(K')^c}(x, y) = \bE_x\Big [\int_{0}^{\infty} 1_{\{\widetilde X_t^K = y,\, t < H_{K'}\}} dt \Big] \mbox{ and } \\ 
&\bP_x[\widetilde X^K_{H_{K'}} = y] =  \bP_x[\widetilde X_{H_{K'}} = y]\: \mbox{for } x, y \in V^K.    
\end{split}
\end{equation}
Denoting the discrete skeleton of the continuous-time chain $\widetilde X^K$ by $X^K =(X^K_n)_{n \in \N}$, we have the following identities in view of 
\eqref{eq:invar_greensfunc} under the same setup:
\begin{equation}\label{def:green_enhance}
\begin{split}
&g_{(K')^c}(x,y) =  \frac{g^K_{(K')^c}(x,y)}{C_y^K} \mbox{ where }g^K_{(K')^c}(x,y)\stackrel{\rm def.}{=} \sum_{n \ge 0} \bP_x[X^K_n = y, n < H_{K'}] \mbox{ and}\\
& \bP_x[\widetilde X^K_{H_{K'}} = y] =  \bP_x[X_{H_{K'}}^K = y]\: \mbox{for } x, y \in V^K.
\end{split}
\end{equation}
Here, similarly as before, we use $H_{K'}$ to denote the hitting time of the set $K' 
\cap V^K$ by the discrete-time walk $X^K$ to avoid clutter. The (discrete) {\em Laplacian} operator on $\cG^K$ corresponding to the walk $X^K$ is given by
\begin{equation}\label{def:laplacianK}
    \Delta^Kf(x) \stackrel{\rm def.}{=} \sum_{\{x,y\} \in E^K} (f(y) - f(x))\, C^K_{\{x,y\}} \: \mbox{for all } f \in \R^{V^K} \mbox{ and } x \in V^K.
\end{equation}
On the other hand, the Laplacian {\em killed} outside $K$ is defined as 
\begin{equation}\label{def:laplaciankilled}
    \Delta^K{K^c}f(x) = \sum_{\substack{\{x,y\} \in E^K,\\ y \in V^K \setminus K}} f(y) \, C^K_{\{x,y\}} - f(x)\Big(\sum_{\{x, y\} \in E^K} C^K_{\{x,y\}}\Big)
\end{equation}
for all $f \in \R^{V^K \setminus K}$ and $x \in V^K \setminus K$. Often we need to 
work with a {\em common refinement} of enhancement sets induced by two sets $K$ and 
$K'$ in $\cK$ with finitely many connected components. To this end, we let
\begin{equation}\label{def:KK'}
\text{\begin{minipage}{0.9\textwidth}
$V^{K, K'}$, $E^{K, K'}$, $C_{\{\cdot, \cdot\}}^{K, K'}$, $C_{\{\cdot\}}^{K, K'}$, $X^{K, K'}$ and $g_{\cdot}^{K, K'}(\cdot, \cdot)$ denote the objects defined above corresponding to the compact set $K \cup \partial K'$ (which has finitely many 
components).\end{minipage}}
\end{equation}
\subsection{The zero-average Gaussian free fields on $\cG$ and $\widetilde \cG$.}\label{subsec:zeroavg} Finally we come to the zero-average Gaussian free field on 
$\cG$ (rather the vertex set $V$) and its extension to the metric graph $\widetilde \cG$. Let us 
recall from the introduction that the {\em zero-average Gaussian free field} on the vertices of 
$\cG$ is a centered Gaussian process $(\varphi(x) : x \in V)$ with covariances given by the 
{\em zero-average Green function} $g_{\cG}(\cdot, \cdot)$ on $\cG$ (see, for instance, 
displays~(2.3)--(2.5) in \cite[Section~2]{MR4642820}). The field values $(\varphi(x) : x \in V)$ 
satisfy the following ``zero-average'' property (cf.~\cite[(1.2)]{MR4642820}):
\begin{equation}\label{eq:Gff_sum}
\sum_{x \in V} d_x \varphi(x) = 0, \quad \mbox{$\P$-a.s.}
\end{equation}
where $d_x$ is the degree of the vertex $x \in V$. Enlarging the underlying 
probability space if necessary (with the probability measure $\P$), we extend $\varphi$ to the metric graph $\widetilde \cG$ (also denoted as $\varphi$) as follows. Conditionally on $(\varphi(x): x \in V)$, 
the processes $(\varphi(x): x \in \overline{I_{\{y, z\}}}), \{y, z\} \in E$ joining 
$\varphi(y)$ and $\varphi(z)$. See \cite[Section~2]{MR3502602} for further details. As 
already noted in the introduction, $e \in E^{\ge h}$ (see the beginning of \S\ref{subsec:results}) if and only if $\varphi(x) \ge h$ for all $x \in \overline{I_e}$. 
Accordingly, for any $h \in \R$, we define the {\em level-set above} height $h$ as
\begin{equation}\label{def:Ggeh}
\widetilde \cG^{\ge h} {=} \{ x \in \widetilde \cG : \varphi(x) \ge h \}.    
\end{equation}
The sets $\widetilde \cG^{\le h}, \widetilde \cG^{< h}$ etc. are defined in a 
similar manner. Notice that $\varphi$ is continuous as a function on $\widetilde \cG$ 
and hence $\widetilde \cG^{\ge h}$ is a compact subspace of $\widetilde \cG$. We 
denote the components in $\cG^{\ge h}$, i.e., $\cC_x(S \cap \widetilde \cG^{\ge h})$ 
and $\cC_{S'}(S \cap \widetilde \cG^{\ge h})$ (see below \eqref{eq:conn_set_G}) as 
$\cC_x^{\ge h}(S)$ and $\cC_{S'}^{\ge h}(S)$ respectively.

\smallskip

The field $\varphi$ is, in fact, the {\em canonical Gaussian free field} on the 
Hilbert space $\tilde{H}^1(\widetilde \cG)$ equipped with the Dirichlet form $\widetilde \cE(\cdot, \cdot)$ (see \eqref{def:dirichlet} and the discussion around 
that). See, for instance,~Section~2 in \cite{MR2322706}  and Section~1.7 of Chapter~1 
in \cite{berestycki2015introduction}. In particular,
\begin{equation}\label{eq:gffvar_formulae}
\text{\begin{minipage}{0.75\textwidth}
The process $\big(\widetilde \cE(\varphi, f) : f \in \tilde{H}^1(\widetilde \cG)\big)$ (cf.~\eqref{eq:tildeE_to_EK}) is a {\em Gaussian Hilbert space} (see, e.g.,~\cite[Definition~2.11]{MR2322706}) of centered Gaussian variables with covariances given by $\cov[\widetilde \cE(\varphi, f), \widetilde \cE(\varphi, g)] = \widetilde \cE(f, g)$ for all $f, g \in \tilde{H}^1(\widetilde \cG)$. 
\end{minipage}}
\end{equation}

Owing to the zero-average constraint \eqref{eq:Gff_sum}, the field $\varphi$ does {\em 
not} possess a {\em domain Markov} or {\em strong Markov} property in the traditional sense (see Section~1 in \cite{MR4169171} or \cite{MR4642820}). However, we can still characterize 
the conditional distributions of $\varphi$ utilizing the Gaussian Hilbert space structure 
(see, e.g., \cite[Theorem 1.52]{berestycki2015introduction}). To this end, for any $U \subset \widetilde \cG$, let us consider the sigma-algebra $\cF_U \stackrel{\rm def.}{=} \sigma( \{ \varphi(x) : x \in U \})$. A {\em random} set $K$ taking values in $\cK_{<\infty}$ (see \eqref{def:cKinfty}) satisfies
\begin{equation}\label{def:rndset}  
\mbox{$\{K \subset U\}$ is measurable for each $U \subset \widetilde\cG$}
\end{equation}
and we define $\cF_K$ as the (sub-)$\sigma$-algebra generated by all events of the form $\{K \subset U\}$ with $U \subset \widetilde \cG$ open. Now for any $K \in \cK_{<\infty}$ and function $\phi : K \to \R$, let
\begin{equation}\label{def:fphia}
\text{\begin{minipage}{0.72\textwidth}$f_\phi: \cG \to \R$ denote the function 
minimizing $\widetilde \cE_{K^c}(f, f)$ (cf.~\eqref{def:tildeE-K}) over all $f: \widetilde \cG \to \R$ satisfying $f_{| \overline{K^c}} \in H^1(\overline{K^c})$, $f_{|K} = \phi$ and $\sum_{x \in V} d_x f(x) = 0$.\end{minipage}}
\end{equation}
When $\phi \in H^1(K)$, we can simplify the above as follows.
\begin{equation}\label{def:fphib}
\mbox{$f_\phi$ minimizes $\widetilde \cE(f,f)$ over all $f \in \tilde{H}^1(\widetilde \cG)$ subject to the constraint $f_{|K} = \phi$.}
\end{equation}
We will give an explicit expression for $f_\phi$ in Lemma~\ref{lem:fKexp} in the next 
section. As in \eqref{state:min_energy}, we have that
\begin{equation}\label{def:harmfphi}
    \mbox{$f_\phi$ is linear on $\mathcal I_{-K}$ (see \eqref{def:int_GK}).}
\end{equation}
We are now ready to state the promised conditional law for $\varphi$; cf.~\cite[Proposition~2.1]{MR4169171} for the discrete version. For a formal proof, see 
\cite[Theorem~1.52]{berestycki2015introduction} or \cite[Section~2.6]{MR2322706} 
(cf.~Lemma~\ref{lem:orthogonality} below) and see Theorem~4 in Chapter 2, Section 2.4 of \cite{MR676644} for the version with random sets. Suppose that the random set $K$ 
taking values in $\cK_{<\infty}$ satisfies the property
\begin{equation}\label{eq:smp_condition}
\mbox{$\{K \subset U\} \subset \cF_U$ for every (deterministic) open subsets $U$ of $\widetilde \cG$.}
\end{equation}
Then, with $\tilde{H}^1_0(\overline{K^c})$ denoting the subspace of $\tilde{H}^1(\overline{K^c})$ (see below \eqref{def:dirichlet}) consisting of $f \in \tilde{H}^1_0(\overline{K^c})$ satisfying $f_{|\partial K} = 0$, we have:
\begin{equation}\label{state:smp}
\text{\begin{minipage}{0.65\textwidth}
conditionally on $\cF_K$, $\varphi_{|\overline{K^c}}$ has the same law as $f_{\varphi_{|K}} + \varphi^{\overline{K^c}}$ where $\varphi^{\overline{K^c}}$ is a Gaussian free field on the Hilbert space $\tilde{H}^1_0(\overline{K^c})$.\end{minipage}}
\end{equation}
We end this section with a linear variance estimate for the sum of $\varphi$ over 
subsets of $V$ which will be useful later to control the ``bulk" component of the 
exploration martingale introduced in Section~\ref{sec:exploration}. The bound follows 
from the expression of the zero-average green function on the graph $\cG$ (see \cite[display~(2.17)]{MR4169171}) and the lower bound on the spectral gap $\lambda_\ast$ (cf.~\ref{def:spectral_gap}) as per our standing assumption. We omit the proof.
 \begin{lemma}\label{lem:var_ubd}
For any $K \subset V$, we have
    \begin{equation}\label{eq:var_ubd}
        \var\big[\sum_{x \in K \setminus \{v\}} d_x \varphi(x)\big] \le C |K|.
    \end{equation}
\end{lemma}
\section{zero-average capacity}\label{sec:zeravgcap}
In this section we will formally introduce the {\em zero-average capacity} and discuss 
some of its important properties. The {\em linear isoperimetry} inherent in the lower 
bound on the spectral gap of $\cG$ (see \eqref{eq:isoperimetry}), which is crucial for 
the probability estimates in Theorems~\ref{thm:main1} and \ref{thm:main2}, enter our 
analysis through this object.

There are two distinct 
ways to arrive at the zero-average capacity both of which will be important for us. The first of these is linked to a special expression 
for the function $f_{1_K}$ from \eqref{def:fphia} where $K \in \cK_{<\infty}$ (see \eqref{def:cKinfty}). We state it for general $f_\phi$ for applications in future 
sections.
\begin{lemma}[An expression for $f_\phi$]\label{lem:fKexp}
Let $K \in \cK_{< \infty}$ and $\phi : K \to \R$. 
Then, for 
$x \in \widetilde \cG$, 
\begin{equation}\label{eq:var_phi}
 f_{\phi}(x) = 
 - \nu_\phi \sum_{y \in V} d_y\, g_{K^c}(y, x) + \bE_x[\phi(\widetilde X_{H_K})],
\end{equation}
(see \eqref{def:hitting_time}--\eqref{def:green-K} for notations and compare with \cite[Proposition~2.1]{MR4169171}) where
\begin{equation}\label{def:lambdaphi}
\nu_\phi \stackrel{{\rm def.}}{=} \frac{\sum_{x \in V} d_x \bE_x[\phi(\widetilde X_{H_K})]}{\sum_{x, y \in V} d_y d_x\,g_{K^c}(y, x)} \, (\in \R).
\end{equation}
\end{lemma}
\begin{proof}
Since $K \in \cK_{<\infty}$, i.e., in particular, $V \setminus K \ne \emptyset$ (recall \eqref{def:cKinfty}), the denominator in 
\eqref{def:lambdaphi} is non-zero and hence $\nu_\phi \in \R$. Clearly both $f_\phi$ and the function defined on the right-hand side of 
\eqref{eq:var_phi} equal $\phi$ on $K$. In view of 
\eqref{def:harmgreen} and \eqref{def:harmfphi}, both of these 
functions are also 
linear on 
$\mathcal I_{-K}$ (recall \eqref{def:linear_cI} and \eqref{def:int_GK}). 
Therefore it suffices to prove \eqref{eq:var_phi} {\em only} on $V^K \setminus K$ (see \eqref{def:VK}). However the 
expression \eqref{eq:var_phi}, when restricted to $V^K$, admits of the following 
equivalent formulation 
due to 
\eqref{def:green_enhance}:
\begin{equation}\label{eq:var_phibis}
 f_{\phi}(x) = \nu_\phi \sum_{y \in V} d_y\, 
 \frac{g_{K^c}^K(y, x)}{C_x^K} + \bE_x[\phi(X^K_{H_K})] \mbox{ for $x \in V^K$}
\end{equation}
(see \eqref{def:total_conduc} and \eqref{def:green_enhance} for notations).

In order to prove \eqref{eq:var_phibis}, 
first note that we obtain from the variational formulation in \eqref{def:fphia}, the 
linearity of $f_\phi$ on 
$\mathcal I_{-K}$ 
and the 
identity \eqref{def:dirichletK} that the vector $f_{\phi|V^K} = (f_{\phi}(x) : x \in V^K) \in \R^{V^K}$ minimizes the quadratic form 
\begin{equation}\label{def:fK}
\text{\begin{minipage}{0.58\textwidth}
$\cE^K(f, f) = \sum_{\{x, y\} \in E^K} (f(x) - f(y))^2 C_{\{x, y\}}^K$ over all $f \in \R^{V^K}$ satisfying $f_{| V^K \cap K} = \phi_{|V^K \cap K}$ and $\sum_{x \in V} 
d_x f(x) = 0$.\end{minipage}}
\end{equation}
The space of {\em feasible} solutions is non-empty since $V \setminus K \ne 
\emptyset$. Using the method of Lagrange multipliers, we obtain that the optimal 
vector $f_{\phi| V^K \setminus K}$ is a solution to the system of equations
\begin{equation}\label{eq:var_prob_1}
-\Delta^K f_{\phi|V^K}(x) + \nu \bar d_x = 0; \, x \in V^K \setminus K
\end{equation}
(recall \eqref{def:laplacianK}) for some $\nu \in \R$ where $\bar d_x = d_x$ if $x \in V$ and $0$ otherwise. Since 
$f_{\phi|V^K \cap K} = \phi_{|V^K \cap K}$, we can rewrite this system as
\begin{equation}\label{eq:poisson}
-
\Delta_{K^c}^K f_{\phi|V^K\setminus K} = -\nu \bar d_{V^K \setminus K} + C_{V^K \setminus K, V^K \cap K}^K \, \phi_{|V^K \cap K},
\end{equation}
(recall \eqref{def:laplaciankilled}) where $\bar d_{V^K \setminus K} = (\bar d_x : x 
\in V^K \setminus K) \in \R^{V \setminus K}$ and $C_{V^K \setminus K, V^K \cap K}^K$ 
is the linear map from $\R^{V^K \cap K}$ to $\R^{V^K \setminus K}$ given by
$C_{V^K \setminus K, V^K \cap K}^K {h}(x) = \sum_{y \in V^K \cap K} C^K_{\{x, y\}} {h}(y)$ (${h} \in \R^{V^K \cap K}, x \in V^K \setminus K$). Now 
we obtain the following expression for $f_{\phi}(x)$ in view of \eqref{eq:poisson} 
when $x \in V^K \setminus K$:
\begin{equation*}
f_\phi(x)  = - \nu \sum_{y \in V \setminus K} d_y \, 
\frac{g_{K^c}^K(y, x)}{C_x^K} + \sum_{y \in V^K \setminus K, z \in V^K \cap K}
g_{K^c}^K(x, y)\frac{C_{\{y, z\}}^K}{C_y^K}\phi(z)
\end{equation*}
(cf.~\eqref{def:green_enhance} and \eqref{def:green-K}). Note that, by the {\em last-exit decomposition} for $X^K$ applied to the set $V^K \setminus K$ (see, e.g., \cite[Proposition~4.6.4]{LawLim10}), we have
\begin{equation}\label{eq:last_exit}
\sum_{y \in V^K \setminus K} 
g_{K^c}^K(x, y)\frac{C_{\{y, z\}}^K}{C_y^K} =
\bP_x[X^K_{H_K} = z],
\end{equation}
for any $x \in V^K \setminus K$ and $z \in V^K \cap K$. Plugging this into the expression 
for $f_\phi$ above, we get
\begin{equation*}
\begin{split}
 f_{\phi}(x) = - \nu \sum_{y \in V \setminus K} d_y\,
 \frac{g_{K^c}^K(y, x)}{C_x^K} + \bE_x[\phi(X^K_{H_K})]
 = - \nu \sum_{y \in V} d_y\, 
 \frac{g_{K^c}^K(y, x)}{C_x^K} + \bE_x[\phi(X^K_{H_K})],
\end{split}
\end{equation*}
for $x \in V^K \setminus K$ where in the final step we used that 
$g_{K^c}^K(y, x) = 0$ for $y \in V^K \cap K (\supset V \cap K)$. 
This verifies \eqref{eq:var_phibis} (both expressions equal $\phi$ on $V^K \cap K$) 
except for the equality between $\nu$ and $\nu_\phi$. To this end we will use 
the constraint $\sum_{x \in V} f_K(x)d_x = 0$ (recall \eqref{def:fK}). In view of 
\eqref{eq:var_phi}, with $\nu$ in place of $\nu_\phi$, this constraint leads 
to the equation
\begin{equation*}
\sum_{x \in V} d_x \bE_x[\phi(X^K_{H_K})] - \nu \sum_{y \in V} d_y d_x \, 
\frac{g_{K^c}^K(y, x)}{C_x^K} = 0,
\end{equation*}
whence we get $\nu = \nu_\phi$ as in \eqref{def:lambdaphi} by 
\eqref{def:green_enhance}.
\end{proof}
We are finally ready to define the so-called ``zero-average capacity''.
\begin{defn}[zero-average capacity]\label{def:nuK}
Let $K \in \cK_{<\infty}$ and $1_K: K \to \R$ denote the constant function 1 on $K$. 
We define the {\em zero-average capacity} of $K$ as the quantity $2|E|\nu_K$ where
\begin{equation}\label{def:lambdaK}
\nu_K \stackrel{{\rm def.}}{=} \nu_{1_K} \stackrel{\eqref{def:lambdaphi}}{=}\frac{2|E|}{\sum_{x, y \in V} d_y d_x\,g_{K^c}(y, x)}\,.
\end{equation}
\end{defn}
\begin{remark}\label{rmk:fKexp}
Notice that, in view of \eqref{eq:var_phi}, we can rewrite $f_K \stackrel{{\rm def.}}{=} f_{1_K}$ as 
\begin{equation}\label{eq:var_prob_4}
f_K(x) = - \nu_K \sum_{y \in V} d_y\, g_{K^c}(y, x) + 1.
\end{equation}
\end{remark}
The rationale behind calling $2|E|\nu_K$ the {\em zero-average capacity} is given by 
the following identity. The reader can compare this with a corresponding variational 
formulation of the {\em capacity} of a set in transient graphs, see, e.g., 
\cite[(2.10), p. 18]{MR1743100} 
and also the discussion 
in the second item in \S\ref{subsec:overview}.
\begin{lemma}[Variational formula for $\nu_K$]\label{lem:characterisation_lambda}
For every $K \in \cK_{<\infty}$, one has (cf.~\eqref{def:fphib})
\begin{equation}\label{eq:characterisation_lambda}
\widetilde \cE(f_K, f_K) \stackrel{\eqref{def:dirichletK}, \eqref{def:harmfphi}}{=} \cE^K(f_{K|V^K}, f_{K|V^K}) = 2|E| \nu_K \, (\in (0, \infty)).
\end{equation}
\end{lemma}
\begin{proof}
Applying the (discrete) Gauss--Green formula, we can write
\begin{equation}\label{eq:gaussgreen}
\cE^K(f_{K|V^K}, f_{K|V^K}) = \sum_{x \in V^K} -\Delta^K f_{K|V^K}(x) f_K(x).
\end{equation}
Now by \eqref{eq:var_prob_1}, we have $\Delta^K f_K(x) = \nu_K \bar d_x$ for all $x \in V^K \setminus K$ whereas by \eqref{eq:var_prob_4},
\begin{equation}\label{eq:LfK}
\Delta^K f_{K|V^K}(x) = -\nu_K \sum_{y \in V \setminus K, z \in V^K} d_y \, 
g_{K^c}^K(y, z)\frac{C_{\{z, x\}}^K}{C_z^K} = -\nu_K \sum_{y \in V \setminus K} d_y 
\bP_y[ X^K_{H_K} = x],
\end{equation}
for $x \in V^K \cap K$. Note that we used the last-exit decomposition 
(cf.~\eqref{eq:last_exit}) in the final step. Plugging these into the right-hand side 
of \eqref{eq:gaussgreen} we obtain
\begin{equation}\label{eq:Eff}
\begin{split}
& \cE^K(f_{K|V_K}, f_{K|V_K}) = -\nu_K \sum_{x \in V \setminus K}  d_x f_K(x) + 
\nu_K \sum_{x \in V^K \cap K} f_K(x)\sum_{y \in V \setminus K} d_y  \bP_y[ X^K_{H_K} = x]\\
= & -\nu_K \sum_{x \in V \setminus K}  d_x f_K(x) + 
\nu_K \sum_{y \in V \setminus K} d_y \bP_y[ H_K < \infty] = \big(\sum_{x \in V}d_x\big) \nu_K,
\end{split}
\end{equation}
where we used $f_{K|V^K \cap K} = 1$ in the second step and then in the final 
step along with the fact that $\sum_{x \in V}d_x f_K(x) = 0$.
\end{proof}
We now explore some useful properties of $\nu_K$ starting with monotonicity and 
volume order bounds in the following proposition. The lower bound on $\nu_K$ crucially 
relies on the lower bound on the {\em spectral gap} of the graph $\cG$. In the sequel 
for any $K \subset \widetilde \cG$, we let $b_0(K)$ denote the number of components of 
$K$ and $|K|_V \stackrel{{\rm def.}}{=} |K \cap V|$.
\begin{proposition}[Monotonicity and linear isoperimetry]\label{lem:bnd_lambda}
The mapping $K \mapsto \nu_K: \cK_{<\infty} \to [0, \infty)$ is increasing w.r.t. set inclusion. Further 
for any $K \in \cK_{\infty}$, we have
\begin{equation}\label{eq:bnd_lambda}
c |K|_V \le 2|E|\nu_K \mbox{ whereas, if $|K|_V + b_0(K) \le c |V|$, then } 2|E|\nu_K \le C (|K|_V + b_0(K)).
\end{equation}
\end{proposition}
\begin{proof}
The monotonicity of $\nu_K$ is clear in view of \eqref{def:fphia} and 
Lemma~\ref{lem:characterisation_lambda}. For the lower bound in \eqref{eq:bnd_lambda}, 
consider the set $\floor{K} \stackrel{{\rm def.}}{=} K \cap V$. If $\floor{K} \ne 
\emptyset$, i.e., $|K|_V > 0$ which we can assume without any loss of generality, then $\floor{K} \in \cK_{\infty}$. Hence using the monotonicity 
of $\nu_K$, we get $\nu_K \ge \nu_{\floor{K}}$. On the other hand, from Lemma~\ref{lem:characterisation_lambda} (see also \eqref{def:dirichletK}) 
and \eqref{def:spectral_gap} we have $2|E|\nu_{\floor{K}} \ge c |K|_V$. Combining these two we obtain $c|K|_V \le 2|E|\nu_K$.

\smallskip

For the upper bound in \eqref{eq:bnd_lambda}, we will first construct a set $\ceil{K} \supset K$ in $\cK_{<\infty}$ such that $\partial \ceil{K} \subset V$ and then derive the required upper bound on $\cE(f^\ast, f^\ast)$ where $f^\ast \in \R^V$ satisfies $f^\ast_{|V \cap \ceil{K}} = 1_{V \cap \ceil{K}}$ as well as $\sum_{x \in V} d_x f^\ast(x) = 0$. This would yield the bound on $2|E|\nu_K$ in view of the 
monotonicity of $\nu_K$, Lemma~\ref{lem:characterisation_lambda} (see \eqref{def:dirichletK}) and \eqref{def:fK}.

The set $\ceil{K}$ is formed by taking the union of $K$ with any $\overline{I_e}$ 
such that $K \cap I_e \ne \emptyset$. Clearly, $b_0(\ceil{K}) \le b_0(K)$ and 
$\partial \ceil{K} \subset V$. Before we can define the vector $f^\ast$, we need to 
obtain a suitable upper bound on $|\ceil{K}|_V$. To this end, let us first suppose 
that $K$ is connected, i.e., $b_0(K) = 1$. It follows that in this case, $|\ceil{K}|_V 
\le 2|K|_V + 1$. More generally, this implies that 
\begin{equation}\label{eq:bnd_ceilK}
|\ceil{K}|_V \le 2|K|_V + b_0(K) \,\, \mbox{ for any } K \in \cK_{<\infty}.
\end{equation} 
Now consider $f^\ast \in \R^V$ defined as $f^\ast(x) = 
1$ for $x \in V \cap \ceil{K}$ and
\begin{equation}\label{def:f*V-K}
f^\ast(x) = -\frac{\sum_{x \in V \cap \ceil{K}} d_x}{\sum_{x \in V \setminus \ceil{K}} d_x} \,\, \mbox{ for } x \in V \setminus \ceil{K}.
\end{equation}
$f^\ast$ is well-defined if $\ceil{K} \in \cK_{<\infty}$, i.e., $V \setminus \ceil{K} 
\ne \emptyset$ which happens as soon as $2|K|_V + b_0(K) < |V|$ due to 
\eqref{eq:bnd_ceilK}. Also in this case, $\sum_{x \in V} d_xf(x) = 0$. Further if 
$|K|_V + b_0(K) \le c|V|$, we have $(0 \ge)f^\ast(x) \ge - C$ for $x \in \ceil{K} \cap 
V$. Thus
\begin{equation}\label{eq:bnd_Eff}
\cE(f^\ast, f^\ast) \le  (1 + C)^2 |\ceil{K}|_V \le C' (|K|_V + b_0(K))
\end{equation}
whenever $|K|_V + b_0(K) \le c|V|$ as desired.
\end{proof}
Although the zero-average capacity is monotone like the usual capacity, it lacks an 
important property possessed by the latter, namely the {\em subadditivity} 
(see, e.g., \cite[Proposition~2.2.1(b)]{Law91}). This is evident from the simple 
observation that while $2|E|\nu_{\{x\}} \le C$ for all $x \in V$ in view of 
Proposition~\ref{lem:bnd_lambda}, $\nu_V = \infty$ by Definition~\ref{def:nuK} (or the 
variational formula in Lemma~\ref{lem:characterisation_lambda}). Notwithstanding, 
there is a weaker version which is sufficient for the purpose of this work.
\begin{lemma}[Coarse Lipschitz property]\label{lem:lip_denergy}
Let $K \in \cK_{< \infty}$ and $\cC \subset \widetilde \cG$ be such that $K \cup \cC 
\in \cK_{<\infty}$. Then 
\begin{equation}\label{eq:lip_denergy}
(0 \le ) \, 2|E|(\nu_{K\cup\cC} - \nu_K) \le C (|\cC|_V + b_0(\cC))
\end{equation}
for $|K|_V + |\cC|_V + b_0(K) + b_0(\cC) \le c|V|$.
\end{lemma}
Later in Section~\ref{sec:regcond_mg}, we will prove the (non-quantitative) continuity of the function $K \mapsto \nu_K$ w.r.t. the Hausdorff 
distance on $\cK_{<\infty}$.

\eqref{eq:lip_denergy} follows from a refined version of the argument used for the proof of upper 
bound in Proposition~\ref{lem:bnd_lambda} combined with a special case of the 
following result. We will use the full strength of this result in 
Section~\ref{sec:dirchlet} below. 
\begin{lemma}[Orthogonality]\label{lem:orthogonality}
Let $K \in \cK_{<\infty}$ and $\phi : K \to \R$. Then for any piecewise linear $g : 
\widetilde \cG \to \R$ (see below \eqref{def:linear_cI}) satisfying $g_{|K} = 0$ and $\sum_{x \in V} d_x g(x) = 0$, we have $\widetilde \cE(f_{\phi}, g) = 0$. On the other hand, if $\phi$ is also piecewise 
linear on $K$, then we have $\widetilde \cE(\psi, f_{\phi}) = 0$ for any $\psi : \widetilde{\cG} \to \R$ such that $\psi_{|K} = 0$ and $\sum_{x \in V} d_x \psi(x) = 0$.
\end{lemma}
We will prove Lemma~\ref{lem:orthogonality} shortly and finish the
\begin{proof}[Proof of Lemma~\ref{lem:lip_denergy}]
It suffices to find a piecewise linear function $g$ satisfying (i) $g_{|K} = 0$ and 
$g_{(K \cup \cC) \setminus K} = 1$, (ii) $\sum_{x \in V} d_xg(x) = 0$ and (iii) 
$\widetilde \cE(g, g) \le C(|\cC|_V + b_0(\cC))$. Indeed, it then follows similarly as 
in the proof of the upper bound on $2|E|\nu_K$ in Proposition~\ref{lem:bnd_lambda} that
\begin{equation*}
\nu_{K \cup \cC} \stackrel{\eqref{eq:characterisation_lambda}}{\le} \widetilde \cE(f_K + g, f_K + g) \stackrel{{\rm Lem.}~\ref{lem:orthogonality}}{=} \widetilde \cE(f_K, f_K) + \widetilde \cE(g, g) \stackrel{\eqref{eq:characterisation_lambda}, (iii)}{\le} 2|E|\nu_K + C(|\cC|_V + b_0(\cC)).
\end{equation*}
In order to construct the function $g$, we will use the sets $\ceil{K}$ and 
$\ceil{\cC}$ introduced in the proof of Proposition~\ref{lem:bnd_lambda}. To this end 
consider the vector $g^\ast \in \R^V$ defined as $g^\ast_{|V \cap \ceil{K}} = 0$, $g^\ast_{|V \cap (\ceil{C} \setminus \ceil{K})} = 1$ and 
\begin{equation*}
g^\ast(x) = -\frac{\sum_{x \in V \cap (\ceil{\cC} \setminus \ceil{K})} d_x}{\sum_{x \in V \setminus (\ceil{K} \cup \ceil{\cC})} d_x} \,\, \mbox{ for } x \in V\setminus(\ceil{K} \cup \ceil{\cC})
\end{equation*}
(cf.~\eqref{def:f*V-K}). It follows from \eqref{eq:bnd_ceilK} that $g^\ast$ is well-
defined for $|K|_V + |\cC|_V + b_0(K) + b_0(\cC) \le c|V|$. Now let $g$ be the 
piecewise linear function obtained by interpolating $g^\ast$ linearly on each interval 
$\overline{I_e}$, $e \in E$. It is straightforward to check that $g$ satisfies the 
properties~(i) and (ii) whereas property~(iii) follows from an analysis similar to the 
one leading to \eqref{eq:bnd_Eff}.
\end{proof}
It remains to give the
\begin{proof}[Proof of Lemma~\ref{lem:orthogonality}]
Since $g$ is piecewise linear, we can write
\begin{equation*}
\widetilde \cE_{K^c}(f_\phi + \lambda g, f_\phi + \lambda g) \stackrel{\eqref{def:tildeE-K}}{=} \widetilde \cE_{K^c}(f_\phi, f_\phi) + 2\lambda\, \widetilde \cE_{K^c}(f_\phi, g) + \lambda^2 \widetilde \cE_{K^c}(f_\phi, f_\phi).
\end{equation*}
Therefore, if $\widetilde \cE_{K^c}(f_\phi, g) = \widetilde \cE(f_\phi, g) \ne 0$ 
(recall \eqref{def:tildeE-K} and that $g_{|K} = 0$), we can choose $\lambda$ 
with small enough absolute value so that $\widetilde \cE_{K^c}(f_\phi + 
\lambda g, f_\phi + \lambda g) < \widetilde \cE_{K^c}(f_\phi, f_\phi)$. On the other hand, from the properties of $g$ we know that $f_\phi + \lambda g$ satisfies the 
constraints in \eqref{def:fphia} and hence $\widetilde \cE_{K^c}(f_\phi + 
\lambda g, f_\phi + \lambda g) \ge \widetilde \cE_{K^c}(f_\phi, f_\phi)$ which leads a 
contradiction. Thus $\widetilde\cE(f_\phi, g) = 0$.

For the second part, notice that in view of \eqref{eq:tildeE_to_EK} and 
\eqref{def:linear_cI} as well as the condition $\psi_{|K} = 0$, we can define 
a piecewise linear function $\tilde \psi$ on $\widetilde \cG$ satisfying 
$\tilde \psi_{|K \cup V} = \psi_{|K \cup V}$ and $\tilde \cE(\psi, 
f_\phi) = \widetilde \cE(\tilde \psi, f_\phi)$. Now we can follow a similar 
argument as above 
to deduce $\widetilde\cE(\tilde \psi, f_\phi) = 0$.
\end{proof}

\section{Martingale property of Dirichlet forms}\label{sec:dirchlet}
In this rather short section, we will introduce a process indexed by the sets in 
$\cK_{<\infty}$ (recall \eqref{def:cKinfty}). This process turns out to be a martingale, even when sampled at 
suitable random sets (see Proposition~\ref{prop:MK_martingale} below). Later in 
Section~\ref{sec:exploration}, we will use this to define a (continuous) $[0, \infty)$ 
indexed martingale ``tied to'' the exploration of components in $\widetilde \cG^{\ge h}$ (cf.~\eqref{def:Ggeh}). 
\begin{defn}\label{def:defn_MK}
For any $K \in \cK_{<\infty}$, we define the random variable $M_K$ as (see \eqref{def:dirichlet} and \eqref{def:dirichletK} for notations)
\begin{equation}\label{eq:defn_MK}
M_{K} = \widetilde \cE(\varphi, f_K) \stackrel{\eqref{eq:tildeE_to_EK}}{=} \cE^K(\varphi_{|V^K}, f_{K|V^K}).
\end{equation}
\end{defn}
\begin{lemma}\label{lem:mgexp}
For any $K \in \cK_{<\infty}$, we have
\begin{equation}\label{eq:mgexp}
 M_K = \nu_K \sum_{x \in V^K \cap K} \varphi(x)\sum_{y \in V} d_y  \bP_y[ 
 \widetilde X_{H_K} = x], 
\end{equation}
where $\nu_K$ is as in \eqref{def:lambdaK} (see also \eqref{def:VK} and \eqref{def:hitting_time}). Further, we can write $M_K = M_K^{{\rm blk}} + M_K^{{\rm bdr}}$ with
\begin{equation}\label{eq:M_K_decomp}
M_K^{{\rm blk}} \stackrel{{\rm def.}}{=} \nu_K \sum_{x \in V^K \cap K} d_x \varphi(x) \,\,\, \mbox{and} \,\,\, M_K^{{\rm bdr}} \stackrel{{\rm def.}}{=} \nu_K \sum_{x \in \partial K} \varphi(x) \big(\sum_{y \in V \setminus K} d_y  \bP_y[ 
\widetilde X_{H_K} = x]\big).
\end{equation}
We refer to $M_K^{{\rm blk}}$ and $M_K^{{\rm bdr}}$ as the ``bulk'' and ``boundary'' part 
of $M_K$ respectively. 
\end{lemma}
\begin{proof}
From the (discrete) Gauss-Green formula we have (cf.~\eqref{eq:gaussgreen})
\begin{equation*}
M_K = \sum_{x \in V^K} -\Delta^K f_{K|V^K}(x) \varphi(x).
\end{equation*}
Plugging the expressions \eqref{eq:var_prob_1} and \eqref{eq:LfK} into right-hand 
side of the above display, we obtain
\begin{equation}\label{eq:expMK}
\begin{split}
M_K &= -\nu_K \sum_{x \in V \setminus K} d_x \varphi(x) + \nu_K \sum_{x \in V^K \cap K} \varphi(x)\sum_{y \in V \setminus K} d_y  \bP_y[ 
\widetilde X_{H_K} = x]\\
& = \nu_K \sum_{x \in V \cap K} d_x \varphi(x) + \nu_K \sum_{x \in V^K \cap K} \varphi(x)\sum_{y \in V \setminus K} d_y  \bP_y[ 
\widetilde X_{H_K} = x],
\end{split}
\end{equation}
where we wrote $\widetilde X_{H_K}$ instead of $X^K_{H_K}$ inside the probabilities 
owing to 
\eqref{def:green_enhance} and in the last step we used $\sum_{x \in 
V}d_x\varphi(x) = 0$ (cf.~\eqref{eq:Gff_sum}). Since 
$\bP_y[ 
\widetilde X_{H_K} = x] = 1$ for some $y \in V \cap K$ and $x \in V^K \cap K$ if $x = y$ and $0$ otherwise, we obtain \eqref{eq:mgexp} from \eqref{eq:expMK}.

Alternatively, observe that in the second term on the right-hand side of 
\eqref{eq:expMK}, the only non-zero contributions come from the points 
$x \in \partial K$ whence we obtain \eqref{eq:M_K_decomp}.
\end{proof}
As a direct consequence of 
the expression \eqref{eq:mgexp}, we record below the ``first'' important property of 
$M_K$ which will be useful later to link large component sizes of $\widetilde 
\cG^{\ge h}$ (recall \eqref{def:Ggeh}) to large values of $M_K$ for subsets $K$ of those components. 
\begin{lemma}\label{cor:value_bdr}
Let $K \in \cK_{<\infty}$ and $h \in \R$. 
If $K \subset \widetilde \cG^{\ge h}$, then we have $M_K \ge 2h|E|\nu_K$.
\end{lemma}
\begin{proof}
Note that when $K \subset \widetilde \cG^{\ge h}$,
\begin{equation*}
\begin{split}
M_K &\stackrel{\eqref{eq:mgexp}}{=} \nu_K \sum_{x \in V^K \cap K} \, \sum_{y \in V} d_y  \bP_y[ X^K_{H_K} = x] \varphi(x) \ge h \nu_K \sum_{x \in V^K \cap K} \, \sum_{y \in V} d_y  \bP_y[ X^K_{H_K} = x]\\
&= h \nu_K \sum_{y \in V} d_y \sum_{x \in V^K \cap K}  \bP_y[ X^K_{H_K} = x] = h \nu_K \sum_{y \in V} d_y = 2h|E| \nu_K. \qedhere
\end{split}
\end{equation*}
\end{proof}
We now come to the most important result of this section.
\begin{proposition}[Martingale property]\label{prop:MK_martingale}
Let $K$ be a random set (recall \eqref{def:rndset}) taking values in $\cK_{<\infty}$ 
that satisfies 
property \eqref{state:smp}. Also let $K' 
\supset K$ be a (deterministic) set in $\cK_{<\infty}$. 
Then, 
\begin{equation}\label{eq:MK_martingale}
\E[M_{K'} \mid \cF_{K} ] = M_K, \: \mbox{$\P$-almost surely.}
\end{equation}
\end{proposition}
\begin{proof}
As $K'$ (hence $f_{K'}$) is deterministic, we can write in view of 
\eqref{eq:tildeE_to_EK} and \eqref{eq:defn_MK} that
\begin{equation*}
\E[M_{K'} \mid \cF_{K}] = \widetilde \cE(\E[\varphi \mid \cF_K], f_{K'}), \: \mbox{$\P$-a.s.}
\end{equation*}
Since $K$ satisfies property~\eqref{state:smp}, we further have $\E[\varphi \mid \cF_K] = f_{\varphi_{|K}}$ $\P$-a.s. 
Combining these two, we get
\begin{equation*}
\E[M_{K'} \mid \cF_{K}] = \widetilde \cE(f_{\varphi|K}, f_{K'}).
\end{equation*}
Now noting that $f_{K'} = f_K = 1$ on $K$ (recall that $K 
\subset K'$) and $\sum_{x 
\in V} d_x f_K(x) = \sum_{x \in V} d_x f_{K'}(x) = 0$ (recall \eqref{def:fphia}), we 
can deduce from Lemma~\ref{lem:orthogonality} (the first part) and \eqref{eq:defn_MK} 
that
\begin{equation*}
\widetilde \cE(f_{\varphi|K}, f_{K'}) = \widetilde \cE(f_{\varphi|K}, f_{K'} - f_K) + \widetilde \cE(f_{\varphi|K}, f_K) = \widetilde \cE(f_{\varphi|K}, f_K).
\end{equation*}
Similarly since $\varphi = f_{\varphi|K}$ on $K$ and $\sum_{x \in V} d_x f_{\varphi|K}(x) = \sum_{x \in V} d_x \varphi(x) = 0$ \eqref{eq:Gff_sum}, we get 
\begin{equation*}
\widetilde \cE(f_{\varphi|K}, f_K) = \widetilde \cE(f_{\varphi|K} - \varphi, f_K) + \widetilde \cE(\varphi, f_K) = \widetilde \cE(\varphi, f_K) \stackrel{\eqref{eq:defn_MK}}{=} M_K,
\end{equation*}
using the second part of Lemma~\ref{lem:orthogonality}. Together, the last three 
displays imply 
\eqref{eq:MK_martingale}.
\end{proof}
If we look at the conditional variance instead, we get:
\begin{lemma}\label{lem:cond_variance}
In the same setup as in Proposition~\ref{prop:MK_martingale}, we have
\begin{equation}\label{eq:qua_var}
           \Var[ M_{K'} \mid \cF_K] = 2|E|(\nu_{K'} - \nu_K) \,\mbox{ $\P$-almost surely.}
       \end{equation}
 \end{lemma}
 \begin{proof}
First of all, let us note that $M_{K'}$ is square integrable owing to its definition 
in \eqref{eq:defn_MK}, the variance formula \eqref{eq:gffvar_formulae} and 
also \eqref{eq:characterisation_lambda}. So we can use the following standard 
decomposition:
\begin{equation}\label{eq:varMK'decomp}
\var[M_{K'}] = \var[\E[M_{K'} | \cF_K] ] + \E[ \var[M_{K'} | \cF_K] ] \stackrel{\eqref{eq:MK_martingale}}{=} \var[M_K] + \E[\var[M_{K'} | \cF_K] ].
\end{equation}
Now suppose that $K$ is deterministic. Then from \eqref{eq:defn_MK}, 
\eqref{eq:gffvar_formulae} and \eqref{eq:characterisation_lambda}, we get
\begin{equation}\label{eq:varMK}
 \var[M_S] = 2|E|\nu_S 
 \mbox{ for any $S \in \cK_{<\infty}$.}
\end{equation} 
On the other hand, in view of \eqref{state:smp}, 
we obtain that the conditional variance $\var[M_{K'} | \cF_K]$ 
is constant almost surely. Plugging these two observations into 
\eqref{eq:varMK'decomp} gives us
\begin{equation*}
\var[M_{K'} | \cF_K] = 2|E|(\nu_{K'} - \nu_K)\, \mbox{a.s.}
\end{equation*}
for any deterministic $K$ (cf.~\eqref{eq:qua_var}). Hence the same expression also holds when $K \subset K'$ is random but satisfies 
property~\eqref{state:smp}, thus yielding the lemma.
\end{proof}

\section{Regularity of the process $(M_K)_{K \in \cK_{<\infty}}$}\label{sec:regcond_mg}
In this section, we will establish the continuity of the functions $K \mapsto M_K$ and 
$K \mapsto \nu_K$ as $K$ ranges over the family $\cK_{<\infty}$ (see Definitions~\ref{def:defn_MK} and \ref{def:nuK} and also \eqref{def:cKinfty}). These will be used to ensure the continuity of a $[0, \infty)$-indexed ``exploration 
martingale'' and its quadratic variation process which we introduce in \S\ref{subsec:martingale}.

The relevant topology on $\cK_{<\infty}$ is given by the {\em Hausdorff distance} which we briefly recall 
below. For any $S, S' \subset \widetilde \cG$, the Hausdorff distance (corresponding 
to the metric $d(\cdot, \cdot)$ on $\widetilde\cG$, recall \eqref{def:distance}) between them is defined as 
\begin{equation}\label{eq:haus_dist}
        d_{\mathrm{H}} (S, S') {=} \inf \{ \varepsilon \ge 0 : S \subset 
        B(S', \varepsilon) \text{ and } S' \subset 
        B(S, \varepsilon)\},
    \end{equation}
where 
$B(S, \varepsilon) \stackrel{{\rm def.}}{=} \cup_{x \in S} \{ y \in \widetilde{\cG} : d( y, x ) < \varepsilon\}$ denotes the $\varepsilon$-neighborhood of $S$. $d_{{\rm H}}$ is a metric on $\cK$ (see, e.g., \cite[pp.~280-281]{MR3728284}). In the 
sequel, we use $\cK$ (or $\cK_{<\infty}$) to refer to both the collection and the 
associated metric space depending on the context. To keep the exposition simple, and 
because it is sufficient for our purpose, we only prove 
a restricted version of continuity as follows (see also Proposition~\ref{lem:lambda_cont} below).
\begin{prop}[On continuity of $K \mapsto M_K$]\label{lem:mg_cont}
Let $K, K' \in \cK_{<\infty}$ be such that either $K \subset K'$ or $K' \subset K$. 
Then for every $\varepsilon > 0$, there exists $\delta = \delta(\varepsilon, K,\cG, 
\varphi) > 0$ satisfying
    \begin{equation*}
            | M_K - M_{K'} | < \varepsilon \mbox{ whenever } d_{\mathrm{H}}( K,K') < 
            \delta.
    \end{equation*}
\end{prop}
The proof needs a similar result for $K \mapsto \nu_K$ as an ingredient which is 
important on its own.
\begin{prop}[On continuity of $K \mapsto \nu_K$]\label{lem:lambda_cont}
    Let $K, K' \in \cK_{<\infty}$ be as in Proposition~\ref{lem:mg_cont}. 
    Then for every $\varepsilon > 0$, there exists $\delta = \delta( \varepsilon, K,\cG) > 0$ 
    satisfying
    \begin{equation*}
            |\nu_K - \nu_{K'}| < \varepsilon \mbox{ whenever } d_{\mathrm{H}}( K,K') < \delta.
    \end{equation*} 
\end{prop}
Assuming Proposition~\ref{lem:lambda_cont}, let us finish the
\begin{proof}[Proof of Proposition~\ref{lem:mg_cont}]
Let us recall the expression of the martingale from \eqref{eq:mgexp} in 
Lemma~\ref{lem:mgexp}: 
 \begin{equation*}
         M_K = \nu_K \sum_{z \in V} \sum_{x \in V^K \cap K}  \varphi( x ) \bP_z[ X_{H_K} = x], \text{ and } M_{K'} = \nu_{K'} \sum_{z \in V} \sum_{x \in V^{K'}\cap K' } \varphi( x ) \bP_z[ X_{H_{K'}} = x],
 \end{equation*}
 where $X$ is the (discrete-time) walk $X^{K, K'}$ (cf.~\eqref{def:KK'} and \eqref{def:VK} for relevant notations) and we invoked 
 \eqref{def:green_enhance} in replacing $\widetilde X$ with $X$ inside the probabilities. In view of 
 Proposition~\ref{lem:lambda_cont} and since $\varphi$ is a continuous 
function on $\widetilde{\cG}$, 
it therefore suffices to prove the following statement for any continuous function $f$ 
on $\widetilde{\cG}$, $z \in V$ and $\varepsilon \in (0,1)$. 
Letting
\begin{equation}\label{def:FK}
        F(z,K)  \stackrel{\rm def.}{=} \sum_{x \in V^K \cap K} f(x) \bP_z[ X_{H_K} = x],  \text{ and } F(z,K')  \stackrel{\rm def.}{=} \sum_{x \in V^{K'} \cap K'} f(x) \bP_z[ X_{H_{K'}} = x],
\end{equation}
there exists $\delta = \delta(\varepsilon, K, \cG, f) > 0$ such that 
\begin{equation}\label{eq:FKK'}
|F(z,K) - F(z,K')| < \varepsilon  
\end{equation} 
whenever $d_\mathrm{H}(K,K') < \delta$. 
To this end, let $\mathrm{M}_{f} \stackrel{{\rm def.}}{=} \sup_{x \in \widetilde{\cG}} 
|f(x)| \vee 1$. Also let 
\begin{equation*}
\delta_{f, \varepsilon} \stackrel{{\rm def.}}{=} \sup\{\delta \in (0, 1) : |f(x) - f(y)| < \tfrac{\varepsilon}8 \text{ whenever } d(x,y) < \delta, x, y \in \widetilde{\cG} \}.
\end{equation*}
Note that $\mathrm{M}_f < \infty$ and $\delta_{f, \varepsilon} > 0$ since $f$ is continuous on 
$\widetilde \cG$ which is compact. Now let
\begin{equation}\label{def:delta_mg}
        \delta \stackrel{\rm def.}{=} \frac{\inf_{x \ne y \in V^{K} } d( x, y) \wedge \delta_{f, \varepsilon}}{100} \cdot \frac{\varepsilon}{\mathrm{M}_{f} |V^K \cap K| }.
\end{equation}
We call $x \in V^{K'} \cap K'$ as {\em visible} from $V$ if 
either $x \in V$ or $x \in (y_1, y_2)_{\cG}$ for some $y_1, y_2 \in V$ such that 
$[y_1, x)_{\cG} \cap K' = \emptyset$ (cf.~\eqref{def:intG} for notation). Let $D_{{\rm v}}^{K'}$ denote the set of all visible points in $V^{K'} \cap K'$. Then we can write
\begin{equation}\label{eq:FK'mod}
F(z,K') = \sum_{x \in D_{{\rm v}}^{K'}} f(x) \bP_z[ X_{H_{K'}} = x],
\end{equation}
as $\bP_z[X_{H_{K'}} = x] = 0$ when $x \notin D_{{\rm v}}^{K'}$ (see \eqref{def:FK}).
We claim that 
\begin{equation}\label{eq:dhDbar}
D_{{\rm v}}^{K'} \subset 
B(V^K \cap K, 2\delta).
\end{equation}
\eqref{eq:dhDbar} follows in fairly straightforward manner from our definition of 
$\delta$ in \eqref{def:delta_mg}. So we omit further details 
and proceed with the proof of \eqref{eq:FKK'} assuming this. 

\smallskip

From the choice of $\delta$ in \eqref{def:delta_mg}, it is clear that $d(x, y) > 4\delta$ for any two distinct $x, y \in V^K \cap K$ and as a result, $B(x, 2\delta) \cap B(y, 2\delta) = 
\emptyset$ for such $x$ and $y$. On the other hand, from \eqref{eq:dhDbar} we have 
$D_{{\rm v}}^{K'} \subset \cup_{x \in V^K \cap K} B(x, 2\delta)$. In view of \eqref{def:FK} and 
\eqref{eq:FK'mod}, these two observations give us the expression:
\begin{equation*}
|F(z,K ) - F(z,K')|  = \Big |\sum_{ x \in V^K \cap K}  f ( x )  \bP_z[ X_{H_K} = x] - \sum_{x \in V^K \cap K} \sum_{y \in  B(x, 2\delta) \cap D^{K'}_{\rm v}} f( y ) \bP_z[ X_{H_{K'}} = y] \Big|.
\end{equation*}
Further since $2\delta < \delta_{f, \varepsilon}$ (see \eqref{def:delta_mg} and 
the display above it) and $\sup_{x \in \widetilde{\cG}} |f(x)| \le M_f$, we can bound the 
above as 
\begin{equation}\label{eq:mg_cont_2}
         | F(z, K ) - F(z,K') | \le \mathrm{M}_{f} \sum_{ x \in V^K \cap K}  \big|\bP_z[ X_{H_K} = x] - \bP_z[ X_{H_{K'}} \in B(x, 2\delta) \cap D^{K'}_{\rm v}]\big|  + \frac{\varepsilon}{8}. 
\end{equation}
At this point, we will split our analysis into two separate cases, namely $K \subset K'$ 
and $K' \subset K$ starting with the former.


\smallskip

\noindent{\em The case $K \subset K'$.} Using the strong Markov property of $X$ in the 
second line below, we obtain for any $x \in V^K \cap K$:
\begin{equation}\label{eq:mg_cont_3}
    \begin{split}
        &  \bP_z[ X_{H_K} = x] - \bP_z[ X_{H_{K'}} \in B(x, 2\delta) \cap D^{K'}_{\text{v}}] \\
        &=  \sum_{y \in D^{K'}_{{\rm v}}} \bE_z\big[ 1{ \{ X_{H_{K'}} =y \} } \bP_y [ X_{H_K} = x ] \big]  - \bP_z[ X_{H_{K'}} \in B(x, 2\delta) \cap D^{K'}_{\rm v}] \\
    &= P^+_x + P^-_x- \bP_z[ X_{H_{K'}} \in B(x, 2\delta) \cap D^{K'}_{{\rm v}}], 
    \end{split}
\end{equation}
where
\begin{equation}\label{def:Px}
\begin{split}
&P^+_x = \sum_{y \in B( x, 2\delta ) \cap D^{K'}_{\rm v}} \bE_z\big[ 1_{ \{ X_{H_{K'}} =y \} } \bP_y [ X_{H_K} = x ] \big], \text{ and} \\   
&P^-_x = \sum_{y \in D^{K'}_{{\rm v}} \setminus  B( x, 2\delta )} \bE_z\big[ 1_{ \{ X_{H_{K'}} = y \}}\bP_y [ X_{H_K} = x ]\big].
\end{split}
\end{equation} 
Now for any $y \in B(x, 2\delta) \cap D^{K'}_{\rm v}$, we want to obtain a lower bound on 
$\bP_y [X_{H_K} = x]$ where we may assume $y \ne x$ without any loss of generality. Since $2\delta < \inf_{x \ne y \in V^{K}} d(x, y)$ by \eqref{def:delta_mg} and $d(x, 
y) < 2\delta$ with $x \in V^K \cap K$, there exist $y_1, y_2 \in V$ such that $x, y \in [y_1, y_2]_{\cG}$ and $y \in (y_1, x)_{\cG}$. Also from \eqref{def:delta_mg} we know that $d(x, x') \ge \frac{100 M_f|V^K \cap K|}{\varepsilon}\delta$ for any $x' (\ne x) \in D^K$. Hence there exists $y_3 \in [y_1, x]_{\cG} \cap V^{K, K'}$ such that $\rho([y_1, 
x]_{\cG}) \ge \frac{100 M_f|V^K \cap K|^2}{\varepsilon}\delta$ and 
\begin{equation*}
    \{ X_{H_K} = x, X_0 =y \} \supset \{ X_{T_{(y_3,x)_{\cG}}} = x, X_0 =y \},
\end{equation*}
where $T_{(y_3, x)_{\cG}}$ is the exit time of $X$ from the interval $(y_3, x)_{\cG}$ 
(cf.~\eqref{def:hitting_time}).
Then using similar arguments as used in the proof of Lemma~\ref{lem:small_greens_func} 
below (see around \eqref{eq:lowebd_prob_3} in particular), we get
\begin{equation}\label{eq:mg_cont_4}
        \bP_y [ X_{H_K} = x ] \ge 1 - \frac{\varepsilon}{100 \mathrm{M}_{f} | V^K \cap K| }. 
\end{equation}
Now from the definition of $P^+_x$ in \eqref{def:Px}, we have 
\begin{equation*}
    P^+_x -  \bP_z[ X_{H_{K'}} \in B(x, 2\delta) \cap D^{K'}_{\rm v}] \le \sum_{y \in B( x, 2\delta ) \cap D^{K'}_{\rm v}} \bE_z[ 1_{ \{ X_{H_{K'}} =y \} } ] - \bP_z[ X_{H_{K'}} \in B(x, \delta) \cap D^{K'}_{\rm v}] = 0.
\end{equation*}
On the other hand, from \eqref{eq:mg_cont_4} we get,
\begin{equation*}
    P^+_x -  \bP_z[ X_{H_{K'}} \in B(x, 2\delta) \cap D^{K'}_{\rm v}] 
    \ge  - \frac{\varepsilon}{100 \mathrm{M}_{f} | V^K \cap K| }. 
\end{equation*}
Together, these two bounds imply
\begin{equation}\label{eq:mg_cont_6}
\big| P^+_x -  \bP_z[ X_{H_{K'}} \in B(x, 2\delta) \cap D^{K'}_{\rm v}]\big| \le  \frac{\varepsilon}{100 \mathrm{M}_{f} | V^K \cap K| }.
\end{equation}
Next we want to bound the term $P_x^-$ from above. Since $D^{K'}_{\rm v} \subset \cup_{x \in V^K \cap K} B(x, 2\delta)$, any $y \in D^{K'}_{\rm v} \setminus B(x, 2\delta)$ belongs to 
$B(x', 2\delta)$ for some $x' (\ne x) \in V^K \cap K$. Then \eqref{eq:mg_cont_4} applied to the 
point $x'$ gives us that 
\begin{equation*}
\bP_y [X_{H_K} = x'] \le \frac{\varepsilon}{100 \mathrm{M}_{f} | V^K \cap K|}.
\end{equation*}
Plugging this into the expression of $P_x^-$ in \eqref{def:Px}, we obtain
\begin{equation}\label{eq:mg_cont_7}
P^-_x \le \frac{\varepsilon}{100 \mathrm{M}_f | V^K \cap K|}. 
\end{equation}
Combining \eqref{eq:mg_cont_6} and \eqref{eq:mg_cont_7}, we get from \eqref{eq:mg_cont_3}:
\begin{equation}\label{eq:mg_cont_8}
         \big| \bP_z[ X_{H_{K}} = x] - \bP_z[ X_{H_{K'}} \in B ( x, 2\delta) \cap D^{K'}_{\rm v}] \big| \le \frac{\varepsilon}{100 \mathrm{M}_f |V^K \cap K|}.
\end{equation}
Plugging \eqref{eq:mg_cont_8} into \eqref{eq:mg_cont_2}, we obtain
\begin{equation}\label{eq:mg_cont_9}
    \begin{split}
        | F(z, K ) - F(z, K') | &\le \mathrm{M}_f \frac{\varepsilon}{100 \mathrm{M}_f  | V^K \cap K|} | V^K \cap K | + \frac{\varepsilon}{8} < \varepsilon,
    \end{split}
\end{equation}
which is precisely the bound \eqref{eq:FKK'} (in this case).

\smallskip

\noindent{\em The case $K' \subset K$.} This case can be dealt with in a similar 
fashion by applying the strong Markov property 
at time $H_{K'}$ instead of $H_K$. 
\end{proof}
We now prepare for the proof of Proposition~\ref{lem:lambda_cont} which requires the 
following lemma.
\begin{lemma}\label{lem:small_greens_func}
Let $K, K' \in \cK_{<\infty}$ be as in Proposition~\ref{lem:mg_cont}. Then for every 
$\varepsilon > 0$, there exists $\delta = \delta(\varepsilon, K, \cG) > 0$ such that 
(see \eqref{def:green-K} for notation)
\begin{equation*}
\max_{x \in K, y \in V}g_{(K')^c}(x, y) \vee \max_{x \in K', y \in V}g_{K^c}(x, y) < \varepsilon \mbox{ whenever } d_{\mathrm{H}}( K,K') < \delta.
\end{equation*}
    
\end{lemma}
We prove Lemma~\ref{lem:small_greens_func} at the end of this section and proceed with 
the
\begin{proof}[Proof of Proposition~\ref{lem:lambda_cont}]
In view of the expression of $\nu_K$ from \eqref{def:lambdaK}, 
it is enough to prove the statement with
$$g_{K^c}(x, y) \stackrel{\eqref{eq:invar_greensfunc}, \eqref{def:green_enhance}}{=} \tfrac{g_{K^c}^{K, K'}(x, y)}{C_y^{K, K'}},$$
in place of $\nu_K$ for all $x, y \in V$ (see~\eqref{def:green_enhance} and \eqref{def:KK'} for notations). To this end, given $\varepsilon \in (0, 1)$, let 
$\delta = \delta(\varepsilon, K, \cG) > 0$ 
be as in Lemma~\ref{lem:small_greens_func}. 
We only give the proof when $K' \subset K$ and the proof for 
the other case
is similar.

\smallskip

So let $K' \, (\in \cK_{<\infty}) \subset K$ be such that $d_{\mathrm{H}}(K,K') < 
\delta$. By the strong Markov property of the walk $X = X^{K, K'}$, we can write
\begin{equation*}
       \frac{g_{(K')^c}^{K, K'}(x, y)}{C_y^{K, K'}} -  
       \frac{g_{K^c}^{K, K'}(x, y)}{C_y^{K, K'}} = \sum_{z \in V^K} \bP_x [X_{H_{K}} = z] \, \frac{g_{(K')^c}^{K, K'}(z, y)}{C_y^{K, K'}} \stackrel{{\rm Lem.}~\ref{lem:small_greens_func}, \, \eqref{def:green_enhance}}{\le} \varepsilon \, \bP_x [H_{K} < \infty] = \varepsilon, 
\end{equation*}
which is precisely the required bound.
\end{proof}
It remains to give the
\begin{proof}[Proof of Lemma~\ref{lem:small_greens_func}]
The proof is split into two parts depending on whether $K' \subset K$ or $K \subset K'$.

\smallskip
    
\noindent{\em The case $K' \subset K$.} Notice that we only need to check the bound on 
$g_{(K')^c}(x, y)$ in this case with $x \in K \setminus K'$. Also due to \eqref{def:harmgreen}, it suffices to consider points $x \in V^K \setminus K'$. Now given $\varepsilon \in 
(0, 1)$, let
    \begin{equation}\label{def:delta}
          \delta \stackrel{\rm def.}{=} \frac{\inf_{x \neq y \in V^{K}} d(x, y)}{(\mathrm{G} \vee 1) 
          d} \cdot \frac{\varepsilon}{10},
    \end{equation}
    with $\mathrm{G}$ from 
    \eqref{eq:sup_greensfunc} 
    and $K' \, (\in \cK_{<\infty}) \subset K$ such that $d_{\mathrm{H}}( K, K') < \delta$. 
    In view of \eqref{def:green_enhance}, we have $g_{(K')^c}(x, y) = \frac{g_{(K')^c}^{K, K'}(x, y)}{C_y^{K, K'}}$. In the sequel, we substitute $\tilde g_{(K')^c}(\cdot, \cdot)$, $\tilde g_{K^c}(\cdot, \cdot)$, $X$ and $C_{(\cdot, \cdot)}$ for $g_{(K')^c}^{K, K'}(\cdot, \cdot), g_{K^c}^{K, K'}(\cdot, \cdot), X^{K, K'}$ and $C_{\{\cdot, \cdot\}}^{K, K'}$ respectively to ease notation. Using the Markov property of the (discrete-time) walk $X$, 
    we can write 
    \begin{equation}\label{eq:decomp_gfunc}
            \frac{\tilde g_{(K')^c}(x, y)}{C_y} 
            = \frac{1_{x = y} +  \bE_x[ 
            \tilde g_{(K')^c}(X_1, y) 1_{ \{X_1 \notin 
            K'\}}]}{C_y} \stackrel{\eqref{eq:sup_greensfunc}, \eqref{def:green_enhance}}{\le} \frac{\ind_{x = y}}{C_y} + \mathrm{G} \bP_x[ X_1 \notin 
            K'],
    \end{equation}
    where $x \in V^K$ and $y \in V$. We will derive suitable bounds on each of these 
    terms. 
    Let us start by noting the following expression:
    \begin{equation}\label{eq:lowerbd_prob}
            \bP_x[ X_1 \in 
            K'] = \frac{\sum_{\{x, z\} \in E^{K, K'},\, 
            z \in K'} C_{\{x, z\}}}{\sum_{\{x, z\} \in E^{K, K'}}C_{\{x, z\}}} = \frac{1}{ 1 +  \frac{\sum_{\{x, z\} \in E^{K, K'},\, z \in K \setminus K'} C_{(x, z)}}{\sum_{\{x, z\} \in E^{K, K'},\, z \in K'} C_{\{x, z\}}}}.
    \end{equation}
    For $z \in K$ such that $\{x, z\} \in E^{K, K'}$, we have 
    $$C_{\{x, z\}} \stackrel{\eqref{def:C^K},\, \eqref{def:distance}}{=} \tfrac1{d(x,z)} \stackrel{\eqref{def:delta}}{\le} \tfrac{\varepsilon}{6d^2 \mathrm{G} \delta}.$$
    On the other hand, since $d_{\mathrm{H}}(K,K') < \delta$ and $x \in V^K \setminus K'$, there exists $z \in K'$ such that $d(x, z) < \delta$ (recall 
    \eqref{eq:haus_dist}). In fact, as $\delta < \inf_{x \ne y \in V^K} d(x, y)$ by our choice in \eqref{def:delta}, we have $(x, z) \in E^{K, K'}$ and thus $C_{\{x, z\}} = \frac{1}{d(x,z)} \ge \frac1{\delta}$. Plugging these two bounds into the right-hand side of 
 \eqref{eq:lowerbd_prob}, we get
 \begin{equation*}
         \bP_x[ X_1 \in K']  \ge 1 -\frac{\varepsilon}
         {10\mathrm{G}}.
 \end{equation*}
 A similar reasoning also tells us that $C_x \ge \frac1{\delta}$. Combined with 
 \eqref{eq:decomp_gfunc}, these estimates give us 
 \begin{equation*}
         \frac{\tilde g_{(K')^c}(x, y)}{C_y} \le \delta +  \mathrm{G} \frac{\varepsilon}{10\mathrm{G}} 
         \stackrel{\eqref{def:delta}}{\le} \frac{\varepsilon}{10} + \frac{\varepsilon}{10} < \varepsilon.
 \end{equation*}

\medskip

\noindent{\em The case $K \subset K'$.}  In this case we only need to check the bound 
on $g_{K^c}(x, y)$ for $x \in V^{K'} \setminus K$. 
 Let $K' \, (\in \cK_{<\infty}) \supset K$ be such that $d_{\mathrm{H}}(K, K') < 
 \delta$ where $\delta$ is still as in \eqref{def:delta}.

 Now let $\hat{x}$ be a point in $V^{K, K'} \cap K$ achieving the minimum distance from $x$. Using similar arguments as in the previous case, we obtain that $\{x, \hat{x}\} \in 
 \overline{I_{\{z_0, z_1\}}}$ for a unique $\{z_0, z_1\} \in E$. 
Without loss of generality, suppose that $x \in [z_0, \hat x]_{\cG}$. 
Then by the strong Markov property of $X$ as well as \eqref{eq:sup_greensfunc} and 
\eqref{eq:invar_greensfunc}, we get the following analogue of \eqref{eq:decomp_gfunc}:
 \begin{equation}\label{eq:decom_gfunc_1}
         g_{K^c}(x, y) = \frac{\tilde g_{K^c}(x, y)}{C_y} 
         \le  \frac{\ind_{x = y}}{C_y} + \mathrm{G} \bP_x[ X_{T_{(z_0, \hat x)_{\cG}}} \notin K ],
 \end{equation}
 where, as defined in \eqref{def:hitting_time}, $T_{(z_0, \hat x)_{\cG}}$ is the exit 
 time of the walk from the set $[z_0, \hat x]_{\cG}$. 
 From standard one-dimensional gambler's ruin type computation (cf.~\eqref{def:C^K}), 
 it follows that 
 \begin{equation}\label{eq:lowebd_prob_3}
         \bP_x[ X_{T_{(z_0, \hat x)_{\cG}}} \in K ] \ge 
         \frac{1}{1 + \frac{d(x, \hat{x})}{d(x , x_0)}}.
 \end{equation}
 Remainder of the proof now follows similarly as in the previous case with \eqref{eq:lowebd_prob_3} and \eqref{eq:decom_gfunc_1} replacing 
 \eqref{eq:lowerbd_prob} and \eqref{eq:decomp_gfunc} respectively. 
\end{proof}

\section{Exploration of level-set components}\label{sec:exploration}
So far, we did not talk so much about percolation. In this section, we relate the 
percolation of $\widetilde \cG^{\ge h}$ (recall \eqref{def:Ggeh}) to the process 
$(M_K)_{K \in \cK_{<\infty}}$ that we introduced and analyzed in the previous two 
sections. In order to do this, we first construct a continuous family of random sets 
$(K_t)_{t \ge 0}$ in \S\ref{subsec:exp_scheme} which arise from a certain exploration 
scheme revealing the components of $V$ in $\cG^{\ge h}$ successively. Then in 
\S\ref{subsec:martingale}, we define the exploration martingale $(\tilde M_t)_{t \ge 0}$ as $(M_{K_t})_{t \ge 0}$ and discuss a few of its useful properties.
\subsection{Exploration process}\label{subsec:exp_scheme}
Let $v \in V$ and $N > 1$ be an integer. By exploring the components of $(V \setminus \{v\}) \cup \partial B(v, \frac 1N)$ inside $\widetilde{\cG}^{\ge h} \setminus B(v, \frac 1N)$ (recall that $|V| > 1$) in a certain order, starting with those of $\partial B(v.\frac 1N)$, and ``interpolating'' between them, one 
can obain a {\em continuous} family (w.r.t. Hausdorff distance) of non-decreasing random compact sets $(K_t)_{t \ge 0}$. We restrict our exploration {\em away} from the 
vertex $v$ in order to prevent $K_t$'s from consuming the entire vertex set $V$. This 
would enable us to define a martingale $(\tilde M_t)_{t \ge 0}$ attached 
to $(K_t)_{t \ge 0}$ in the next subsection using the results from 
Section~\ref{sec:dirchlet}. $N$ is essentially a spurious parameter (see 
Lemma~\ref{eq:rel_component} below) and will be sent to $\infty$ at the end in all 
applications. We also want to design the exploration in a way so that the sets $K_t$ 
satisfy some favorable properties, mainly in relation to the process $(\tilde M_t)_{t \ge 0}$. This is the subject of Proposition~\ref{prop:Kt}. See below \eqref{def:Ggeh} 
in Section~\ref{sec:prelim} to recall our notations regarding the components in 
$\widetilde \cG^{\ge h}$.
\begin{proposition}[Exploration process]\label{prop:Kt}
 There exists a family of (random) compact subsets $(K_t)_{t \ge 0} = (K_t(v, h, N))_{t \ge 0}$ of $\widetilde{\cG} \setminus B(v, \frac 1N)$, with $K_0 = K_0(v, N) \stackrel{{\rm def.}}{=} \partial B(v, \frac 1N)$, satisfying the following 
 properties: 
\begin{enumerate}[label = {\arabic*}., ref = {\arabic*}]
\item \label{Kt1} 
$(K_t)_{t \ge 0}$ is increasing in $t$ with respect to set inclusion and the map $t \mapsto K_t: [0, \infty) \to \cK$ is continuous (see below \eqref{eq:haus_dist}).
\item \label{Kt2} 
For all $t \ge 0$, each component of $K_t$ intersects $K_0$ and $\{K_t \subset U\} \in \mathcal F_U$ for any open $U \subset \widetilde \cG$ (cf.~\eqref{state:smp}).
\item \label{Kt3} There exist $n \in \N$ and (random) timepoints $0 \le \tau_{1, 1} 
\le \tau_{1, 2} \le \ldots \le \tau_{n, 1} \le \tau_{n, 2} < \infty$ such that the following hold $\P$-a.s.
\begin{enumerate}[label=\roman*., ref=\arabic{enumi}-\roman*]
\item \label{Kt3-1} $K_{\tau_{1, 1}} = \cC_{K_0}^{\ge h}(B^c(v, \frac 1N)) \cup K_0$ 
whereas $V \setminus \{v\} \subset K_{\tau_{n, 2}}$ and $K_t = K_{\tau_{n, 2}}$ for all $t \ge \tau_{n, 2}$.
\item  \label{Kt3-2}
If $t \in [\tau_{k, 1}, \tau_{k, 2}]$, then $K_t = K_{\tau_{k, 1}} \cup I_t$ where 
$I_t$ is an interval and $K_{\tau_{k, 1}}$ further satisfies $\partial K_{\tau_{k, 1}} \setminus \widetilde \cG^{\le h} \subset K_0$ 
and $\bP_x[\widetilde X_{H_{K_{\tau_{k, 1}}}} \in 
K_0 \setminus \widetilde \cG^{\le h}] = 0$ for all $x \in V \setminus \{v\}$. 
\item \label{Kt3-3}On the other hand if $t \in \, (\tau_{k, 2}, \tau_{k+1, 1}]$, 
then $K_t = K_{\tau_{k, 2}} \cup \cC_t$ where $\cC_t \in \cK$ 
is a connected subset of $\cC^{\ge h}_{x_k}(B^c(v, \tfrac 1N))$ with $x_k \in V$.
\end{enumerate}
\end{enumerate}
\end{proposition}
\begin{proof}
We will do a certain type of {\em breadth-first exploration}. However, since we are 
exploring the components in the metric graph $\widetilde \cG$ and we need to exercise 
delicate control on the nature of the sets $K_t$'s as well as the value of $\varphi$ 
on these sets (see properties~\ref{Kt2} and \ref{Kt3}), we need to take additional 
care.

Let us start by describing in detail the timepoints $\tau_{1, 1}$ and $\tau_{1, 2}$ 
along with the sets $(K_t)_{0 \le t \le \tau_{1, 2}}$. The timepoints $\tau_{k, 1}, \tau_{k, 2}$ and the sets $(K_t)_{\tau_{k-1, 2} \le t \le \tau_{k, 2}}$ for $k \ge 2$ 
will then be constructed inductively. We will verify properties~\ref{Kt1}--\ref{Kt3} 
alongside the construction.

As already hinted, we obtain the sets $(K_t)_{t \ge 0}$ as well as the timepoints 
$\tau_{k, i}$'s from a certain {\em breadth-first algorithm} exploring the components. 
To this end, we introduce yet another sequence of non-decreasing (random) time points $\{\tau_1^j: j \in \N\}$ which correspond to ``discrete rounds'' of the underlying 
algorithm. Let us start from $\tau_1^0 = 0$ by inspecting the set $\partial B(v, \frac 1N) \cap \widetilde \cG^{\ge h} = \{x_1', \ldots, x_{d'}'\}$ where $x_1', \ldots, x_{d'}'$ are arranged according to some arbitrary but deterministic ordering on $V$ 
(which we assume to be fixed from now on). Letting $x_1, \ldots, x_{d'}$ denote the 
(distinct) neighbors of $v$ such that $x_i' \in [v, x_i]_{\cG}$ (recall 
\eqref{def:intG}), consider the points
\begin{equation}\label{def:xi''}
x_i'' \stackrel{\rm def.}{=} \sup\{x \in [x_i', x_i]_{\cG} : 
[x_i', x]_{\cG} \subset \widetilde \cG^{\ge h} \setminus B(v, \tfrac1N)\},
\end{equation}
for $1 \le i \le d' (\le d)$ where we take the 
supremum to be $x_i'$ if this set is empty. Next we set the time point $\tau_1^1$ as $\sum_{i = 1}^{d'}\rho([x_i', x_i'']_{\cG}) = \tau_1^0 + \sum_{i = 1}^{d'}\rho([x_i', x_i'']_{\cG})$ (an empty summation is 0 by convention) 
and define the sets $K_t$ for $t \in (\tau_1^0, \tau_1^1]$ as follows. Let $d_t \in \{0, \ldots, d'-1\}$ denote the unique index such that $$\sum_{i=1}^{d_t - 1}\rho([x_{i}', x_{i}'']_{\cG}) < t \, (= t - \tau_1^0) \le \sum_{i=1}^{d_t}\rho([x_{i}', 
x_{i}'']_{\cG}).$$ 
Now let
\begin{equation}\label{def:explored_set}
K_t = K_0 \cup \bigcup_{1 \le i < d_t} \, [x_i', x_i'']_{\cG} \cup [x_{d_t}', x_{d_t}''']_{\cG},
\end{equation}
where $x_{d_t}''' \in [x_{d_t}', x_{d_t}'']_{\cG}$ is such that 
$$\rho([x_{d_t}', x]_{\cG}) =  t - \rho(\cup_{i=1}^{d_t - 1} [x_{i}', x_i'']_{\cG}) = t - \tau_1^0 - \rho(\cup_{i=1}^{d_t - 1} [x_{i}', x_i'']_{\cG}).$$ 
This concludes the first round in the construction of $(K_t)_{0 \le t \le \tau_{1, 2}}$. The following properties 
are direct consequences of this definition and the continuity of $\varphi$. 
\begin{equation}\label{eq:Kt1_ver1}
\text{\begin{minipage}{0.8\textwidth}(a) $(K_t)_{0 \le t \le \tau_1^1}$ is increasing in $t$ and the map $t \mapsto K_t: [0, \tau_1^1) \to \cK$ 
is continuous (cf.~property~\ref{Kt1}); (b) $\tau_1^1 = 0$ if $\partial B(v, \tfrac 1N) \cap \widetilde \cG^{\ge h} = \emptyset$ and for all $0 \le t \le \tau_1^1$, $K_t = \cup_{x \in K_0} \, \cC_{x, t}$ where each $(x \in ) \, \cC_{x, t} 
\in \cK$ is a connected subset of $\cC_{x}^{\ge h}(B^c(v, \tfrac 1N))$ 
if $x \in \widetilde \cG^{\ge h}$ and $\{x\}$ otherwise (cf.~properties~\ref{Kt2} and \ref{Kt3-3});  (c) $\partial K_{\tau_1^1} \setminus \partial (\cC_{K_0}^{\ge h}(B^c(v, \tfrac 1N)) \cup K_0) \subset V$; and, (d) $\partial K_{\tau_1^1} \setminus \widetilde \cG^{\le h} \subset V \cup K_0$ 
and $\bP_x[\widetilde X_{H_{K_{\tau_1^1}}} \in 
K_0 \setminus \widetilde \cG^{\le h}] = 0$ for all $x \in V \setminus \{v\}$ (cf.~property~\ref{Kt3-2}).\end{minipage}}
\end{equation}
Also for any $t \in [0, \infty)$, we have (cf.~property~\ref{Kt2})
\begin{equation}\label{eq:Kt2_ver1}
\text{$\{K_{t \wedge \tau_1^1} \subset U\} \in \cF_U$ for any open $U \subset \widetilde\cG$.}
\end{equation}
To see this, let us consider the timepoints $\tau_1^{1, U}$ and the sets $(K_t^U)_{0 
\le t \le \tau_1^{1, U}}$ which are defined analogously to $\tau_1^1$ and $(K_t)_{0 
\le t \le \tau_1^1}$ respectively but with 
$\widetilde \cG$ replaced by 
$U$ (see \eqref{def:xi''} above). Clearly, $\tau_1^{1, 
U}$ and $(K_{t \wedge \tau_1^{1, U}}^U)_{t \ge 0}$ are measurable relative to 
$\cF_U$. Furthermore since $U$ is open, for any $t \ge 0$,
\begin{equation}\label{eq:Kt2_ver2}
\text{$K_{t \wedge \tau_1^1} \subset U$ if and only if $K_{t\wedge \tau_1^{1, U}}^U \subset U$ in which case $K_{t \wedge \tau_1^1} = K_{t\wedge \tau_1^{1, U}}^U$.}
\end{equation}
Together these observations imply \eqref{eq:Kt2_ver1}.

Now consider the set $S_1^1 \stackrel{{\rm def.}}{=} \{x \in \partial K_{\tau_1^1} \cap V : \varphi(x) > h\}$. 
If $S_1^1 \ne \emptyset$, we continue almost in a similar way as before. 
More precisely, suppose that 
we have completed some round $j \ge 1$ with 
\begin{equation}\label{def:S1j}
S_1^j \stackrel{{\rm def.}}{=} \{x \in \partial K_{\tau_1^j} \cap V : \varphi(x) > h\} \ne \emptyset.
\end{equation}
Let $x_{S_1^j} \in V$ denote the minimum vertex in $S_1^j$. We construct the timepoint 
$\tau_1^{j+1}$ and the sets \sloppy$(K_t)_{\tau_1^j < t \le \tau_1^{j+1}}$ similarly 
as $\tau_1^1$ and $(K_t)_{\tau_1^0 < t \le \tau_1^1}$ above with $\{x_{S_1^j}\}$ and 
$\tau_1^j$ and $K_{\tau_1^j}$ in place of $\partial B(v, \frac 1N)$, $\tau_1^0 (= 0)$ 
and $K_0$ (in \eqref{def:explored_set}) respectively. If, on the other hand, $S_1^j = \emptyset$, we simply let $\tau_1^{j + 1} = \tau_1^j$. One can verify the 
properties~(a)--(d) in \eqref{eq:Kt1_ver1} as well as \eqref{eq:Kt2_ver1} (and 
\eqref{eq:Kt2_ver2}) for $\tau_1^{j+1}$ instead of $\tau_1^1$ using the same arguments 
inductively.

\smallskip

Since $\varphi$ is continuous on $\widetilde{\cG}$, it follows from this 
construction and \eqref{def:S1j} that if $S_1^j \ne \emptyset$ for some $j \ge 1$, 
then the points $x_{S_1^{\ell}} \in V; 1 \le \ell \le j$ are distinct. Consequently 
$S_1^J = \emptyset$ for some $J \ge 1$. In the rest of this paragraph we shall argue 
that 
\begin{equation}\label{eq:comp_reveal}
K_{\tau_1^J} = \cC_{K_0}^{\ge h}(B^c(v, \tfrac 1N)) \cup K_0 \mbox{ and }\partial K_{\tau_1^J} \setminus \widetilde \cG^{\le h} \subset  K_0 \mbox{ with probability 1 (under $\P$)}
\end{equation}
(cf.~property~\ref{Kt3-1}). The inclusion above is an immediate consequence of property~(d) in \eqref{eq:Kt1_ver1} and  \eqref{def:S1j}. For the equality of sets, let $\cC_x$ denote the component of $K_{\tau_1^J}$ containing $x \in K_0$ ($\subset K_{\tau_1^J}$). By property~(b) of $K_{\tau_1^J}$ in \eqref{eq:Kt1_ver1}, we have $K_{\tau_1^J} = \cup_{x \in K_0} \, \cC_x$. Therefore it suffices to show that, $\P$-a.s., $\cC_x = \cC_x'$ for all $x \in K_0$ where $\cC_x'$ is the component of 
$\cC_{K_0}^{\ge h}(B^c(v, \tfrac 1N)) \cup K_0$ containing $x$. So let us suppose that $\cC_x \ne \cC_x'$ for some $x \in K_0$ with positive 
probability. In view of property~(b) of $K_{\tau_1^J}$, we have $\cC_x \subset \cC_x'$. As $\cC_x$ and $\cC_x'$ are compact, connected subsets of $\widetilde \cG$ (recall that $\varphi$ is continuous on $\widetilde \cG$) and connected subsets of $\widetilde \cG$ are path-connected (see \eqref{eq:conn_set_G}), it follows from this inclusion (and $\cC_x \ne \cC_x'$) that there exists $y \in \partial \cC_x$ and a continuous path $\gamma_y$ starting at $y$ such that ${\rm image}(\gamma_y) \subset K_0 \cup \cC_x^{\ge h}(B^c(v, \tfrac 1N))$ whereas ${\rm image}(\gamma_y) \cap \cC_x = \{y\}$. Consequently if $y \notin V$, then $y \in \partial \cC_x \setminus \partial (\cC_{K_0}^{\ge h}(B^c(v, \tfrac 1N))$. But the latter set is empty owing to property~
(c) of $K_{\tau_J^1}$ and hence $y \in V$. However, since 
$\varphi(x) \ne h$ for any $x \in V$ almost surely (note that $\var[\varphi(x)] > 0$ in view of \eqref{eq:mgexp}, \eqref{eq:varMK} and \eqref{eq:bnd_lambda}), we obtain from the definition of $S_1^J$ in \eqref{def:S1j} and the inclusion in property~(d) of $K_{\tau_1^J}$ in \eqref{eq:Kt1_ver1} that 
$\partial K_{\tau_1^J} \cap V = \emptyset$ with probability 1. This implies that 
$\partial \cC_x \cap V = \emptyset$, and thus $y \notin V$, almost surely as $\partial \cC_x \subset \partial K_{\tau_1^J}$ which holds because the space $\widetilde \cG$ is locally (path-)connected. Therefore we have arrived at a contradiction which shows 
that our starting assumption $\cC_x \ne \cC_x'$ is false and completes the verification of \eqref{eq:comp_reveal}.

\smallskip

We now have a clear choice for $\tau_{1, 1}$, namely
$$\tau_{1, 1} \stackrel{{\rm def.}}{=} \tau_1^J.$$ If  
$(V \setminus \{v\}) \setminus K_{\tau_{1, 1}} \ne \emptyset$, there exists a minimum vertex $x \in (V \setminus \{v\}) \setminus K_{\tau_{1, 1}}$ (say) such that $K_{\tau_{1, 1}} \cap \overline{I_{\{x, x'\}}} \ne \emptyset$ for some $\{x, x'\} \in E$. Let 
$x_{\tau_{1, 1}}$ denote the point in $K_{\tau_{1, 1}} \cap \overline{I_{\{x, x'\}}}$ that is nearest to $x$ (so, in particular, $x_{\tau_{1,1}} \in \partial K_{\tau_{1, 1}}$). 
Also let $$\tau_{1, 2} \stackrel{{\rm def.}}{=} \tau_{1, 1} + 
\rho([x_{\tau_{1, 1}}, x]_{\cG}) \mbox{ and } K_t \stackrel{{\rm def.}}{=} 
K_{\tau_{1, 1}} \cup [x_{\tau_{1, 1}}, x'_t]_{\cG} \mbox{ for }\tau_{1, 1} \le t 
\le \tau_{1, 2},$$ 
where $x'_t \in [x_{\tau_{1, 1}}, x]_{\cG}$ is such that $\rho([x_{\tau_{1, 1}}, 
x'_t]_{\cG}) = t - \tau_{1, 1}$. We call the point $x$ as $x_{\tau_{1, 2}}$ ($\in V \cap K_{\tau_{1, 2}}$) in 
the sequel. If, on the other hand, $(V \setminus \{v\}) \setminus K_{\tau_{1, 1}} = \emptyset$, we simply let $\tau_{1, 2} = \tau_{1, 1}$. Since 
$\tau_{1, 1} = \tau_1^J$ satisfies \eqref{eq:Kt1_ver1} and \eqref{eq:comp_reveal}, the 
following properties follow directly from the above definitions (we interpret $\{x_{\tau_{1, 2}}\}$ as $\emptyset$ when $(V \setminus \{v\}) \setminus K_{\tau_{1, 1}} = \emptyset$). 
\begin{equation}\label{eq:Kt1_ver2}
\text{\begin{minipage}{0.8\textwidth}(a') $(K_t)_{0 \le t \le \tau_{1, 2}}$ is increasing in $t$ and the map $t \mapsto K_t: [0, \tau_{1, 2}) \to \cK$ 
is continuous; (b') for $0 \le t \le \tau_{1, 2}$, each component of $K_t$ intersects 
$K_0$; (c') for $\tau_{0, 2} ( = 0) < t \le \tau_{1, 1}$, $K_t$ has a component contained in $\cC_x^{\ge h}(B^c(v, \tfrac 1N))$ for some $x \in K_0$; 
(d') for $\tau_{1, 1} < t \le \tau_{1, 2}$, $K_t = K_{\tau_{1, 1}} \cup I_t$ for some interval $I_t$; and, (e') $\partial K_{\tau_{1, 2}} \setminus \widetilde \cG^{\le h} \subset \{x_{\tau_{1, 2}}\} \cup 
K_0$ a.s. whereas $\bP_x[\widetilde X_{H_{K_{\tau_{1, 2}}}} \in 
K_0 \setminus \widetilde \cG^{\le h}] = 0$ for all $x \in V \setminus \{v\}$.\end{minipage}}
\end{equation}
Restricting this construction so as to only ``look inside'' an open set $U \subset \widetilde \cG$, as in \eqref{eq:Kt2_ver2}, we also obtain property~\ref{eq:Kt2_ver1} 
(and \ref{eq:Kt2_ver2}) for $\tau_{1, 2}$. 

\smallskip

Now suppose that we have constructed timepoints $(\tau_{1, 1} < \tau_{1, 2} \le)\,\tau_{2, 1} \ldots \le \tau_{k, 1} < \tau_{k, 2}$ (notice the strict 
inequalities between $\tau_{\ell, 1}$ and $\tau_{\ell, 2}$'s) and the corresponding 
family of sets $(K_t)_{0 \le t \le \tau_{k, 2}}$ for some $k \ge 1$ such that --- in addition 
to the aforementioned properties of $(K_t)_{0 \le t \le \tau_{1, 2}}, \tau_{1, 1}$ and 
$\tau_{1, 2}$ --- one also has the following properties for {\em each} $1 < \ell \le 
k$. 
\eqref{eq:Kt2_ver1} holds with $\tau_{\ell, 2}$ in place of $\tau_1^1$, 
\eqref{eq:comp_reveal} holds with 
$\tau_{\ell - 1, 2}$ in place of $\tau_1^J$, and properties~(a'), (b') and (c')--(e') in \eqref{eq:Kt1_ver2} hold with $\tau_{\ell - 1, 2}$, $\tau_{\ell, 1}$ 
and $\tau_{\ell, 2}$ instead of $\tau_{0, 2}$, $\tau_{1, 1}$ and $\tau_{1, 2}$ 
respectively for some $x_{\tau_{\ell, 2}} \in V \cap K_{\tau_{\ell, 2}}$. Finally, 
\eqref{eq:Kt1_ver2}-(b') is replaced by: 
$$\mbox{$K_t = K_{\tau_{\ell, 2}} \cup \cC_t$ for some connected $C_t \subset \cC_{x_{\tau_{\ell, 2}}}^{\ge h}(B^c(v, \frac 1N))$ for 
all $\tau_{\ell, 1} < t \le \tau_{\ell, 2}$}.$$ 
If $(V \setminus \{v\}) \setminus 
K_{\tau_{k, 1}} = \emptyset$, we stop and let $K_t = K_{\tau_{k, 2}}$ for all $t \ge 
\tau_{k, 2}$ and $\tau_{m, 1} = \tau_{m, 2} = \tau_{k, 2}$ for all $m > k$. 
Properties~\ref{Kt1}--\ref{Kt3} in Proposition~\ref{prop:Kt} with $n = k$ are 
immediate consequences of 
our hypothesis.

\smallskip

Otherwise if $(V \setminus \{v\}) \setminus K_{\tau_{k, 1}} \ne \emptyset$, we can repeat the same sequence of steps with $\{x_{K_{\tau_{k, 2}}}\}, \tau_{k, 2}$ and 
$\{x_{K_{\tau_{k, 2}}}\}$ instead of $\partial B(v, \tfrac1N)$, $\tau_{0, 2} (= 0)$ 
and $K_0$ (in \eqref{def:explored_set}) respectively to obtain timepoints $0 \le \tilde \tau_{k+1, 1} \le \tilde \tau_{k+1, 2} < \infty$ and sets $(\tilde K_t)_{\tau_{k+1, 1} \le t \le \tilde 
\tau_{k+1, 2}}$. 
Thus letting $\tau_{k+1, i} = \tau_{k, 2} + \tilde \tau_{k+1, i}$ ($i = 1, 2$) and 
$K_t \stackrel{{\rm def.}}{=} K_{\tau_{k, 1}} \cup \tilde K_{t - \tau_{k, 2}}$ for 
$\tau_{k, 1} < t \le \tau_{k+1, 2}$, we see that 
the hypothesis in the previous paragraph is satisfied with $k+1$ instead of $k$. 
However, since $x_{K_{\tau_{k, 2}}} \in (K_{\tau_{k, 2}} \setminus K_{\tau_{k, 1}}) 
\cap V$ in this case, after some $K \le |V|$ iterations of this sequence we will meet the stopping 
condition $(V \setminus \{v\}) \setminus K_{\tau_{K, 1}} = \emptyset$ 
leading to our desired family of timepoints and the corresponding sets.
\end{proof}

Next we present a result which shows that $\cC^{\ge h}_v$ is not ``too far'' and can be 
easily reconstructed from $\cC^{\ge h}_{K_0}(B^c(v, \tfrac 1N))$ (see 
property~\ref{Kt3-1} above) on an event whose probability tends to 1 as $N \to \infty$. To this end, fix $v \in V, N \ge 1$ and $h \in \R$ and consider the following 
event:
\begin{equation}\label{def:ExhN}
        F_v = F_v(h,N) \stackrel{\rm def.}{=} \big\{\, \overline{B(v, \tfrac1N)} \subset \cG^{\ge h} \mbox{ if } \varphi(v) \ge h \mbox{ and } \overline{B(v, \tfrac1N)} \subset \cG^{< h} \mbox{ otherwise}\big\}.
\end{equation}
\begin{lemma}[Asymptotic irrelevance of $N$]\label{lem:rel_component}
On the event $F_v$ we have, with $K_0 = \partial B(v, \frac 1N)$,
    \begin{equation}\label{eq:rel_component1}
        \begin{split}
            \cC^{\ge h}_v = \cC^{\ge h}_{K_0}(B^c(v, \tfrac 1N)) \cup B(v, \tfrac 1N),
        \end{split}
    \end{equation}
if $\varphi(v) \ge h$ whereas    
    \begin{equation}\label{eq:rel_component}
        \begin{split}
             \cC^{\ge h}_{v} = \cC^{\ge h}_{K_0}(B^c(v, \tfrac 1N)) = \emptyset,
        \end{split}
    \end{equation}
    if $\varphi(v) < h$. Furthermore,
    \begin{equation}\label{eq:rel_component2}
        \begin{split}
            \P[F_v^c] \le \frac{1}{N^c}.
        \end{split}
    \end{equation}
\end{lemma}

\begin{proof}
\eqref{eq:rel_component} is immediate since $\overline{B(v, \tfrac 1N)} \subset \widetilde \cG^{< h}$ on the event $F_v \cap \{\varphi(v) < h\}$ 
(cf.~\eqref{def:ExhN}). So we give the argument for \eqref{eq:rel_component1}. In 
this case we have 
\begin{equation}\label{eq:Fvincl}
\overline{B(v, \tfrac 1N)} \subset \widetilde \cG^{\ge h}.    
\end{equation}
Let $y \in \cC^{\ge h}_v$ and we will argue that $y \in \cC^{\ge 
h}_{K_0}(B^c(v, \frac 1N)) \cup B(v, \frac 1N)$. This is clear if $y \in B(v, \frac 1N)$. Otherwise $y \in \cC^{\ge h}_v \setminus B(v, \frac 1N)$ 
and hence must be connected to $\partial B(v, \frac1N) (= K_0)$ in $B^c(x, \frac1N) \cap \cG^{\ge h}$. But in that case $y \in \cC^{\ge h}_{K_0}(B^c(v, \frac 1N))$. Thus we get $$\cC^{\ge h}_{v} \subset \cC^{\ge h}_{K_0}(B^c(v, \tfrac 1N)) \cup B(v, \tfrac 1N).$$ For the 
opposite inclusion, consider $y \in \cC^{\ge h}_{K_0}(B^c(v, \frac 1N)) \cup B(v, \frac 1N)$. If $y \in B(v, \tfrac 1N)$, we see that $y \in \cC_v^{\ge h}$ because of 
\eqref{eq:Fvincl}. Else $y \in \cC^{\ge h}_{K_0}(B^c(v, \tfrac 1N))$ and hence is connected to $K_0$ in $\widetilde \cG^{\ge h}$. Since $K_0 = \partial B(v, \frac 1N)$ 
is itself connected to $v$ in $\cG^{\ge h}$ due to \eqref{eq:Fvincl}, we therefore obtain
$$\cC^{\ge h}_{K_0}(B^c(v, \tfrac 1N)) \cup B(v, \tfrac 1N) \subset \cC^{\ge h}_v.$$
Combined with the previous inclusion \eqref{eq:rel_component1} follows.

\smallskip

    Next we prove \eqref{eq:rel_component2}. Letting $\varepsilon > 0$ (choice will be specified later), we introduce the following events.
    \begin{equation*}
        \begin{split}
            D_{v,1} = \big\{ \varphi(v) \ge h + \varepsilon,\,\, &\overline{B(v, \tfrac 1N)} \cap \widetilde \cG^{ < h} \ne \emptyset \big\},\,\, D_{v,2} = \big\{\varphi(v) \le h - \varepsilon,\,\, \overline{B(v, \tfrac 1N)} \cap \widetilde \cG^{\ge h} \ne \emptyset \big\} \mbox{ and} \\ &D_{v,3} = \big\{h - \varepsilon < \varphi(v) < h + \varepsilon \big\}.
              \end{split}
    \end{equation*}
    Clearly $F^c_v \subset D_{v,1} \cup D_{v,2} \cup D_{v,3}$ (see \eqref{def:ExhN}) and hence
    \begin{equation}\label{eq:rel_component3}
        \begin{split}
            \P[  F^c_v ] \le \P[D_{v,1} \cup D_{v,2}] + \P[D_{v,3}].
        \end{split}
    \end{equation}
    We will bound each of these three probabilities separately starting with 
    $\P[D_{v,3}]$. Since $\varphi(v) \sim N(0, \sigma_v^2)$ with $\sigma_v \ge c$ 
    (follows from \eqref{eq:mgexp}, \eqref{eq:varMK} and \eqref{eq:bnd_lambda}), we have
    \begin{equation}\label{eq:rel_component4}
        \begin{split}
            \P[D_{v,3}] = \P[ h - \varepsilon < \varphi(v) < h + \varepsilon ] \le C \varepsilon.
        \end{split}
    \end{equation}
    Next we bound $\P[D_{v,1}]$ and $\P[D_{v,2}]$. Notice that, in view of the definitions of $D_{v, 1}$ and $D_{v, 2}$, 
    \begin{equation}\label{eq:bndDx2}
    \begin{split}
    &\P[D_{v,1} \cup D_{v, 2}] \le \mbox{$\P\big[\sup_{x \in \overline{B(v, \frac 1N)}} | \varphi(v) -  \varphi(x)| > \varepsilon/2\big]$} \\
    &\le d \sup_{e \sim v}\mbox{$\P\big[\sup_{x \in \overline{I_e},\, d(v, x) \le \frac1N} | \varphi(v) -  \varphi(x)| > \varepsilon/2\big]$}
    \end{split}
    \end{equation}
    where $e \sim v$ if $v$ is an endpoint of $e$. Since the process $(\varphi(x))_{x \in \overline{I_e}}$, where $e = (v, y)$, is distributed as a standard Brownian 
    bridge between $\varphi(v)$ and $\varphi(y)$ conditionally on $(\varphi(v), \varphi(y))$ (see below display \eqref{eq:Gff_sum} in Section~\ref{sec:prelim}), we can write
    \begin{equation*}
    \mbox{$\P\big[\sup_{x \in \overline{I_e},\, d(v, x) \le \frac1N} | \varphi(v) -  \varphi(x)| > \varepsilon/2\big]$} \le \mbox{$\P\big[ |\varphi(v) -  \varphi(y)| > N\varepsilon/4 \big] + \PBr[\sup_{t \in [0, \frac 1N]} |B_t| > \varepsilon/4]$},
    \end{equation*}
    where $B_t$ is a standard Brownian bridge of length $1$ under $\PBr$. The first probability is bounded by $e^{-c N^2\varepsilon^2}$ as $\varphi(v) -  \varphi(y) \sim N(0, \sigma_{vy}^2)$ for some $\sigma_{vy} \le C$ whereas the second probability is 
    bounded by $Ce^{-c N\varepsilon^2}$ which we obtain using standard results on Brownian 
    bridges (see, e.g., \cite[Chapter IV.26]{BorodinBrown}). Plugging these two bounds 
    into the previous display and the resulting bound into \eqref{eq:bndDx2} 
    we get
    \begin{equation*}
        \P[D_{v, 1} \cup D_{v, 2}] \le C e^{-CN\varepsilon^2}.
    \end{equation*}
    Together with \eqref{eq:rel_component4} this yields, in view of 
    \eqref{eq:rel_component3},
    \begin{equation*}
            \P[F^c_v] \le C( \exp(- c N \varepsilon^2) + \varepsilon ). 
    \end{equation*}
Now setting $\varepsilon = N^{-1/4}$ we obtain \eqref{eq:rel_component}.
\end{proof}

\subsection{Exploration martingale}\label{subsec:martingale}
For any $K \in \cK_{<\infty}$, recall the random variable $M_K = \widetilde\cE(\varphi, f_K)$ from Section~\ref{sec:dirchlet}. Now given any 
$v \in V$, $h \in \R$ and integer $N > 1$, let us consider the family of sets $(K_t)_{t \ge 0} = (K_t(v, h, N))_{t \ge 0}$ given by 
Proposition~\ref{prop:Kt}. Since $K_t \in \cK_{<\infty}$ by property~\ref{Kt2} of $(K_t)_{t \ge 0}$, we may define
\begin{equation}\label{def:tildeMt}
\tilde M_t = \tilde M_t(v, h, N) \stackrel{{\rm def.}}{=} M_{K_t}
\end{equation}
for all $t \ge 0$. In the following lemma, we record some important properties of the 
process $(\tilde M_t)_{t \ge 0}$.
\begin{lemma}[Quadratic variation of $(\tilde M_t)_t$]\label{lem:Mt_mart}    
The process $(\tilde M_t)_{t \ge 0}$ is a continuous $(\cF_{K_t})$-martingale with 
quadratic variation given by 
\begin{equation}\label{eq:qua_var_M_t}
\langle \tilde M \rangle_t = 2|E| (\nu_{K_t} - \nu_{K_0}), \quad t \ge 0.
\end{equation}
\end{lemma}
\begin{proof}
The continuity of $\tilde M_t = M_{K_t}$ as a function of $t$ follows from property~\ref{Kt1} of $K_t$'s and Lemma~\ref{lem:mg_cont}. Owing to property~\ref{Kt2}, $K_t$ satisfies 
property~\eqref{state:smp} in view of \eqref{eq:smp_condition} for 
all $t \ge 0$. Also each $K_t$ is a subset of the compact set $K_\infty \stackrel{{\rm def.}}{=} \widetilde \cG \setminus B(v, \frac 1N) \in \cK_{<\infty}$. Consequently, by 
Proposition~\ref{prop:MK_martingale},
$$\tilde M_t = \E[M_{K_\infty} \mid \cF_{K_t}]$$
and thus $(\tilde M_t)_{t \ge 0}$ is a (uniformly integrable) $(\cF_{K_t})$-martingale. Finally, since $\Var[ M_{K_\infty}\mid\cF_{K_t}] = 2|E|(\nu_{K_\infty} - 
\nu_{K_t})$ by Lemma~\ref{lem:cond_variance} which is $\P$-almost surely continuous as 
a function of $t$ due to Lemma~\ref{lem:lambda_cont}, it follows from 
\cite[Corollary~10]{MR4112719} that the quadratic variation of $(\tilde M_t)_{t \ge 
0}$ is given by the expression
\begin{equation*}
\langle \tilde M \rangle_t = \Var[M_{K_\infty}\mid\cF_{K_0}] - \Var[ M_{K_\infty}\mid\cF_{K_t}] = 2|E|(\nu_{K_t} - \nu_{K_0})
\end{equation*}
for all $t \ge 0$.
\end{proof} 
The following result follows from application of the classical Dubins–Schwarz theorem 
(see, e.g., \cite[Chapter V, Theorem~1.7]{revuz2013continuous} for the 
particular version used here and also \cite[Chapter IV, Proposition~1.13]{revuz2013continuous}) to the continuous martingale $(\tilde M_t)_{t \ge 0}$. In 
Section~\ref{sec:critical_regime}, this will enable us to calculate {\em tail 
probabilities} for the clusters of $\cG^{\ge h}$ after being translated into 
suitable events measurable relative to $\tilde M_t$.
\begin{theorem}\label{thm:time-change}
For any $t \ge 0$, let $T_t \stackrel{{\rm def.}}{=} \inf\{s \ge 0 : \langle \tilde M 
\rangle_s > t\}$. Then the time-changed process $(\tilde M_{T_{t}})_{t \ge 0}$ is 
distributed as $(B_{t_0 + t \wedge \langle \tilde M\rangle_{\infty}})_{t \ge 0}$ where 
$(B_t)_{t \ge 0}$ is a standard Brownian motion and $t_0 = 2|E|\nu_{K_0}$. 
Furthermore, $\tilde M_s$ is constant on each interval $[T_{t-}, T_t]$ where $T_{t-} \stackrel{{\rm def.}}{=} \lim_{u \uparrow t} T_u$ for $t \ge 0$.
\end{theorem}
\section{Tail probabilities for cluster volume}\label{sec:critical_regime}
This section is devoted to the proof our main results, namely, 
Theorems~\ref{thm:main1} and \ref{thm:main2}. The corresponding upper tail 
probabilities derive from the following result.
\begin{theorem}[Upper bounds on cluster volume]\label{prop:ubd_cri_vol}
For any $\delta \in (0,1)$ and $h \stackrel{}{=} -A |V|^{-1/3}$ with $
A \ge 0$, we have
    \begin{equation}\label{eq:ubd_cri_vol}
            \P \big[ |\scC^{\ge h}_{\max}| \ge \tfrac{1}{\delta} |V|^{2/3} \big] \le  C (\delta^{3/2} +   \delta A). 
    \end{equation}
On the other hand, for $h = A |V|^{-1/3}$ and $\delta |V|^{2/3}\ge C$ we have
\begin{equation}\label{eq:ubd_cri_vol2}
            \P\big[ |\scC^{\ge h}_{\max}  | \ge \delta |V|^{2/3} \big] \le  \frac{C}{\delta^{3/2}}e^{-c A^2\delta}. 
    \end{equation}
\end{theorem}
For the lower tail probabilities we need:
\begin{theorem}[Lpper bounds on cluster volume]\label{prop:lbd_cri_vol}
For $h \stackrel{}{=} A |V|^{-1/3}$ with $A \ge 0$ and any $\delta \in (0, 1)$ 
such that $|V| \ge \frac{(1 + A)^8}{\delta^{3}}$, we have
     \begin{equation}\label{eq:lbd_cri_vol}
             \P[|\scC^{\ge h}_{\max}| \le \delta |V|^{2/3}] \le 
             C(1 + A)^{1/2}\delta^{1/5}.
     \end{equation}
On the other hand, there exists $c \in (0, \infty)$ such that for $h = -A|V|^{-1/3}$ and 
any $\delta \in (0, 1)$ satisfying $\tfrac1{\delta} \le c \min(|V|^{1/3}, A)$, we have
\begin{equation}\label{eq:lbd_cri_vol2}
             \P[|\scC^{\ge h}_{\max}| \le \tfrac1{\delta} |V|^{2/3}] \le C ({\tilde \delta}^{1/5} + e^{-c A^2 \tilde \delta})
     \end{equation}
   for all $|V|^{-1/3} \le \tilde \delta \le 1$.
 \end{theorem}
We can now deduce our main results assuming these bounds.
\begin{proof}[Proofs of Theorems~\ref{thm:main1} and \ref{thm:main2}]
The bound \eqref{eq:main_critbnd} in Theorem~\ref{thm:main1} follows by combining 
\eqref{eq:ubd_cri_vol} in Theorem~\ref{prop:ubd_cri_vol} with \eqref{eq:lbd_cri_vol2} in Theorem~\ref{prop:lbd_cri_vol} above. For the bound \eqref{eq:main_subcritbnd} in Theorem~\ref{thm:main2}, we use \eqref{eq:ubd_cri_vol2} in 
Theorem~\ref{prop:ubd_cri_vol} with $\delta = C \, \frac {\log eA}{A^2}$ for a suitable 
$C \in (0, \infty)$. For \eqref{eq:main_supcritbnd}, we apply 
\eqref{eq:lbd_cri_vol2} in Theorem~\ref{prop:lbd_cri_vol} with the choices $\delta = c 
A$ and $\tilde \delta$ for a suitable $c \in (0, \infty)$.
\end{proof}
In \S\ref{subsec:ubd_critical} and \S\ref{subsec:lbd_critical} below, we give the 
proofs of Theorems~\ref{prop:ubd_cri_vol} and \ref{prop:lbd_cri_vol} respectively. In 
the sequel, for any two events $E_1$ and $E_2$ defined on the probability space 
underlying $\P$, we say that $E_1 \subset E_2$ {\em on an event with} ($\P$-) {\em 
probability 1} if $E_1 \cap E \subset E_2 \cap E$ for some event $E$ such that $\P[E] 
= 1$.
\subsection{Upper tail probabilities for cluster volume}\label{subsec:ubd_critical}
We prove Theorem~\ref{prop:ubd_cri_vol} in this subsection. The following bound on the 
size of a generic cluster is crucial (see,~e.g., \cite{MR2653185} for similar upper 
bounds in the context of Erd\H{o}s-R\'enyi graphs).
\begin{lemma}\label{lem:max_cluster_size}
For $T \ge 1$ and $h = - A|V|^{-1/3}$ with $A \ge 0$, one has
    \begin{equation}\label{eq:max_cluster_size}
             \max_{x \in V} \P [ | \scC^{\ge h}_x | \ge T ] \le  C\big( \tfrac{1}{\sqrt{T}} + A |V|^{-1/3}\big),
    \end{equation}
whereas for $h = A|V|^{1/3}$ and $T \ge C$, we have
\begin{equation}\label{eq:max_cluster_size2}
             \max_{x \in V} \P [ | \scC^{\ge h}_x | \ge T ] \le 
             \tfrac{C}{\sqrt{T}} e^{-c h^2T}.
    \end{equation}
\end{lemma}
Before we prove Lemma~\ref{lem:max_cluster_size}, let us deduce 
Theorem~\ref{prop:ubd_cri_vol} using it.
\begin{proof}[Proof of Theorem~\ref{prop:ubd_cri_vol}]
    For $T \ge 1$, consider the (random) set $N^{\ge h}_T \stackrel{\rm def.}{=} \{ 
    x \in  V : |\scC_x^{\ge h}| \ge T\}$. By definition of $N^{\ge h}_T$ and the Markov inequality, we have
    \begin{equation}\label{eq:ubd_cls_size_critical_1}
            \P [ |\scC^{\ge h}_{\max}  | \ge T ] = \P [ |N^{\ge h}_T| \ge T ] \le \frac{\E[N^{\ge h}_T]}{T}. 
    \end{equation}
    When $h = -A|V|^{-1/3}$, the expectation on the right-hand side can be bounded as:
    \begin{equation*}
     \frac{\E[N^{\ge h}_T]}{T} = \frac{1}{T} \sum_{x \in V}  \P [ | \scC^{\ge h}_x | \ge T ] \le  \frac{|V|}{T} \max_{x \in V} \P [ | \scC^{\ge h}_x | \ge T ] \stackrel{\eqref{eq:max_cluster_size}}{\le} C\Big(\frac{|V|}{T^{3/2}} + \frac{A |V|^{2/3}}{T}\Big). 
    \end{equation*}
    Now setting $T = \frac{1}{\delta} |V|^{2/3}$, we can deduce \eqref{eq:ubd_cri_vol} from these two bounds.

\smallskip

    On the other hand, when $h = A|V|^{-1/3}$, we can recompute \eqref{eq:ubd_cls_size_critical_1} as follows:
    \begin{equation*}
     \frac{\E[N^{\ge h}_T]}{T} \le \frac{|V|}{T} \max_{x \in V} \P [ | \scC^{\ge h}_x | \ge T ] \stackrel{\eqref{eq:max_cluster_size2}}{\le} C\frac{|V|}{T^{3/2}} e^{-\frac{ch^2T}2}
    \end{equation*}
    (provided $T \ge C$). Now plugging $T = \delta |V|^{2/3}$ we obtain 
    \eqref{eq:ubd_cri_vol2}.
 \end{proof}
 Our next result is an important step in the pending proof of 
 Lemma~\ref{lem:max_cluster_size}.
\begin{lemma}[Large cluster to positivity of Brownian motion]\label{lem:cls_to_brw}
 There exist $\Cl{C:br}, \Cl[c]{c:br} \in (0, \infty)$ such that for any $x \in V$, $h \in \R$ and $T \ge 2$ we have 
    \begin{equation}\label{eq:cls_to_brw0}
             \P[ |\scC^{\ge h}_x | \ge T ] \le \PB_0 \big[\inf_{\Cr{C:br} \le t \le \Cr{c:br}T} (B_t - h t) \ge 0 \big],  
    \end{equation}
    where 
    $(B_t)_{t \ge 0}$ is a standard Brownian motion under $\PB_0$. 
\end{lemma}
\begin{proof}
   For $x \in V$ and $N \ge 1$ an integer, let $K_0 = K_0(x, N) = \partial B(x, \tfrac 1N)$ (cf.~Proposition~\ref{prop:Kt}). Recall the 
   event $F_x = F_x(h, N)$ from \eqref{def:ExhN}. It follows from \eqref{eq:rel_component} in Lemma~\ref{lem:rel_component} that
   \begin{equation*}
          \{ |\scC^{\ge h}_x| \ge T \} \subset  \big\{\big|\cC^{\ge h}_{K_0}( B^c(x, \tfrac 1N))\big|_V \ge T - 1 
          \big\} \mbox{ on the event $F_x$}\end{equation*}
          where, as in Section~\ref{sec:zeravgcap}, $|S|_V = |S \cap V|$ (note that $\{|\scC^{\ge h}_x| \ge T\} \subset \{\varphi(x) 
   \ge h\}$ as $T > 0$). 
   Hence for $T \ge 2$, we have
   \begin{equation*}
          \{|\scC^{\ge h}_x| \ge T \} \subset  \big\{\big|\cC^{\ge h}_{K_0}( B^c(x, \tfrac 1N))\big|_V \ge cT \big\} \mbox{ on the event $F_x$}.
   \end{equation*}
   Now consider the family of sets $(K_t)_{t \ge 0}$ given by 
   Proposition~\ref{prop:Kt} with $v = x$, $h$ as in the statement of the lemma and $N$ as above. From properties~\ref{Kt1} and \ref{Kt3-1} (the first part) of $(K_t)_{t \ge 0}$, 
   we get the inclusion
   \begin{equation*}
   \begin{split}
           & \, \big\{\big|\cC^{\ge h}_{K_0}( B^c(x, \tfrac 1N))\big|_V \ge cT \big\} \subset \big\{ |K_{\tau_{1, 1}}|_V \ge cT \big\} \\
           \subset& \, \big\{\exists \, T'\ge 0 \text{ such that }K_t \subset \cG^{\ge h} \text{ for } t \in [0, T'] \text{ with } |K_{T'}|_V \ge cT \big\}
   \end{split}        
   \end{equation*}
   on an event with probability 1 (this will be henceforth assumed implicitly for all inclusions in the proof). Next consider the continuous $(\mathcal F_{K_t})$-martingale $(M_{K_t})_{t \ge 0} = (\tilde M_t)_{t \ge 0}$ attached 
   to the random sets $(K_t)_{t \ge 0}$ (see \eqref{def:tildeMt} and Lemma~\ref{lem:Mt_mart}). Using Lemma~\ref{cor:value_bdr}, we can write
   \begin{equation*}
   \begin{split}
            &\big\{\exists \, T'\ge 0 \text{ such that }K_t \subset \cG^{\ge h} \text{ for } t \in [0, T'] \text{ with } |K_{T'}|_V \ge cT \big\}\\ \subset \, &\big\{\exists \, T' \ge 0 \text{ such that }\tilde M_t \ge 2h |E|\nu_{K_t}  \text{ for } t \in [0, T'] \text{ with } |K_{T'}|_V \ge cT\big\}
     \end{split}       
   \end{equation*}
   Recalling the definition of $T_t$ from Theorem~\ref{thm:time-change}, we further have
   \begin{equation*}
   \begin{split}
       &\big\{\exists \, T'\ge 0 \text{ such that }\tilde M_t \ge 2h |E|\nu_{K_t}  \text{ for } t \in [0, T'] \big\}\\ 
       \subset\, &\big\{\exists \, T'\ge 0 \text{ such that }\tilde M_{T_r} \ge 2h |E|\nu_{K_{T_r}}  \text{ for } r \in [0, 
       \langle \tilde M\rangle_{T'}]\big\}.
       \end{split}
   \end{equation*} 
   Combining the chain of inclusions comprising the last four displays, we obtain
   \begin{equation}\label{eq:cls_to_brw}
          \{|\cC^{\ge h}_x| \ge T \} \subset \big\{\exists \, T'\ge 0 \text{ such that }\tilde M_{T_r} \ge 2h |E|\nu_{K_{T_r}}  \text{ for } r\in [0, 
          \langle \tilde M \rangle_{T'}] \text{ and }|K_{T'}|_V \ge cT\big\}
\end{equation}
on the event $F_x$.
   
   Theorem~\ref{thm:time-change} 
   also tells us that the process $(\tilde M_{T_r})_{r \ge 0}$ has the same law as the 
   stopped (standard) Brownian motion $(B_{r_0 + r \wedge \langle \tilde M\rangle_{\infty}})_{r \ge 0}$ with 
   $r_0 = 2|E|\nu_{K_0}$. 
   Since $\langle \tilde M\rangle_{r} = 2|E|\nu_{K_r} - r_0 = r \wedge \langle 
   \tilde M \rangle_{\infty} - r_0$ by Lemma~\ref{lem:Mt_mart} and the continuity of quadratic variation, together with \eqref{eq:cls_to_brw} and the lower bound on 
   $2|E| \nu_K$ from Proposition~\ref{lem:bnd_lambda} this yields:
            \begin{equation*}
             \P\big[ \big|\cC^{\ge h}_{K_0}( B^c(x, \tfrac 1N))\big|_V \ge cT \big] \le \PB_0 \big[\inf_{r_0 \le r \le cT} (B_r - h r) \ge 0 \big] + \P[F_x^c].  
    \end{equation*}
From this and the bound \eqref{eq:rel_component2} in Lemma~\ref{lem:rel_component}, 
we can deduce \eqref{eq:cls_to_brw0} by sending $N \to \infty$.
\end{proof}

\begin{proof}[Proof of Lemma~\ref{lem:max_cluster_size}]
We first give the proof of \eqref{eq:max_cluster_size}. We can assume, without any loss of generality, that $T \ge C$ and $|h| = A |V|^{-1/3} \le 1$. In view of Lemma~\ref{lem:cls_to_brw}, it suffices to prove the upper bound for $\PB_0 
\big[\inf_{\Cr{C:br} \le t \le \Cr{c:br}T} (B_t - h t) \ge 0 \big]$. Notice that
\begin{equation*}
\PB_0\big[\inf_{\Cr{C:br} \le t \le \Cr{c:br}T} (B_t - h t) \ge 0 \big] \le \overline{\mathbf P}_0^{{\rm B}}\big[\inf_{0 \le t \le \Cl[c]{c:br2}T} (B_t - h t) \ge  X\big],
\end{equation*}
for some $\Cr{c:br2} > 0$ where $X \sim N(\Cr{C:br}h, \Cr{C:br})$ is independent of the standard Brownian motion $(B_t)_{t \ge 0}$ under 
$\overline{\mathbf P}_0^{{\rm B}}$. Using 
existing results (see, e.g., \cite[Equation~2.0.2 in Part~II]{BorodinBrown}), we can 
write,
\begin{equation*}
\overline{\mathbf P}_0^{{\rm B}}\big[\inf_{0 \le t \le \Cr{c:br2}T} (B_t - h t) \ge  X\big] = \int_{{\rm x} > 0} \Phi \Big( \frac{{\rm x} - \Cr{c:br2}hT}{\sqrt{\Cr{c:br2}T}}\Big)\phi_h({\rm x}) \, d{\rm x} - \int_{{\rm x} > 0} e^{2{\rm x}h} \, \overline{\Phi} \Big( \frac{{\rm x} + \Cr{c:br2}hT}{\sqrt{\Cr{c:br2}T}} \Big) \phi_h({\rm x})  \,d{\rm x},\end{equation*}
where $\phi_h({\rm x}) = \frac1{\sqrt{2\pi\Cr{C:br}}}\exp(-\frac{({\rm x} + \Cr{C:br}h)^2}{2\Cr{C:br}})$ is the density function of $N(-\Cr{C:br}h, \Cr{C:br})$. Since 
$\overline{\Phi}(b) = \Phi(-b)$ and $e^{2ah} \ge 1 + 2ah$, we get the bound
    \begin{equation}\label{eq:cls_to_brw_3}
        \begin{split}
            &\overline{\mathbf P}_0^{{\rm B}}\big[\inf_{0 \le t \le \Cr{c:br2}T} (B_t - h t) \ge  X\big] \\ 
            \le& \int_{{\rm x} >0} \Big( \Phi \Big( \frac{{\rm x} - \Cr{c:br2}hT}{\sqrt{\Cr{c:br2}T}}\Big) - \Phi \Big( \frac{-{\rm x} - \Cr{c:br2}hT}{\sqrt{\Cr{c:br2}T}}\Big)\Big) \phi_h({\rm x})  \,d{\rm x}   - \int_{{\rm x} > 0} 2{\rm x}h \phi_h({\rm x}) \,d{\rm x} = I_1 + I_2.
        \end{split}
    \end{equation}
    Using the fact that $|\Phi({\rm x}) - \Phi({\rm y})| \le |{\rm x} - {\rm y}|$ for all ${\rm x}, {\rm y} \in \R$, we obtain (recall that $h < 0$ for the final step)
    \begin{equation}\label{eq:cls_to_brw_4}
        \begin{split}
            I_1 &\le \frac{C}{\sqrt{T}} \int_{{\rm x} > 0} {\rm x} \phi_h({\rm x}) \,d{\rm x} = \frac{C}{\sqrt{T}} \int_{{\rm x} >0} ({\rm x} + \Cr{C:br}h) \phi_h({\rm x}) \,d{\rm x} - \frac{C}{\sqrt{T}} \Cr{C:br}h \int_{{\rm x} >0} \phi_h({\rm x}) \,d{\rm x} \\
            & \le \frac{C}{\sqrt{T}} (e^{-c h^2} - Ch).
        \end{split}
    \end{equation}
    On the other hand, by same computations,
    \begin{equation*}
            I_2 
            \le 
            -Che^{-ch^2} + Ch^2. 
    \end{equation*}
    Since $-h = A |V|^{-1/3} \le 1$ by our assumption, plugging these into \eqref{eq:cls_to_brw_3} we get
    \begin{equation*}
        \overline{\mathbf P}_0^{{\rm B}}\big[\inf_{0 \le t \le \Cr{c:br2}T} (B_t - h t) \ge  X\big] \le C(\tfrac 1{\sqrt{T}} + A |V|^{-1/3})
    \end{equation*}
    which concludes the proof of \eqref{eq:max_cluster_size}.

    We now give the proof of \eqref{eq:max_cluster_size2}. We proceed similarly until \eqref{eq:cls_to_brw_3} (recall that $T \ge C$) and note that $I_2 \le 0$ as $h > 
    0$. So we only need to bound $I_1$ which we need to do more carefully. To this 
    end, let us decompose
    \begin{equation*}
        I_1 = \int_{\tfrac{\Cr{c:br2}}2hT \, \ge \, x > 0} \, + \, \int_{x > \tfrac{\Cr{c:br2}}2hT}
    \end{equation*}
    where 
    the two terms correspond to the integral in $I_1$ (see \eqref{eq:cls_to_brw_3}) 
    restricted to the respective ranges. To bound the first term, we use the improved 
    estimate $|\Phi({\rm x}) - \Phi({\rm y})| \le e^{-c h^2T}|{\rm x} - {\rm y}|$ for 
    ${\rm x}, {\rm y} \le -\tfrac{\Cr{c:br2}}2hT$. This gives us 
    (cf.~\eqref{eq:cls_to_brw_4})
    \begin{equation*}
            \int_{\tfrac{\Cr{c:br2}}2hT \, \ge \, x > 0} \le \frac{Ce^{-c h^2T}}{\sqrt{T}} \int_{{\rm x} > 0} {\rm x} \phi_h({\rm x}) \,d{\rm x}  \le \frac{Ce^{-c h^2T}}{\sqrt{T}}.
    \end{equation*}
    For the other part, we bound as in \eqref{eq:cls_to_brw_4} and obtain a similar 
    bound, i.e.,
    \begin{equation*}
            \int_{x > \tfrac{\Cr{c:br2}}2hT} \le \frac{C}{\sqrt{T}} \int_{{\rm x} > \tfrac{\Cr{c:br2}}2hT} {\rm x} \phi_h({\rm x}) \,d{\rm x}  \le \frac{Ce^{-c h^2T}}{\sqrt{T}}.
    \end{equation*}
    Adding these two bounds we get \eqref{eq:max_cluster_size2}.
\end{proof}
\subsection{Lower tail probabilities for cluster volume}\label{subsec:lbd_critical}
In this subsection we give the proof of Theorem~\ref{prop:lbd_cri_vol}. We 
will only give the proof of \eqref{eq:lbd_cri_vol} as the proof of \eqref{eq:lbd_cri_vol2} follows by adapting parts of the argument in a 
relatively straightforward manner. Let $v \in V$ be arbitrary (but fixed) and consider the family of sets $(K_t)_{t \ge 0}$ obtained from 
Proposition~\ref{prop:Kt} with $v$, $h$ as above \eqref{eq:lbd_cri_vol} and some integer $N > 1$. In the sequel $v$, $h$, $A$ and 
$(K_t)_{t \ge 0}$ always refer to these particular choices.

Let us recall the sequence of (possibly random) timepoints $0 \le \tau_{1, 1} \le 
\tau_{1, 2} \le \ldots \le \tau_{n, 1} \le \tau_{n, 2}< \infty$ from property~\ref{Kt3} of $(K_t)_{t \ge 0}$ in Proposition~\ref{prop:Kt}. We augment this sequence by letting $\tau_{0, 1} = \tau_{0, 2} = 0$. The following two 
events play crucial roles in the proof of 
\eqref{eq:lbd_cri_vol}. Attached to the family 
$(K_t)_{t \ge 0}$ is the continuous $(\cF_{K_t})$-martingale $(M_{K_t})_{t \ge 0} = (\tilde M_t)_{t \ge 0}$ (\eqref{def:tildeMt} and 
Lemma~\ref{lem:Mt_mart}). Throughout this section, $\delta', \delta'' \in (0, 1)$ and $\beta, b \in (0, \infty)$ represent generic parameters 
for 
our events whereas their dependence on $N$ is kept implicit. Let us start with
\begin{equation}\label{def:Fdstar}
    \Cl[tilF]{Flbc:1}(\delta', \beta, b) 
    \stackrel{\rm def.}{=} \left\{
		\begin{array}{c}
			\mbox{$\tilde M_t \le ({\delta'}  + 
    \beta A {\delta'}^{1/2}) |V|^{1/3}$ for any 
    $\tau_{k, 1} \le t \le \tau_{k, 2}$} \\
			\mbox{and $0 \le k \le n$ such that $|K_{\tau_{k, 1}}|_V \le b {\delta'}^{1/2}|V|^{2/3}$}
		\end{array}
		\right\}.
\end{equation}
Next let us consider
\begin{equation}\label{def:F2}
\Cl[tilF]{Flbc:2}(\delta', \delta'', \beta, b) 
\stackrel{\rm def.}{=} \left\{\begin{array}{c}\mbox{$\exists \, 0 \le t_1 \le t_2$ s.t. $2|E|(\nu_{K_{t_2}} - \nu_{K_{t_1}}) \ge \beta \delta'' |V|^{2/3}$, $|K_{t_2}|_V \le$}\\
\mbox{$b {\delta'}^{1/2}|V|^{2/3}$ and $\tilde M_t > ({\delta'}  + \beta A {\delta'}^{1/2}) |V|^{1/3}$ for all $t_1 \le t \le t_2$}
\end{array}
\right\}.
\end{equation}
\begin{lemma}\label{lem:inclusion_lbd}
There exists $\Cl{C:beta} \in (0, \infty)$ such that for any $\delta, \delta' \in (0, 1)$, 
$\beta \ge \Cr{C:beta}$, $b > 0$ and $|V| \ge C((1 / \delta)^{3/2} \vee b^3)$, 
we have
  \begin{equation}\label{eq:inclusion_lbd}
          \Cr{Flbc:1}(\delta', \beta, b) \cap \Cr{Flbc:2}(\delta', \delta, \beta, b) \subset \{ |\scC^{\ge h}_{\max}| \ge \delta |V|^{2/3} \}
  \end{equation} 
  on an event with probability 1.
\end{lemma}

\begin{proof}
Since $K_t$'s are increasing by property~\ref{Kt1} and $K_t = K_{\tau_{n, 2}}$ for 
all $t \ge \tau_{n, 2}$ by property~\ref{Kt3-1} (in Proposition~\ref{prop:Kt}), it 
follows from the definitions in \eqref{def:Fdstar} and \eqref{def:F2} that
\begin{equation}\label{eq:F12_inclusion}
\Cr{Flbc:1}(\delta', \beta, b) \cap \Cr{Flbc:2}(\delta', \delta, \beta, b) \subset \left\{\begin{array}{c}\mbox{$\exists \, 0 \le k < n$ and $\tau_{k, 2} < t < \tau_{k+1, 1}$ s.t. $|K_t|_V \le$}\\ \mbox{$b {\delta'}^{1/2}|V|^{2/3}$ and $2|E|(\nu_{K_t} - \nu_{K_{\tau_{k, 2}}}) \ge \beta \delta |V|^{2/3}$}
\end{array}
\right\}.
\end{equation}
on an event with $\P$-probability 1 (we will assume this implicitly for all relevant inclusions in the remainder of the section). Now from property~\ref{Kt3-3} and (the first part of) property~\ref{Kt2}, we know that, a.s., $K_t = K_{\tau_{k, 2}} \cup \cC_t$ for any $t \in (\tau_{k, 2}, \tau_{k+1, 1})$ and $0 
\le k < n$ where either $\cC_t \in \cK$ is a connected subset of $\cC^{\ge h}_{x_k}
(B^c(v, \tfrac 1N))$ for some $x_k \in V$ or each component of $\cC_t$ is a connected 
subset of $\{x\} \cup \cC^{\ge h}_{x}$ for some $x \in K_0 = \partial B(v, \frac 1N)$. 
By Lemma~\ref{lem:lip_denergy}, we thus have that
\begin{equation*}
2|E|(\nu_{K_t} - \nu_{K_{\tau_{k, 2}}}) \le 
C(|\cC_t|_V + b_0(K_{\tau_{k, 2}})),
\end{equation*}
provided 
$|\cC_t|_V + |K_{\tau_{k, 2}}|_V + b_0(K_{\tau_{k, 2}})\le c |V|$. Since 
$b_0(K_{\tau_{k, 2}}) \le d$ by property~\ref{Kt2}, we can combine this display 
with \eqref{eq:conn_set_G} and \eqref{eq:F12_inclusion} to deduce
\begin{equation*}
\Cr{Flbc:1}(\delta', \beta, b) \cap \Cr{Flbc:2}(\delta', \delta, \beta, b) \subset \{ |\scC^{\ge h}_{\max}| \ge c\beta \delta |V|^{2/3} \} \mbox{ when } |V| \ge 
C ((1/\beta \delta)^{3/2} \vee b^3).
\end{equation*}
From this \eqref{eq:inclusion_lbd} follows 
by letting $\Cr{C:beta} = \frac1c \vee 1$.
\end{proof}
In our next result, we give lower bounds on the probabilities of these two events.
\begin{lemma}\label{lem:F1F2bnd}
Let $\delta \in (0, 1)$. There exist $C \in [\Cr{C:beta}, \infty)$, $c, c' \in (0, 
\infty)$ and $\delta' = \delta'(\delta, A) \in (0, 1)$ such that for any 
$|V| \ge (1 + A)^8 / \delta^3$ we have
\begin{equation}\label{eq:F1F2bnd}
\P[ (\Cr{Flbc:1}(\delta', C, c) )^c] \le \frac C{|V|} \mbox{ and } \P[ (\Cr{Flbc:2}(\delta', \delta, C, c))^c] \le C(1 + A)^{1/2}\delta^{1/5} + 1_{\delta \ge c'(1 + A)^{-2}}.
\end{equation}
\end{lemma}

\begin{proof}[Proof of Theorem~\ref{prop:lbd_cri_vol}, \eqref{eq:lbd_cri_vol}]
\eqref{eq:lbd_cri_vol} is an immediate consequence of Lemmas~\ref{lem:inclusion_lbd} and \ref{lem:F1F2bnd}.
\end{proof}
 
Proceeding to the proof of Lemma~\ref{lem:F1F2bnd}, let us start with the probability 
bound for the event $\Cr{Flbc:2}(\ldots)$.
\begin{lemma}\label{lem:brown_prob_cri}
Let $\delta \in (0, 1)$, $\beta, b > 0$ and $\delta' = (1 + 4\beta A)^{-2} \delta^{4/5}$. Then there exist 
$C = C(b, \beta)$ and $c = c(b, \beta) \in (0, \infty)$ such that,
\begin{equation}\label{eq:brown_prob_crit}
\P[(\Cr{Flbc:2}(\delta', \delta, \beta, b))^c] \le C(1 + A)^{1/2}\delta^{1/5} + 1_{\delta \ge c(1 + A)^{-2}},
\end{equation}
whenever 
$|V| \ge (1 / \delta)^{3/2}$.
\end{lemma}
For the other part of \eqref{eq:F1F2bnd}, we have the following result.
\begin{lemma}\label{lem:F1_prob_cri}
There exist $C, c \in (0, \infty)$ such that for any $\delta' \in (0, 1)$ and $|V| \ge (1 / \delta')^{15/4}$, 
\begin{equation}\label{eq:brown_prob_crit}
\P[(\Cr{Flbc:1}(\delta', C, c))^c] \le \frac C{|V|}.
\end{equation}
\end{lemma}
\begin{proof}[Proof of Lemma~\ref{lem:F1F2bnd}]
Lemma~\ref{lem:F1F2bnd} follows directly from Lemmas~\ref{lem:brown_prob_cri} and 
\ref{lem:F1_prob_cri}.
\end{proof}
We now proceed to proving the last two lemmas starting with 
Lemma~\ref{lem:brown_prob_cri}.
\begin{proof}[Proof of Lemma~\ref{lem:brown_prob_cri}]
Let us recall the stopping times $T_t$ from Theorem~\ref{thm:time-change} and the 
expression for the (continuous) quadratic variation $\langle\tilde M\rangle_s$ from 
Lemma~\ref{lem:Mt_mart}. Also recall that by Theorem~\ref{thm:time-change}, the 
process $\tilde M_s$ remains constant on each interval $[T_{t-}, T_t]$. In view of the 
definition \eqref{def:F2}, we can therefore write
\begin{equation*}
\Cr{Flbc:2}(\delta', \delta, \beta, b) = \left\{\begin{array}{c}\mbox{$\exists \, 0 \le 
r_1 \le r_2$ with $r_2 - r_1 \ge \beta \delta |V|^{2/3}$ s.t. $|K_{T_{r_2}}|_V \le b 
{\delta'}^{1/2}|V|^{2/3}$, $\langle \tilde M \rangle_{T_{r_2}} =$}\\
\mbox{$2|E|(\nu_{K_{T_{r_2}}} - \nu_{K_0}) = r_2$ and $\tilde M_{T_r} > ({\delta'}  + \beta A {\delta'}^{1/2}) |V|^{1/3}$ for all $r_1 \le r \le r_2$}
\end{array}
\right\}.
\end{equation*}
Further from Proposition~\ref{lem:bnd_lambda} we have that $2|E|\nu_{K_{T_{r_2}}} \ge 
c |K_{T_{r_2}}|_V$. Hence, letting $r_0 = 2|E|\nu_{K_0}$, we obtain the inclusion \begin{equation}\label{eq:tildeF2inclusion}
\Cr{Flbc:2}(\delta', \delta, \beta, b) \supset \left\{\begin{array}{c}\mbox{$\exists \, 0 \le r_1 \le r_2 \le c b {\delta'}^{1/2}|V|^{2/3} - r_0$ with $r_2 - r_1 \ge \beta \delta |V|^{2/3}$ s.t.}\\
\mbox{$\langle \tilde M\rangle_\infty \ge r_2$ and $\tilde M_{T_r} > ({\delta'}  + \beta A {\delta'}^{1/2}) |V|^{1/3}$ for all $r_1 \le r \le r_2$}
\end{array}
\right\}.
\end{equation}
Next let us note, by property~\ref{Kt3-1} of $(K_t)_{t \ge 0}$, that $|K_{\tau_{n, 2}}|_V = |V| - 1$ $\P$-a.s. and thus $$\langle \tilde M\rangle_\infty \ge \langle \tilde M\rangle_{\tau_{n, 2}} = 2|E|\nu_{K_{\tau_{n, 2}}} - r_0 \ge c |V| - r_0$$
in view of Proposition~\ref{lem:bnd_lambda}. Consequently, the lower bound on $\langle 
\tilde M\rangle_\infty$ on the right-hand side in \eqref{eq:tildeF2inclusion} can be 
dropped as soon as $|V| \ge C b^3$. Since the process $(\tilde M_{T_r})_{r \ge 0}$ is 
distributed as the stopped (standard) Brownian motion $(B_{r_0 + r \wedge \langle \tilde M 
\rangle_{\infty}})_{r \ge 0}$ by Theorem~\ref{thm:time-change}, the above observation 
leads to the following restatement of \eqref{eq:tildeF2inclusion} in terms of the 
corresponding probabilities when $|V| \ge Cb^3$. For 
\begin{equation*}
\Cr{Flbc:2}{_{, 0}}(\delta', \delta, \beta, b) \stackrel{{\rm def.}}{=} \left\{\begin{array}{c}\mbox{$\exists \, r_0 \le r_1 \le r_2 \le \Cl[c]{c:r2} b {\delta'}^{1/2}|V|^{2/3}$ with $r_2 - r_1 \ge \beta \delta |V|^{2/3}$ s.t.}\\
\mbox{$B_r > ({\delta'}  + \beta A {\delta'}^{1/2}) |V|^{1/3}$ for all $r_1 \le r \le r_2$}\end{array}\right\},
\end{equation*}
where $\Cr{c:r2} \in (0, \infty)$, we have
\begin{equation}\label{eq:F2F20}
\PB_0[(\Cr{Flbc:2}{_{, 0}}(\delta', \delta, \beta, b))^c] \ge \P[ (\Cr{Flbc:2}(\delta', \delta, \beta, b))^c], 
\end{equation}
where $\PB_0$ is the law of the standard Brownian motion $(B_r)_{r \ge 0}$. So it 
suffices to prove the bound \eqref{eq:brown_prob_crit} with $\Cr{Flbc:2}{_{, 0}}(\delta', \delta, \beta, b)$ instead of $\Cr{Flbc:2}(\delta', \delta, \beta, b)$. 
To this end, 
\begin{equation}\label{eq:constants}
    \text{let } \delta'_1 = 
    \frac{\Cr{c:r2}}{2} b {\delta'}^{1/2} \text{ and } \delta_1 = \frac{2(1+ 4\beta A)}{\Cr{c:r2} b} \delta'_1 \text{ with }\delta' = \frac{1}{(1+ 4\beta A)^2}\delta^{4/5}
\end{equation}
(recall the statement of Lemma~\ref{lem:brown_prob_cri}). Now consider the stopping 
time 
\begin{equation}\label{def:taudelta1}
        \tau_{\delta_1} \stackrel{\rm def.}{=} \inf\{r \ge 0: B_r \ge \delta_1 |V|^{1/3}\},
\end{equation}
and the event 
\begin{equation}\label{def:Ftilde}
        \Cr{Flbc:2}{_{, 1}}(\delta', \delta, \beta, b) \stackrel{\rm def.}{=} \Big\{ r_0 \le \tau_{\delta_1} \le \delta'_1 |V|^{2/3}, \inf_{\tau_{\delta_1} \le r \le \tau_{\delta_1} + 
        \beta \delta |V|^{2/3} } B_r > (\delta' + 
        \beta A {\delta'}^{1/2}) |V|^{1/3} \Big\}.
\end{equation}
Clearly, for $\delta'_1 + \beta \delta \le \Cr{c:r2} b {\delta'}^{1/2}$ which holds 
when $\delta \le c(b, \beta) (1 + A)^{-2}$ in view of our choices in \eqref{eq:constants}, we 
have the inclusion
\begin{equation}\label{eq:F21F20}
\Cr{Flbc:2}{_{, 1}}(\delta', \delta, \beta, b) \subset \Cr{Flbc:2}{_{, 0}}(\delta', \delta, \beta, b).
\end{equation}
Shifting the focus to $\Cr{Flbc:2}{_{, 1}}(\delta', \delta, \beta, b)$, note that (see 
\eqref{def:Ftilde}),
\begin{equation}\label{eq:brow_prob_cri1}
        \begin{split}
            &(\Cr{Flbc:2}{_{, 1}}(\delta', \delta, \beta, b))^c \\
            \subset& \{ \tau_{\delta_1} < r_0 \} \cup \big\{ \tau_{\delta_1} > \delta'_1 |V|^{2/3}\big\} \cup \big\{\inf_{\tau_{\delta_1} \le r \le \tau_{\delta_1} + \beta \delta |V|^{2/3} } B_r \le (\delta' + \beta A {\delta'}^{1/2}) |V|^{1/3}\big\}.
        \end{split}
    \end{equation}
Since $r_0 = 2|E|\nu_{K_0} \le C$ (recall Lemma~\ref{lem:characterisation_lambda} and also 
that $K_0 = \partial B(v, \frac 1N)$ with $N \ge 2$) and the maximum of a standard 
Brownian motion up to a time $t$ is distributed as the absolute value of a Gaussian 
variable with variance $t$, we get the following bounds:
\begin{equation*}
\PB_0[\tau_{\delta_1} < r_0] \le e^{-c \delta_1^2|V|^{\frac23}} \stackrel{\eqref{eq:constants}}{\le} e^{-c \delta^{\frac45}|V|^{\frac23}}  \mbox{ and } \PB_0[ \tau_{\delta_1} > \delta'_1 |V|^{\frac23} ] \le \tfrac{\delta_1}{\sqrt{\delta_1'}} \stackrel{\eqref{eq:constants}}{\le} C(b, \beta) (1 + A)^{\frac12}\delta^{\frac15}.
\end{equation*}
Using similar estimates along with the strong Markov property of Brownian motion, 
we obtain
\begin{equation*}
\PB_0 \big[ \inf_{\tau_{\delta_1} \le r \le \tau_{\delta_1} + \beta \delta |V|^{\frac23} } B_r \le (\delta' + \beta A {\delta'}^{\frac12}) |V|^{\frac13} \big] \le \exp( - \tfrac{ (\delta_1 - \delta' -  \beta A {\delta'}^{\frac12})^2}{2 \beta \delta}) \stackrel{\eqref{eq:constants}}{\le} \exp(-c(\beta) (1 / \delta)^{\frac15})
\end{equation*}
for any $\delta \in (0, 1)$. 
Putting together these estimates we get from the inclusion \eqref{eq:brow_prob_cri1} that
\begin{equation*}
\PB_0[(\Cr{Flbc:2}{_{, 1}}(\delta', \delta, \beta, b))^c] \le C(b, \beta) (1 + A)^{1/2} \delta^{1/5}
\end{equation*}
for all $\delta \in (0, 1)$ and $|V| \ge \tfrac{1}{\delta^{3/2}}$. This 
yields the bound \eqref{eq:brown_prob_crit} in view of \eqref{eq:F21F20} and 
\eqref{eq:F2F20}.
\end{proof}


Proof of Lemma~\ref{lem:F1_prob_cri} goes through a number of intermediate events. To 
this end, for any $\delta' \in (0,1)$ and $\beta > 0$, consider the event
 \begin{equation}\label{def:Edelta}
         \Cl[tilF]{Flbc:blk}(\delta', b) 
         \stackrel{\rm def.}{=} \big\{\tilde M^{\rm blk}_{t} \le \delta' |V|^{1/3} \text{ for any $t \ge 0$ such that } |K_t|_{V} \le  b\,{\delta'}^{1/2} |V|^{2/3}\big\},
 \end{equation}
where $\tilde M^{\rm blk}_{t} = M^{\rm blk}_{K_t}$ (see \eqref{eq:M_K_decomp}).

\begin{lemma}\label{lem:bulk_cri}
There exists $\Cl[c]{c:beta} > 0$ such that for any $\delta' \in (0, 1)$,
    \begin{equation}\label{eq:bulk_cri}
    \P\big[(\Cr{Flbc:blk}(\delta', \Cr{c:beta}))^c\big] \le C e^{-c {\delta'}^{1/2}|V|^{2/3}}.
    \end{equation}
\end{lemma}

\begin{proof}
Note that $K_t \cup \{v\}$ is connected by property~\ref{Kt2} of the family $(K_t)_{t \ge 0}$. Consequently, the set $(K_t \cup \{v\}) \cap V$ is connected {\em in the} graph $\cG$. Therefore, it suffices to show in view of the definition of $M^{\rm blk}_K$ from \eqref{eq:M_K_decomp} and the upper bound on $\nu_K$ from 
Proposition~\ref{lem:bnd_lambda} that
\begin{equation}\label{eq:bulk_cri2}
\P\Big[ \bigcup_{K \in \cK(c')} \big\{\textstyle{\sum_{x \in K \setminus \{v\}}} d_x \varphi(x) \ge c \delta' |V|^{1/3}\frac{|V|}{|K|}\big\}\Big] \le C e^{-c \delta'^{1/2}|V|^{2/3}}
\end{equation}
for some $c' \in (0, \infty)$ where 
\begin{equation*}
    \cK(c') \stackrel{\rm def.}{=} \{ K \subset V : v \in K, K \text{ is connected in $\cG$ and } |K| \le c' {\delta'}^{1/2} |V|^{2/3} \}.
\end{equation*}
Now, from Lemma~\ref{lem:var_ubd} we know that $\sum_{x \in K \setminus \{v\}} d_x\varphi(x)$ is a centered Gaussian variable with variance bounded by $C|K|$ for all 
$K \in \cK(c')$. Using Gaussian tail bounds, we thus obtain
\begin{equation*}
\P\big[\, \textstyle{\sum_{x \in K \setminus \{v\}}} d_x \varphi(x) \ge c \delta' |V|^{1/3}\frac{|V|}{|K|} \,\big] \le e^{-c {\delta'}^2 \frac{|V|^{2 + 2/3}}{|K|^3}}.
\end{equation*}
Combined with the standard upper bound $|\cK(c')| \le e^{C c' 
{\delta'}^{1/2}|V|^{2/3}}$ (recall that the maximum degree of any vertex in $\cG$ is 
$d$), this yields \eqref{eq:bulk_cri2} for a suitable choice of $c'$ via a union bound.
\end{proof}
Now for any $\delta' \in (0, 1)$ and $\beta > 0$, 
let
\begin{equation}\label{def:Fprime}
\Cl[tilF]{Flbc:k1}(\delta', \beta) 
\stackrel{\rm def.}{=} \left\{\begin{array}{c}\mbox{$\tilde M_{\tau_{k, 1}} \le (\delta' + \beta A {\delta'}^{1/2})|V|^{1/3} + \beta \sqrt{\log 
|V|}$ for any}\\ \mbox{$0 \le k \le n$ such that $|K_{\tau_{k, 1}}|_V \le \Cr{c:beta} {\delta'}^{1/2} |V|^{2/3}$} \end{array}
\right\}.
\end{equation}

\begin{lemma}\label{lem:smallMtau}
There exists $\Cl{C:bdr} > 0$ such that for any $\delta' \in (0, 1)$ and $|V| \ge (1 / \delta')^{3/4}$,
    \begin{equation}\label{eq:smallMtau}
        \P[\Cr{Flbc:blk}(\delta', \Cr{c:beta}) \setminus \Cr{Flbc:k1}(\delta', \Cr{C:bdr})] \le \frac{C}{|V|}.
    \end{equation}
\end{lemma}

\begin{proof}
In view of Lemma~\ref{lem:mgexp}, we can write
\begin{equation*}
\tilde M_{t} = \tilde M_{t}^{{\rm blk}} + \tilde M_{t}^{{\rm bdr}},
\end{equation*}
where $\tilde M_t^{{\rm blk}} = M_{K_t}^{{\rm blk}}$ and $\tilde 
M_t^{{\rm bdr}} = M_{K_t}^{{\rm bdr}}$ for all $t \ge 0$. By 
\eqref{def:Edelta} we know that, on the event $\Cr{Flbc:blk}(\delta')$,
    \begin{equation*}
        \tilde M^{\rm blk}_{t} \le 
        \delta' |V|^{1/3} \text{ whenever }|K_t|_V \le \Cr{c:beta} {\delta'}^{1/2}|V|^{2/3}.
    \end{equation*}
    As to the term $\tilde M_t^{{\rm bdr}}$, we have the following expression from \eqref{eq:M_K_decomp}.
    \begin{equation*}
        \tilde M^{\rm bdr}_t = \nu_{K_t} \sum_{x \in \partial K_t} \varphi(x) \big(\sum_{y \in V \setminus K_t} d_y  \bP_y[ \widetilde X_{H_{K_t}} = x]\big) = \nu_{K_t}\big(\textstyle{\sum_{\partial K_t \cap \widetilde \cG^{\le h}}} + \textstyle{\sum_{\partial K_t \setminus \widetilde \cG^{\le h}}}\big),
    \end{equation*}
    where the first and the second summations are over $x \in \partial K_t \cap \widetilde 
    \cG^{\le h}$ and $x \in \partial K_t \setminus \widetilde \cG^{\le h}$ respectively. 
    We can bound the first summation as follows (recall that $h = A |V|^{-1/3} \ge 0$):
    \begin{equation*}
    \begin{split}
     \mbox{$\sum_{\partial K_t \cap \widetilde\cG^{\le h}}$} &\le h \sum_{x \in \partial K_t \setminus \partial B(v, \frac 1 N)} \, \sum_{y \in V \setminus K_t} d_y  \bP_y[ \widetilde X_{H_{K_t}} = x] \le h \sum_{y \in V \setminus K_t} d_y  \sum_{x \in \partial K_t} \bP_y[ \widetilde X_{H_{K_t}} = x] \\
     &= h \sum_{y \in V \setminus K_t} d_y \le 2|E|\,  A |V|^{-1/3}.
     \end{split}
    \end{equation*}
    In order to bound the second term, we will use property~\ref{Kt3-2} of $(K_t)_{t \ge 
    0}$ from Proposition~\ref{prop:Kt} which tells us that $\partial K_t \setminus 
    \widetilde \cG^{\le h} \subset \partial B(v, \frac 1N)$ and 
    $\bP_y [ \widetilde X_{H_{K_t}} \in \partial B(v, \frac 1N) \setminus \widetilde \cG^{\le h}] = 0$ for all $y \in V \setminus \{v\}$ when $t = \tau_{k, 1}$ ($k \ge 1$). Hence in this case,
    \begin{equation*}
    \begin{split}
     \mbox{$\sum_{\partial K_t \setminus \widetilde \cG^{\le h}}$} &\le \frac1{d_x} \sum_{x \in \partial B(v, \frac 1N)}|\varphi(x)|.
     \end{split}
    \end{equation*}
    Combining all the previous displays we obtain the inclusion,
    \begin{equation}\label{eq:F20include}
    \Cr{Flbc:blk}(\delta', \Cr{c:beta}) \setminus \Cr{Flbc:k1}{_{, 0}}(\delta', \alpha) 
    \subset 
    \big\{\mbox{$\sum_{x \in \partial B(v, \frac 1N)}|\varphi(x)| \ge \alpha d$} \big\} \cup \{\tilde M_0 \ge \alpha\}
    \end{equation}
for any $\alpha \in (0, \infty)$ where 
$$\Cr{Flbc:k1}{_{, 0}}(\delta', \alpha) \stackrel{{\rm def.}}{=} \left\{\begin{array}{c}\mbox{$\tilde M_{\tau_{k, 1}} \le (\delta' + 2|E|\nu_{K_{\tau_{k, 1}}}A|V|^{-2/3})|V|^{1/3} + \alpha$ for any}\\ \mbox{$0 \le k \le n$ such that $|K_{\tau_{k, 1}}|_V \le \Cr{c:beta} {\delta'}^{1/2} |V|^{2/3}$}\end{array}\right\}$$
(cf.~\eqref{def:Fprime}). Now from Proposition~\ref{lem:bnd_lambda} we have $2|E|\nu_K \le C (|K|_V + b_0(K))$ for any $K \in 
\cK_{<\infty}$ such that $|K|_V + b_0(K) \le c|V|$. Since each component of $K_t$ intersects $K_0 = \partial B(v, \frac 1 N)$ 
by property~\ref{Kt2} of $(K_t)_{\ge 0}$, i.e., $b_0(K_t) \le d$, 
we get the upper bound 
\begin{equation}\label{eq:nuKtbnd}
2|E|\nu_{K_t} \le C\delta'^{1/2}|V|^{2/3} \,\, \forall \mbox{ $t \ge 0$ s.t. $|K_t|_V \le \Cr{c:beta} {\delta'}^{1/2} |V|^{2/3}$ when $|V| \ge C \vee (1 / \delta')^{3/4}$}
\end{equation}
(recall our standing assumptions on $\delta$ and $|V|$). 
Plugging \eqref{eq:nuKtbnd} into the definition of $\Cr{Flbc:k1}{_{, 0}}(\delta', 
\alpha)$, we obtain from \eqref{eq:F20include} that
\begin{equation*}
    \begin{split}
    &\Cr{Flbc:k1}(\delta') \setminus \big\{ \tilde M_{\tau_{k, 1}} \le (\delta' + CA\delta'^{1/2})|V|^{1/3} + \alpha \text{ for any } 0 \le k \le n \text{ s.t. } |K_{\tau_{k, 1}}|_V \le \Cr{c:beta} {\delta'}^{1/2} |V|^{2/3} \big\}\\
    \subset & \,\, \big\{\mbox{$\sum_{x \in \partial B(v, \frac 1N)}|\varphi(x)| \ge \alpha d$} \big\} \cup \{\tilde M_0 \ge \alpha\}.
    \end{split}
    \end{equation*}
From this we can now conclude the proof of \eqref{eq:smallMtau} by setting $\alpha = 
C \sqrt{\log |V|}$ and using Gaussian tail bounds for $\varphi(x); x \in \partial 
B(v, \frac 1N)$ and $\tilde M_0$, each of which is a centered Gaussian variable with 
variance bounded by some $C > 0$ (follows from \eqref{eq:mgexp}, \eqref{eq:varMK} and \eqref{eq:bnd_lambda}).
\end{proof}

Next for any $\delta' \in (0, 1)$ and $\beta > 0$, let us introduce the event
\begin{equation}\label{def:Fstar}
\begin{split}
    &\Cl[tilF]{Flbc:k12}(\delta', \beta) = \Cr{Flbc:k12}(\delta', \beta, N)\\ 
    \stackrel{\rm def.}{=}& \big\{ \tilde M_t \le \tilde M_{\tau_{k, 1}} + \beta \sqrt{\log |V|} \text{ for } \tau_{k,1} \le t \le \tau_{k,2}, 0 \le k \le n \text{ s.t. } |K_{\tau_{k, 1}}|_V \le \Cr{c:beta} {\delta'}^{\frac12} |V|^{\frac23}\big\}.
    \end{split}
\end{equation}

\begin{lemma}\label{lem:smallM}
There exists $\Cl{C:bm_log} > 0$ such that for any $\delta' \in (0, 1)$,
\begin{equation}\label{eq:smallM}
\P[(\Cr{Flbc:k12}(\delta', \Cr{C:bm_log}))^c] \le \frac{C}{|V|}.
\end{equation}
\end{lemma}

\begin{proof}
  Recall the definition of $T_t$ from Theorem~\ref{thm:time-change} and since $\tilde M_s$ is constant on each interval $[T_{t-}, T_t]$ by the same result, we have the inclusion
  \begin{equation}\label{eq:smallM1}
      \begin{split}
          & \Cr{Flbc:k12}(\delta', \beta) \\  
          \supset& \, 
          \big\{\tilde M_{T_r} - \tilde M_{T_{q_{k, 1}}} \le \beta \sqrt{\log |V|} 
          \,\, \forall \,\, q_{k, 1} \le r  \le q_{k, 2}, 0 \le k \le n \text{ s.t. } |K_{\tau_{k, 1}}|_V \le \Cr{c:beta} {\delta'}^{\frac12} |V|^{\frac23}\big\}.
      \end{split}
  \end{equation}
where $q_{k, i} = \langle \tilde M \rangle_{\tau_{k,i}}$. From \eqref{eq:qua_var_M_t} in Lemma~\ref{lem:Mt_mart} we know that $q_{k, i} = 2|E|(\nu_{K_{\tau_{k, i}}} - 
\nu_{K_0})$ ($i = 1, 2$). Also from property~\ref{Kt3-2} of 
Proposition~\ref{prop:Kt} we know that $K_{\tau_{k,2}} = K_{\tau_{k,1}} \cup I$ for some interval $I$. Thus in view of Lemma~\ref{lem:lip_denergy}, there exists $C > 0$ such that $q_{k, 2} - q_{k, 1} \le C b_0(K_{\tau_{k, 1}})$ provided 
$|K_{\tau_{k, 1}}|_V + b_0(K_{\tau_{k, 1}}) + C \le c|V|$. Since $b_0(K_{\tau_{k, 
1}}) \le d$ by property~\ref{Kt2} of Proposition~\ref{prop:Kt}, this gives us
\begin{equation}\label{eq:smallM2}
    q_{k, 2} - q_{k, 1} \le C \mbox{ for $|V| \ge C$.}
\end{equation}
 Further, we know from Theorem~\ref{thm:time-change} that the process $(\tilde 
 M_{T_r})_{r \ge 0}$ has the same law as the stopped Brownian motion $(B_{r_0 + r 
 \wedge \langle \tilde M \rangle_{\infty}})_{r \ge 0}$ with $r_0 = 2|E|\nu_{K_0}$. 
 Combining this with \eqref{eq:smallM2}, \eqref{eq:smallM1} and \eqref{eq:nuKtbnd} 
 we obtain, for $|V| \ge C$,
\begin{equation*}
    \P[(\Cr{Flbc:k12}(\delta', \beta)^c] \le \PB_0\big[\sup_{0 \le r \le |V|,\, 0 \le s \le 1} (B_{r + s} - B_r) > \beta \sqrt{\log |V|}\big].
\end{equation*}
Covering the interval $[0, |V|]$ with intervals of length $2$ having both endpoints 
in $\Z$ and using (Gaussian) tail estimates for the maximum and minimum of a 
Brownian motion, we get via a union bound,
\begin{equation*}
\P[(\Cr{Flbc:k12}(\delta', C)^c] = \PB_0\big[\sup_{0 \le r \le |V|,\, 0 \le s \le 1} (B_{r + s} - B_r) > C \sqrt{\log |V|}\,\big] \le \frac C{|V|}
\end{equation*}
for some $C > 0$.
\end{proof}
We are now ready to prove Lemma~\ref{lem:F1_prob_cri}.
\begin{proof}[Proof of Lemma~\ref{lem:F1_prob_cri}]
It follows from the definitions of the events $\Cr{Flbc:blk}(\delta', \Cr{c:beta})$, 
$\Cr{Flbc:k1}(\delta', \Cr{C:bdr})$ and $\Cr{Flbc:k12}(\delta', \Cr{C:bm_log})$ in 
\eqref{def:Edelta}, \eqref{def:Fprime} and \eqref{def:Fstar} respectively that
\begin{equation*}
\Cr{Flbc:blk}\big(\tfrac{\delta'}2, \tfrac{\Cr{c:beta}}{\sqrt{2}}\big) \cap \Cr{Flbc:k1}\big(\tfrac{\delta'}2, \Cr{C:bdr}\big) \cap \Cr{Flbc:k12}(\delta', \Cr{C:bm_log}) \subset \Cr{Flbc:1}\big(\delta', \Cr{C:bdr}, \tfrac{\Cr{c:beta}}{\sqrt{2}}\big)
\end{equation*}
whenever $|V| \ge C$. \eqref{eq:brown_prob_crit} now follows from this inclusion via a 
union bound with the corresponding probability bounds provided by 
Lemmas~\ref{lem:bulk_cri}, \ref{lem:smallMtau} and \ref{lem:smallM} respectively.
\end{proof}

\bigskip

\noindent \textbf{Acknowledgments.} SG's research was partially supported by a grant 
from the Department of Atomic Energy, Government of India, under project 
12R\&DTFR5.010500. We thank Pierre-Fran\c{c}ois Rodriguez, Guillaume Conchon-Kerjan, 
Soumendu Sundar Mukherjee and Debapratim Banerjee for insightful discussions during 
the early stages of this project. We are also grateful to Remco van der Hofstad and 
Rajat Subhra Hazra for their valuable comments. SG acknowledges the hospitality of the 
Statistics and Mathematics Unit, Indian Statistical Institute (ISI), Kolkata on 
several occasions. DP thanks the School of Mathematics, Tata Institute of 
Fundamental Research (TIFR), Mumbai for their generous hospitality.

\end{document}